\newif\ifArXiV

\ArXiVtrue	%ArXiV style
\ifArXiV
\documentclass{article}
\else
\documentclass[authoryear,preprint]{elsarticle}
\fi

\usepackage{fullpage}

\usepackage{lmodern}
\usepackage[english]{babel}

\usepackage{amsmath,amssymb,amsthm,graphicx}
\usepackage{comment}
\usepackage{enumitem}
\usepackage{todonotes}
\usepackage{multirow}
\usepackage{mathtools}
\usepackage{tikz}
\usepackage{hyperref}
\usepackage{subcaption}
\usepackage{booktabs}
\usepackage{xspace}
\usepackage{lscape}
\usepackage{multicol}
\usepackage{rotating} %for sidewaytable
\usepackage[normalem]{ulem} %for sout
\usepackage{placeins} % for \FloatBarrier = \clearpage without \newpage

\newtheorem{theorem}{Theorem}
\newtheorem{definition}[theorem]{Definition}

\newtheorem{lemma}[theorem]{Lemma}
\newtheorem{corollary}[theorem]{Corollary}

\newtheorem{observation}[theorem]{Observation}

\theoremstyle{remark}

\usepackage[linesnumbered, noend, ruled]{algorithm2e}
\let\oldnl\nl% Store \nl in \oldnl
\newcommand{\nonl}{\renewcommand{\nl}{\let\nl\oldnl}}% Remove line number 
\DeclareMathOperator{\st}{s.t.}
\newcommand{\obj}{z}
\newcommand{\loc}{J}
\newcommand{\cus}{I}

\newcommand{\LBGS}{LB^{\#}}

\newcommand{\PPCaLP}{\mathcal P(PC1)}
\newcommand{\PPCbLP}{\mathcal P(PC2)}

\newcommand{\mytag}[1]{(\hypertarget{#1}{\mathrm{#1}})}
\newcommand{\myref}[1]{\textnormal{(\hyperlink{#1}{#1})}}

\newcommand{\PCP}{pCP\xspace}
\newcommand{\BC}{B\&C\xspace}
\newcommand{\blue}[1]{{\color{blue}#1}}

\ifArXiV
\usepackage[affil-it]{authblk}
\usepackage[authoryear]{natbib}
\newenvironment{frontmatter}{}{}
\newenvironment{keyword}{\small \textbf{Keywords:}}{}
\let\address\affil

\renewcommand{\blue}[1]{{#1}}
\fi

\begin{document}
		\begin{frontmatter}
			\title{%Teaching an old formulation some new tricks: 
				A scaleable projection-based branch-and-cut algorithm for the 
				\texorpdfstring{$p$}{p}-center problem}
			%
			% If the paper title is too long for the running head, you can set
			% an abbreviated paper title here
			%
			
			\ifArXiV
			\author[1]{Elisabeth Gaar\thanks{elisabeth.gaar@jku.at}}
			\author[1,2]{Markus Sinnl\thanks{markus.sinnl@jku.at}}
			\affil[1]{Institute of Production and Logistics Management, 
			Johannes Kepler University Linz, Linz, Austria}
			\affil[2]{JKU Business School, Johannes Kepler University 
			Linz, Linz, Austria}			
			\date{}
			\maketitle
			
			\else
			
			\author[jku]{Elisabeth Gaar}
			\ead{elisabeth.gaar@jku.at}
			\author[jku,jkubus]{Markus Sinnl}
			\ead{markus.sinnl@jku.at}
			\address[jku]{Institute of Production and Logistics Management, 
			Johannes Kepler University Linz, Linz, Austria}
			\address[jkubus]{JKU Business School, Johannes Kepler University 
				Linz, Linz, Austria}	
					
			\fi
%\maketitle

\begin{abstract}
The $p$-center problem (\PCP) is a fundamental problem in location science, where 
we are given customer demand points and possible facility locations, and 
we want 
to choose $p$ of these locations to open 
a facility such that the maximum distance of any customer demand point to its 
closest open
facility is minimized. 
State-of-the-art 
solution 
approaches of \PCP use its connection to the set cover problem to solve 
\PCP in an iterative fashion by repeatedly solving set cover problems. The classical textbook integer programming (IP) 
formulation of \PCP is usually dismissed due to its size and bad linear 
programming (LP)-relaxation bounds.  

We present a novel solution approach that works on a new IP 
formulation that can be obtained by a projection from the classical 
formulation. The formulation is solved by means of branch-and-cut, where cuts 
for demand points are iteratively generated. Moreover, the formulation can be strengthened with 
combinatorial information to obtain a much tighter LP-relaxation. In 
particular, 
we present a novel way to use lower bound information to obtain stronger cuts. 
We show that the LP-relaxation bound of our strengthened formulation has 
the same strength as the best known bound in literature, which is based on a 
semi-relaxation.

Finally, we also present a 
computational study on instances \blue{from the literature} with up to more than 700000 
customers and locations. Our solution algorithm is competitive with highly 
sophisticated set-cover-based solution algorithms, which depend on various 
components and parameters.

% for other researchers.
%computational comparison of highly sophisticated solution frameworks \blue{from literature} with our 
%straightforward implementable idea \blue{}
 
\end{abstract}
\end{frontmatter}

\section{Introduction}

The \emph{(vertex) $p$-center problem} (\PCP) is a fundamental problem in 
location science \citep{laporte2019location,snyder2011fundamentals}, where
we are given customer demand points and potential facility locations, and 
we want to choose $p$ of these locations to open 
a facility such that the maximum distance of any customer demand point to its 
closest open
facility is minimized. The \PCP is \emph{NP-hard} \citep{kariv1979algorithmic} 
and has applications where it is critical that all customer demand points can 
be reached quickly, i.e., at least one open facility can be reached quickly by 
each customer demand point. This includes emergency service locations, such as 
ambulances or fire stations, and also relief actions in humanitarian crisis 
\citep{calik2013double,lu2013robust,jia2007modeling}. \blue{Next to applications in location science, the \PCP is also used for clustering of large-scale data \citep{geraci2009k,kleindessner2019fair,malkomes2015fast}, for feature selection \citep{meinl2011maximum} and in computer vision \citep{friedler2010approximation}.}
%In the context of a 
%global 
%pandemic such as the COVID-19 outbreak starting in 2019, the \PCP can offer an 
%useful 
%modeling approach for the location of test-centers on a strategic level, as 
%people may be more motivated to go to such a test-center if the travel 
%distance is short. 

A formal definition of the \PCP is as follows. Given an integer $p$, a 
set of customer 
demand points $\cus$ with cardinality $|\cus|=n$, a set of potential facility 
locations $\loc$ of cardinality $|\loc|=m \geq p$ and a distance $d_{ij}$ from 
a 
customer 
demand point $i$ to the potential facility location $j$ for every $i \in \cus$ 
and 
$j \in \loc$,
find a subset $S \subseteq \loc$ with cardinality $|S|=p$ of facilities to 
\emph{open} such that the maximum 
distance between a customer demand point and its closest 
%potential facility location in $S$ 
open facility is minimized, i.e., such that 
$\max_{i \in \cus} \min_{j \in S} \{d_{ij} \}$ is minimized.

%In the remainder of the paper 

There exists various integer programming (IP) formulations for the \PCP. In 
this paper we focus on the classical textbook formulation, see for 
example~\citet{daskin2013network}. It uses binary variables $y_j$ to indicate 
whether a facility opens at the potential facility location $j \in \loc$, 
binary variables $x_{ij}$ to indicate whether a customer demand point $i \in 
\cus$ 
is 
served by the the potential facility location $j \in \loc$
and a continuous variable~$z$ to measure the distance in the objective function.
\begin{subequations}
\begin{alignat}{3}
\mytag{PC1} \qquad
& \min & \obj \phantom{iiiii} \label{pc1:z}  \\ 
& \st~ &  \sum_{j \in \loc} y_j &= p \label{pc1:sumy} \\       
&& \sum_{j \in \loc} x_{ij} & =  1 && \forall i \in \cus \label{pc1:sumx} \\
&& x_{ij} &\leq y_j && \forall i \in \cus, \forall j \in \loc\label{pc1:xy}\\   
&& \sum_{j \in \loc} d_{ij} x_{ij} & \leq  \obj && \forall i \in \cus 
\label{pc1:sumdx}\\
&& x_{ij} &\in  \{0,1\} \qquad&& \forall i \in \cus, \forall j \in \loc  
\label{pc1:xbin}\\
&& y_{j} &\in  \{0,1\} && \forall j \in \loc  \label{pc1:ybin}\\
&& \obj & \in \mathbb{R}
\end{alignat}
\end{subequations}

In particular, \eqref{pc1:sumy} makes sure that exactly $p$ facilities are 
open, \eqref{pc1:sumx} guarantees that every customer is 
assigned to a 
facility, \eqref{pc1:xy} ensures that customers are only assigned to open 
facilities, \eqref{pc1:sumdx} forces $\obj$ to be at least the distance of any 
customer to its assigned facility, the objective function~\eqref{pc1:z} 
minimizes this $\obj$ and the constraints~\eqref{pc1:xbin} and~\eqref{pc1:ybin} 
make certain that $x$ and $y$ are binary. 
However, this formulation suffers from bad linear programming (LP)-relaxation 
bounds
\citep[see, e.g.,][]{snyder2011fundamentals} and also has scalability issues, as there are 
$O(|\cus|\cdot|\loc|)$ variables and constraints and large-scale instances 
\blue{from the literature} for the \PCP usually have $I=J=V$ with $|V|$ in the 
ten-thousands.

%Alternative popular approaches 
Thus, state-of-the-art exact approaches \citep[see, 
e.g.,][]{chen2009,contardo2019scalable}
for the \PCP use its connection to the \emph{set 
cover problem (SCP)}. 
Given a ground set $U$ and a set $W$ of 
subsets of~$U$, such that the union of the sets in $W$ is $U$, in the SCP we 
want to 
find a 
subset $T$ of $W$ of minimum cardinality, such that the union of the sets in 
$T$ is $U$. 
The question of whether the 
optimal objective function value of an instance of the \PCP is less than or 
equal 
to a given value $\bar z$ can be modeled by an instance of the SCP: we set 
$U=\cus$, define $W_j = \{i \in \cus: d_{ij} \leq \bar z\}$ for 
	every $j \in \loc$, and set $W = 
\cup_{j \in \loc} W_j$. Then the optimal objective function of \PCP is less 
than 
or 
equal to $\bar z$ if and only if the corresponding set cover instance has an 
optimal solution that uses at most $p$ sets. 
State-of-the-art exact 
approaches for the \PCP use this connection and solve the \PCP by iteratively 
solving SCPs for different values of $\bar z$. 
More 
details on this approach and also other existing IP formulations are given in 
Section \ref{sec:litreview}.

\subsection{Contribution and outline}

In this paper, we present a novel solution approach that works on an IP 
formulation that can be obtained by a projection from the classical 
formulation. The formulation is solved by means of branch-and-cut \blue{(\BC)}, 
where cuts 
for customer demand points are iteratively generated. This makes the method 
suitable for 
large scale instances, as the complete distance matrix does not need to be kept in memory.
Moreover, we show how our formulation can be strengthened with 
combinatorial information to obtain a much tighter LP-relaxation compared to 
the 
LP-relaxation of the classical formulation. This strengthening procedure uses a 
novel way to use lower bound information to obtain stronger cuts. We also show 
how our formulation is connected to the SCP and present a 
computational comparison with state-of-the-art solution algorithms on instances 
\blue{from the literature}. 

In the remainder of this section, we discuss previous and related work to the 
\PCP. Section \ref{sec:formulation} contains our new formulation, and Section \ref{sec:lifting}
the strengthening procedure for the inequalities included in our formulation 
and theoretical results on the connection to the SCP and existing lower bounds \blue{from the literature}. Section 
\ref{sec:implementation} describes implementation details of our \blue{\BC} 
based solution algorithm. In Section \ref{sec:results}, the computational study 
%including the case study 
is presented. Finally, Section \ref{sec:conclusion} 
concludes the paper.

\subsection{Literature review}
\label{sec:litreview}

The \PCP was first mentioned in 1965 by \citet{hakimi1965} and it was 
proven to be NP-hard by \citet{kariv1979algorithmic} 
for $p \geq 2$, even in the case that $I = J$ is the set of vertices of a 
planar graph with maximum degree three and all distances are equal to one. 
However, 
there are special cases of \PCP that can be solved in polynomial time, for 
example it can be solved in $O(|\cus|^2 \log |\cus|)$ time if $\cus = \loc$ is 
the set of vertices of a tree \citep{kariv1979algorithmic}.

Since the introduction of the \PCP there has been an tremendous amount on work 
on 
both heuristic and exact solution methods.
As our work is on the design of an exact solution algorithm, we focus our 
literature review on existing exact methods and refer to the recent 
survey~\citep{garcia2019survey} for approximation algorithms and heuristics for 
the \PCP.

%As our goal is to design a novel exact solution approach, here we focus 
%entirely on these and refer to the recent survey of 
%Garcia-Diaz, Menchaca-Mendez, Menchaca-Mendez,  
%Hern{\'a}ndez, P{\'e}rez-Sansalvador and Lakouari \cite{garcia2019survey} for 
%approximation algorithms and heuristics for the \PCP.

The first exact solution method for the \PCP was proposed by~\citet{minieka1970} and used the above described connection with the 
SCP. Minieka suggested to 
start with an arbitrary solution of the \PCP and then iteratively solve SCPs in 
order to find 
out whether there is a better solution. 
Let $D=\{d_{ij}: i \in 
\cus, j \in \loc\}$  denote the set of all possible distances 
and let $d_1$, $\ldots$, $d_{K}$ be the values contained in $D$, so 
$D=\{d_1, \ldots, d_{K}\}$. 
Clearly, the 
optimal 
objective function value of \PCP is in $D$ and there 
are at most $|\cus|\cdot|\loc|$ potential optimal values.
Minieka's approach 
can be interpreted as searching for the optimal value in $D$ by first removing 
some 
values from $D$ and then going through 
all remaining values of $D$ one by one starting with the largest one and 
checking 
their optimality.  
Later \cite{garfinkel1977} elaborated this approach 
and proposed to use a 
heuristic to obtain a starting solution in order to further reduce potential 
optimal values in $D$ 
and to then use a binary search with the remaining values of $D$ instead of a 
linear search. 

In the early 2000s a lot of work was done on the \PCP.
\citet{ilhan2001} introduced a two-phase algorithm based on the 
idea of Minieka, where \blue{they}
solve feasibility linear programs (LPs) in order to obtain a good lower bound 
in the 
first phase, and iteratively solve feasibility SCPs within  
performing a linear search on $D$ starting from the lower bound in the second 
phase. 
Later \citet{alKhedhairi2005} proposed some enhancements 
to this approach in order to reduce the number of iterations of the second 
phase.
\citet{ilhan2002} perform a binary search to 
detect a suitable lower bound by solving LPs, and then 
perform a linear 
search on $D$.
\citet{caruso2003} use the connection to the SCP 
by providing two heuristics and two exact alogorithms for the \PCP based on 
weak 
and strong dominance relationships.

In 2004, \cite{elloumi2004} introduced a new IP 
formulation for the \PCP. This formulation has a binary variable $y_j$ for $j 
\in \loc$ indicating whether facility $j$ opens, analogously 
to~\myref{PC1}. 
Furthermore there is a binary variable for each %potential optimal 
 value in $D$ that indicates whether the optimal value of 
\PCP is less or equal than this value. 
Towards this end let  $u_k = 0$ if all customers have an open facility with 
distance at most $d_{k-1}$, otherwise $u_k = 1$ for all 
$k \in \{2, \dots, K\}$. 
\begin{subequations}
\begin{alignat}{3}
\mytag{PCE} \qquad
& \min & \blue{d_1} + \sum_{k=2}^{K} (d_k &- d_{k-1})u_k  \label{pcE:ojb}  \\ 
& \st~ &  \sum_{j \in \loc} y_j &\leq p \label{pcE:sumyp} \\       
&&\sum_{j \in \loc} y_j &\geq 1 \label{pcE:sumy1} \\  
&& u_k + \sum_{j\colon d_{ij} < d_k} y_{j} &\geq  1 && \forall i \in \cus, 
\forall k \in \{2, \dots, K\} \label{pcE:sumyu} \\
&& u_{k} &\in  \{0,1\} \qquad&& \forall  k \in \{2, \dots, K\}\label{pcE:ubin}\\
&& y_{j} &\in  \{0,1\} && \forall j \in \loc  
%\label{pcE:ybin}
%\nonumber
\end{alignat}
\end{subequations}
\cite{elloumi2004} show that the lower bound 
obtained 
by the LP-relaxation of this IP formulation \myref{PCE} is tighter than the 
one from the 
LP-relaxation of the classical IP formulation~\myref{PC1} and how the 
advantages 
of their 
formulation can help in the computation of an optimal solution. 

Yet another IP formulation was introduced by
% {\c{C}}al{\i}k and Tansel 
\citet{calik2013double}.
% in 2013. 
They also have the same binary variables $y$ as 
in~\myref{PC1}, and additionally to that there is a binary variable for each  
value in~$D$ that indicates whether the optimal value of 
\PCP is exactly this value. Thus, in contrast to the IP formulation 
from~\cite{elloumi2004} here exactly one of this new binary variables is equal 
to one. \citet{calik2013double} also investigated the connection of their IP 
formulation to the one of~\cite{elloumi2004} and provided the best  
LP-relaxation for the \PCP so far. They also utilize their IP formulation 
within 
an exact solver for the \PCP based on successive restrictions of this new 
formulation.

\cite{chen2009} introduced an LP-relaxation-based iterative SCP-based
algorithm, in which they only consider some of the customer demand points in 
each iteration. So, instead of having few iterations with large sub-problems 
like 
%all approaches at that time
previous approaches, they considered many iterations with small 
sub-problems. 
Recently, \cite{contardo2019scalable} have 
successfully enhanced this idea of considering only a subset of costumer demand 
points and included it into an exact SCP-based binary search algorithm for the 
\PCP, that allows to solve large scale instances.

Finally, several different variants of the 
\PCP have also been investigated, among them the capacitated \PCP, where each 
potential facilitiy location has a maximum capacity, the continuous \PCP, where 
the potential facility locations are not restricted to a set, and several 
versions with uncertain parameters. We refer to the chapter on the \PCP 
\citep{calik2019p} 
in~\citet{laporte2019location} for more details.

%In recent survey Garcia-Diaz, Menchaca-Mendez, Menchaca-Mendez,  
%Hern{\'a}ndez, P{\'e}rez-Sansalvador and Lakouari \cite{garcia2019survey}: 
%among exact algorithms are
%\begin{itemize}
%\item Daskin \cite{daskin2000}
%\item Ilhan, Özsoy and Pinar \cite{ilhan2002}
%\item Elloumi, Labb\'{e} and Pochet \cite{elloumi2004}
%\item Al-Khedhairi and Salhi \cite{alKhedhairi2005}
%\item Chen and Chen \cite{chen2009}
%\item {\c{C}}al{\i}k and Tansel \cite{calik2013double}
%\item Contardo, Iori and Kramer \cite{contardo2019scalable}
%\end{itemize}
%
%Additional sources in Laporte, Nickel and 
%Saldanha~da~Gama\cite{laporte2019location} in the chapter on p-center problems 
%by {\c{C}}al{\i}k, Labb{\'e} and Yaman \cite{calik2019p}:
%\begin{itemize}
%	\item {\c{C}}al{\i}k \cite{calik2013PhD} (PhD thesis)
%	\item Hakimi \cite{hakimi1965}
%	\item Minieka \cite{minieka1970}
%	\item Garfinkel, Neebe and Rao \cite{garfinkel1977}
%	\item Textbook Daskin \cite{daskin2013network}
%	\item Ilhan and Pinar \cite{ilhan2001} 
%	\item Elloumi, Labb\'{e} and Pochet \cite{elloumi2004}
%	\item {\c{C}}al{\i}k and Tansel \cite{calik2013double}	
%\end{itemize}
%
%
%And also found
%\begin{itemize}
%	\item Caruso, Colorni and Aloi \cite{caruso2003}
%\end{itemize}

%%%%%%%%%%%%%%%%%%%%%%%%%%%%%%%%%%%%%%
\section{A new IP formulation for the \texorpdfstring{$p$}{p}-center problem} 
\label{sec:formulation}

\blue{
In this section, we start by presenting our new IP formulation for 
the 
$p$-center problem in Section~\ref{sec:newIPFormulation}. 
Next, we show that the feasible region of the LP-relaxation of our new 
formulation can be seen as a projection of the feasible region of the 
LP-relaxation of~\myref{PC1} in Section~\ref{sec:lprelaxation}.
Furthermore, we provide details about how our new formulation can be obtained 
from~\myref{PC1} as Benders decomposition in Section~\ref{sec:benders}.
}

\subsection{New IP formulation}\label{sec:newIPFormulation}
\blue{Our new IP formulation \myref{PC2} for 
the 
$p$-center problem} uses $y$ and 
$z$ variables with the same meaning as in \myref{PC1}, i.e., the binary 
variables $y_j$ indicate whether a facility opens at location $j \in \loc$ 
and the continuous variable~$z$ represents the distance in the objective 
function.
\begin{subequations}
\begin{alignat}{3}
\mytag{PC2}
 \qquad 
& \min & \obj    \\ 
& \st~ &  \sum_{j \in \loc} y_j &= p   \label{pc2:sumy} \\       
&&  \obj &\geq d_{ij} - 
\sum_{j' \colon   d_{ij'} < d_{ij}} (d_{ij} - d_{ij'})y_{j'} \qquad
&& \forall i \in \cus, \forall j \in \loc  \label{eq:cuts} \\ 
&& y_{j} &\in  \{0,1\} && \forall j \in \loc  \label{pc2:ybin} \\
&& \obj & \in \mathbb{R}
%\nonumber
\end{alignat}
\end{subequations}

Next, we  
give combinatorial  arguments for the correctness of \myref{PC2}. 
The constraint~\eqref{pc2:sumy} makes sure that exactly $p$ facilities 
open and 
\eqref{pc2:ybin} forces $y$ to be binary.
The constraints~\eqref{eq:cuts} ensure that $\obj$ is at least the distance to 
the 
nearest open facility for each customer $i \in \cus$ and work as follows: For 
each $j \in \loc$, they ensure that $\obj$ is at least $d_{ij}$ in case no 
closer
facility $j'$ to customer~$i$ (i.e., a facility $j'$ with $d_{ij'}<d_{ij}$) is 
open. Thus, $\obj$ 
has to take at least the distance to the nearest open facility to customer $i$. 
Moreover, if for $j \in \loc$ at least one closer facility is open, then let 
$j'$ be the closest facility to customer $i$ that is open. Then the right-hand 
side of~\eqref{eq:cuts} is at most $d_{ij}-(d_{ij}-d_{ij'})=d_{ij'}$ 
(maybe even more is \blue{subtracted} from $d_{ij}$), thus 
$\obj$ is forced to be larger than a value that is at most $d_{ij'}$, which is 
the distance of customer $i$ to its closest open facility. Thus, 
constraints~\eqref{eq:cuts} never lead to an 
overestimation of the distance of a customer to its nearest open facility.

We observe that formulation  \myref{PC2} has $O(|\loc|)$ variables and 
$O(|I|\cdot|\loc|)$ constraints. Thus, the number of variables is $O(|\cus|)$ 
times less than for formulation \myref{PC1}. However, similar to \myref{PC1}, 
the number of constraints can become prohibitive for solving large scale 
instances. We thus do not solve \myref{PC2} directly, but use a branch-and-cut 
approach, where a lifted version of constraints~\eqref{eq:cuts} are separated 
on-the-fly. Section~\ref{sec:lifting} describes the lifting of~\eqref{eq:cuts}  
and Section~\ref{sec:implementation} discusses the separation.

\subsection{Projection-based point of view}
\label{sec:lprelaxation}
Before exploiting \myref{PC2} computationally, we show the connection between 
\myref{PC1} and \myref{PC2}.
\blue{In particular, we show that 
the feasible region of the 
LP-relaxation of \myref{PC2} is a projection of the feasible region of 
the LP-relaxation of \myref{PC1}. Thus, the LP-relaxations 
of \myref{PC1} and \myref{PC2} give the same bound for the \PCP.
}

\blue{Towards this end,} keep in mind that we have already argued that both 
IPs model the  \PCP. 
% and use this connection to prove that 
%\myref{PC2} is indeed an IP formulation of the \PCP. 
We start by examining their 
LP-relaxations. 
Let $\PPCaLP$ be the feasible region of the LP-relaxation 
of \myref{PC1}, i.e., 
$$\PPCaLP=\{(x,y,z)\in \mathbb R^{|I|\cdot|J|+|J|+1}: 
\eqref{pc1:sumy},\eqref{pc1:sumx},\eqref{pc1:xy},\eqref{pc1:sumdx},
~x_{ij}\geq 0 \quad \forall i \in \cus ~ \forall j \in \loc, 
~0\leq y_j\leq 1 \quad \forall j \in \loc \} $$
and let $\PPCbLP$ be the feasible region of the LP-relaxation of \myref{PC2}, 
i.e., 
$$\PPCbLP=\{(y,z)\in \mathbb R^{|J|+1}: \eqref{pc2:sumy},\eqref{eq:cuts}, 
~0\leq y_j\leq 1 \quad \forall j \in \loc \}. $$

\begin{theorem}\label{th:projection}
The feasible region $\PPCbLP$ is the projection of $\PPCaLP$ to the space of 
$(y,z)$-variables, so $\PPCbLP=\{(y,z)\in \mathbb R^{|\loc|+1}: \exists x \in 
\mathbb R^{|\cus|\cdot|\loc|}: (x,y,z) \in \PPCaLP \}$.
\end{theorem}	

\begin{proof}
We start by showing that any point in the projection of $\PPCaLP$ is in 
$\PPCbLP$. Towards that end let $(x^\circ, y^\circ,\obj^\circ)$ be in 
$\PPCaLP$. Clearly $(y^\circ,\obj^\circ)$ fulfills~\eqref{pc2:sumy} due to 
\eqref{pc1:sumy} and 
$0 \leq y^\circ_j \leq 1$ for all $j \in \loc$ by construction, so we only have 
to show that it fulfills~\eqref{eq:cuts}.

Consider some $i \in \cus$ and $j \in \loc$. 
%If $\obj^\circ \geq d_{ij}$, then 
%clearly for this $i$ and $j$ constraint~\eqref{eq:cuts} is satisfied, as a 
%positive value 
%is subtracted from $d_{ij}$ in the right-hand side of~\eqref{eq:cuts}.
%If $\obj^\circ < d_{ij}$, then 
Because of~\eqref{pc1:xy} for the right-hand 
side of~\eqref{eq:cuts} clearly 
$$d_{ij} - \sum_{j' \colon   d_{ij'} < d_{ij}} (d_{ij} - d_{ij'})y^\circ_{j'}
\leq 
d_{ij} - \sum_{j' \colon   d_{ij'} < d_{ij}} (d_{ij} - d_{ij'})x^\circ_{ij'}
=
d_{ij} \left(1 -  \sum_{j' \colon   d_{ij'} < d_{ij}}x^\circ_{ij'} \right) + 
\sum_{j' 
\colon   
d_{ij'} < d_{ij}} d_{ij'}x^\circ_{ij'}$$
holds
and because of~\eqref{pc1:sumx} this equals
$$
%d_{ij}\left(1 -  \sum_{j' \colon   d_{ij'} < d_{ij}}x_{ij'} \right) - 
%\sum_{j' \colon   
%d_{ij'} < d_{ij}} d_{ij'}x_{ij'}
%=
d_{ij}\left(\sum_{j' \colon   d_{ij'} \geq d_{ij}}x^\circ_{ij'} \right) + 
\sum_{j' 
\colon   d_{ij'} 
< d_{ij}} d_{ij'}x^\circ_{ij'}
\leq
\sum_{j' \in \loc} d_{ij'}x^\circ_{ij'}$$
which is at most $\obj^\circ$ because of~\eqref{pc1:sumdx}. Thus, the 
constraints~\eqref{eq:cuts} are fulfilled for any $i$ and $j$ and therefore 
$(y^\circ,\obj^\circ)$ is in $\PPCbLP$.

What is left to show is that any point in $\PPCbLP$ is also in the projection 
of $\PPCaLP$. 
Let $(y^*,z^*) \in \mathbb R^{|\loc|+1}$ be in $\PPCbLP$. We start by 
constructing $x^* \in \mathbb R^{|\cus|\cdot|\loc|}$ and then show that 
$(x^*, y^*, z^*) \in \PPCaLP$.

For every $i \in \cus$, let 
$j_i \in \arg \min_{j\in \loc} 
\{d_{ij}:
\sum_{j' \colon d_{ij'} < d_{ij}}y^*_{j'} < 1
\text { and }
\sum_{j' \colon d_{ij'} \leq d_{ij}}y^*_{j'} \geq 1
\},$
i.e., $j_i$ is an index such that 
all $y^*$ strictly closer than $d_{ij_i}$ to customer $i$ have weight less than 
1, and 
all $y^*$ with distance at most $d_{ij_i}$ to customer $i$ have weight at least 
1. In other words $y^*$ reaches weight $1$ at 
distance~$d_{ij_i}$ for customer $i$. Then for every $j \in \loc$ we set
$$
x_{ij}^*= \left\{
\begin{array}{ll}
y^*_j 
&  \text{ if } d_{ij} < d_{ij_i}, \\
y^*_j\left(1 - \sum_{j' \colon d_{ij'} < d_{ij_i}}y^*_{j'} \right)
/ \left( \sum_{j' \colon d_{ij'} = d_{ij_i}}y^*_{j'} \right) 
& \text{ if } d_{ij} = d_{ij_i}, \\
0 &  \text{ if } d_{ij} > d_{ij_i}.\\
\end{array} 
\right.
$$

Next we prove that $(x^*, y^*, z^*)$ is in $\PPCaLP$. 
Clearly $0\leq y^*_j\leq 1$ for all $j \in \loc$ holds 
and~\eqref{pc1:sumy} is satisfied because of~\eqref{pc2:sumy}.  
Furthermore $x^*_{ij}\geq 0$ 
for all $i \in \cus$ and $j \in \loc$
holds by 
construction. Moreover  
$$
\sum_{j \in \loc} x^*_{ij}
= 
\sum_{j \colon d_{ij} < d_{ij_i}}y^*_{j}
+
\left(1 - \sum_{j' \colon d_{ij'} < d_{ij_i}}y^*_{j'} \right)
/ \left( \sum_{j' \colon d_{ij'} = d_{ij_i}}y^*_{j'} \right)
\sum_{j \colon d_{ij} = d_{ij_i}}y^*_{j}
=1, 
$$
so~\eqref{pc1:sumx} holds.
In addition to that, due to the fact that 
$\sum_{j' \colon d_{ij'} \leq d_{ij_i}}y^*_{j'} \geq 1$ holds, we have that 
$x^*_{ij} \leq y^*_{j}$, hence~\eqref{pc1:xy} is fulfilled.
Eventually we consider the last constraint~\eqref{pc1:sumdx}. For a fixed $i$, 
due to~\eqref{pc1:sumx} we have 
$$
\sum_{j \in \loc} d_{ij} x^*_{ij}
=
d_{ij_i} \left( 1 - \sum_{j \in \loc} x^*_{ij} \right)
+
\sum_{j \in \loc} d_{ij} x^*_{ij}
=
d_{ij_i} - \sum_{j \in \loc}(d_{ij_i} - d_{ij})x^*_{ij}. 
$$
As a consequence of $x^*_{ij} = 0$ for $d_{ij} > d_{ij_i}$ and 
$x^*_{ij} = y^*_{j}$ for $d_{ij} < d_{ij_i}$, 
this is the same as
$$
d_{ij_i} - \sum_{j \colon d_{ij} < d_{ij_i}} (d_{ij_i} - d_{ij'})x^*_{ij'}
=
d_{ij_i} - \sum_{j \colon d_{ij} < d_{ij_i}} (d_{ij_i} - d_{ij'})y^*_{j'}, 
$$
which is at most $z^*$ due to~\eqref{eq:cuts} for $i$ and $j_i$. Hence we have 
shown that $\sum_{j \in \loc} d_{ij} x^*_{ij} \leq z^*$, so~\eqref{pc1:sumdx} 
holds and therefore $(x^*, y^*, z^*)$ is in $\PPCaLP$.
\end{proof}

Thus, Theorem~\ref{th:projection} confirms that the feasible region of the 
LP-relaxation of \myref{PC2} is in fact a projection of the feasible region of 
the LP-relaxation of \myref{PC1}. 
The following corollary is a simple consequence.

\begin{corollary} \label{cor:LPsAreTheSame}
	The optimal objective function values of the LP-relaxations of 
	\myref{PC1} and 
	\myref{PC2} coincide.
	Furthermore any optimal $(y^*,z^*)$ of the LP-relaxation of \myref{PC2} is 
	optimal for the LP-relaxation of \myref{PC1} 
	 and vice versa.
\end{corollary}

As a consequence of Corollary~\ref{cor:LPsAreTheSame}, the LP-relaxations 
of \myref{PC1} and \myref{PC2} give the same bound for the \PCP.

\subsection{Reference to Benders decomposition \label{sec:benders}}

We now present an alternative approach on how \myref{PC2} can be obtained 
from \myref{PC1} as Benders decomposition. Towards this end let 
$\mytag{PC1-Rx}$ be formulation 
\myref{PC1} with relaxed  $x$  
variables, 
i.e., \myref{PC1-Rx} is \myref{PC1} without~\eqref{pc1:xbin} and with 
the constraint  
$x_{ij}\geq 0$ for all $i \in \cus$ and for all $j \in \loc$.

\begin{observation}\label{obs:relaxation}
	For any feasible solution $(x,y,z)$ of \myref{PC1-Rx}, a binary 
	$\tilde 
	x$ can be constructed such that 
	$(\tilde x,y,z)$ is a feasible solution of 
	\myref{PC1}. Note that both solutions have the same 
	objective function value $z$.
\end{observation}
\begin{proof}
Let $(x,y,z)$ be a feasible solution of \myref{PC1-Rx}. We construct 
$\tilde x$ in the following way. 
%For any~$i$, for which $x_{ij}$ is binary for all $j \in \loc$, 
%we set $\tilde x_{ij} = x_{ij}$ for all $j \in \loc$.
For any~$i \in \cus$, 
%for which there is some $j \in \loc$ such that 
%$x_{ij}$ is not binary, 
we choose $j_i\in \arg\min_{j \in 
	\loc} \{ d_{ij}: y_j=1 \}$ and define
$\tilde x_{ij_i} = 1$ and $\tilde x_{ij} = 0$ for all 
$j \in \loc\setminus \{j_i\}$. Clearly this $\tilde x$ is binary. 

What is left to show is that $(\tilde x,y,z)$ is feasible for 
\myref{PC1}. \blue{Clearly}~\eqref{pc1:sumy} and~\eqref{pc1:ybin} hold because 
$(x,y,z)$ is feasible for \myref{PC1-Rx}.
Moreover it is easy to see that~\eqref{pc1:sumx},~\eqref{pc1:xy} 
and~\eqref{pc1:xbin} 
are fulfilled by construction. What is left so show is 
that~\eqref{pc1:sumdx} is satisfied. 
%Towards this end 
%we consider every 
%customer $i$ separately depending on its construction of $\tilde{x}$. 
%If $i$ is in such a way that $x_{ij}$ is binary for all $j \in \loc$, 
%then $\tilde x_{ij} = x_{ij}$ for all $j \in \loc$ and therefore 
%$\sum_{j \in \loc} d_{ij} \tilde x_{ij} 
%= 
%\sum_{j \in \loc} d_{ij} x_{ij}$ holds in this case.
%If $i$ is in such a way that $x_{ij}$ there is some $j \in \loc$ such 
%that $x_{ij}$ is not binary, 
%then
%we
%consider the $j'$ from the definition 
%of $\tilde x_{ij}$. 
For any~$i \in \cus$  we have by construction
$\sum_{j \in \loc} d_{ij} \tilde x_{ij} 
= d_{ij_i} \tilde x_{ij'}
=  d_{ij_i}$
and, because $(x,y,z)$ is feasible for \myref{PC1-Rx}, 
$d_{ij_i}$  
equals
$ d_{ij_i}\sum_{j \in \loc}  x_{ij}
=
\sum_{j \in \loc} d_{ij_i} x_{ij}$. 
Due to the fact that 
$x_{ij} = 0$ whenever $d_{ij} < d_{ij_i}$ 
(because in this case $y_j = 0$), this is less or equal to
$\sum_{j \in \loc} d_{ij} x_{ij}$. 
This is less or equal to $z$, because 
$(x,y,z)$ is feasible for 
\myref{PC1-Rx}. Thus~\eqref{pc1:sumdx} is fulfilled and 	
$(\tilde x,y,z)$ is feasible for \myref{PC1}.
\end{proof}	

As a consequence of Observation~\ref{obs:relaxation}, the \PCP can be solved by 
using formulation \myref{PC1-Rx} instead of \myref{PC1}. Our new formulation 
\myref{PC2} can be seen as being 
obtained by applying Benders decomposition to \myref{PC1-Rx} to project out the 
$x$-variables. Similar to the Benders decomposition reformulation for the 
uncapacitated facility location problem (UFL) presented in 
\citet{fischetti2017redesigning}, the obtained formulation i) is compact, i.e., 
it has a polynomial number of variables and constraints, and ii) has a 
combinatorial interpretation, that we have already presented. 
The main difference between the Benders decomposition for the UFL and our new 
formulation for the \PCP is that in the UFL, there is a variable $z_i \geq 0$ 
for every customer  
$i \in \cus$ to measure the cost for each customer $i$, and the variables 
are then summed in the objective function, while in our case we have just a 
single $z$ as the \PCP has the objective function to minimize the maximum 
distance. We use this connection for an efficient separation of inequalities 
\eqref{eq:cuts} (resp., a lifted variant of them) in our branch-and-cut 
algorithm. Details of the separation are given in Section \ref{sec:sep}.
For more background on 
Benders decomposition we refer to e.g., Chapter 8 of the book 
\citep{conforti2014integer} and the surveys 
\citep{costa2005survey,rahmaniani2017benders}.

Note that 
Observation~\ref{obs:relaxation} and Theorem~\ref{th:projection} 
also directly imply 
the following corollary.

\begin{corollary}
	\label{cor:IPformulationsCoincide}
	The optimal objective function values of \myref{PC1} and 
	\myref{PC2} coincide.
	Furthermore any optimal~$y^*$ of \myref{PC1} is optimal for 
	\myref{PC2} and vice versa.
\end{corollary}
%\begin{proof}
%	\todo[inline]{This is the exhaustive proof. It can also be written as "This 
%	is an easy consequence of Corollary~\ref{cor:LPsAreTheSame} and 
%	Observation~\ref{obs:relaxation}."}

%	Let $(y^*,z^*)$ be an optimal solution of \myref{PC2}, then $(y^*,z^*)$ is 
%	in $\PPCbLP$ and due to Theorem~\ref{th:projection} there is a $x \in 
%	\mathbb{R}^{|\cus|\cdot |\loc|}$ such that $(x,y^*,z^*)$ is in $\PPCaLP$. 
%	Due to  
%	Observation~\ref{obs:relaxation} 
%	and the fact that $(y^*,z^*)$ is binary this implies that there is a binary 
%	$\tilde x$ such that $(\tilde x,y^*,z^*)$ is feasible for \myref{PC1}.

%	On the other hand, let $(x^\circ,y^\circ,z^\circ)$ be an optimal solution 
%	of \myref{PC1}. Then $(x^\circ,y^\circ,z^\circ) \in \PPCaLP$ holds and due 
%	to Theorem~\ref{th:projection} this implies that $(y^\circ,z^\circ)$ is in 
%	$\PPCbLP$ and because it is binary $(y^\circ,z^\circ)$ is a feasible 
%	solution of \myref{PC2}.

%	To summarize, any solution of \myref{PC1} can be transferred into a 
%	feasible solution of \myref{PC2} with the same objective function value and 
%	vice versa. Therefore, the optimal objective function values of \myref{PC1} 
%	and \myref{PC2} coincide.
%\end{proof}

Thus, as a byproduct of this different Benders decomposition viewpoint,  
Corollary~\ref{cor:IPformulationsCoincide} implies that \myref{PC2} is indeed 
an 
IP formulation for the \PCP and allows us to exploit \myref{PC2} 
computationally.

\section{Lifting the new formulation}
\label{sec:lifting}
In this section we first present a method to lift the LP-relaxation of our 
new IP formulation \myref{PC2}
\blue{in Section~\ref{sec:exploitLB}, which allows us to utilize a known lower 
bound and obtain a lower bound that is at least as strong as this lower bound 
by solving an LP}. 
Then we discuss 
\blue{how to compute the best possible lower bound obtainable this way and give 
a combinatorial interpretation of it in terms of the set cover problem in 
Section~\ref{sec:bestLB}.
In Section~\ref{sec:comparisonLiterature} we show that our best possible lower 
bound coincides with the currently best lower 
bound known from the literature.
Finally, we show that only some inequalities of the LP are necessary for 
obtaining the best possible lower bound in 
Section~\ref{sec:notAllInequalitiesNecessary}.}
Aspects relevant for computations are discussed in 
Section~\ref{sec:implementation}.

\blue{
\subsection{Exploiting lifted optimality cuts}
\label{sec:exploitLB}
}

The key ingredient to our lifting is the following lemma that involves the 
lifted optimality cuts \eqref{eq:loptimality}.

\begin{lemma}
	\label{lem:validCuts}
	Let $LB \geq 0 $ be a lower bound on the optimal objective function value 
	of 
	\myref{PC2}. 
	Then for any 
	$i \in \cus$ and $j \in \loc$ the inequality
	\begin{equation} 
\obj \geq \max\{LB,d_{ij}\} -\sum_{j' \colon d_{ij'} < d_{ij}} 
\left(\max\{LB,d_{ij}\} - \max\{LB,d_{ij'}\}\right)y_{j'}	
\label{eq:loptimality}\tag{L-OPT}
	\end{equation}
	is valid for \myref{PC2}, i.e., every feasible solution of \myref{PC2} 
	fulfills 
	\eqref{eq:loptimality}.
\end{lemma}
\begin{proof}
	Let $i \in \cus$, $j \in \loc$ and let $(y,\obj)$ be a feasible solution of 
	\myref{PC2}. 
	Then $y$ is a feasible solution to the considered instance of 
	\PCP with objective function value $\obj$. In particular, every 
	customer 
	demand point $i \in \cus$ has at most distance $\obj$ to at least one of 
	the 
	locations 
	indicated in $y$. As $LB$ is a lower bound on the optimal objective 
	function value of \myref{PC2} we have $LB \leq \obj$. 
	As a consequence,  $(y,\obj)$ is 
	also a feasible solution of the slightly modified instance of \PCP
    where we replace 
	every distance $d_{ij}$ with $\max\{LB,d_{ij}\}$. 
	Clearly~\eqref{eq:loptimality} is nothing else than~\eqref{eq:cuts} for 
	this 
	slightly modified instance, 
	hence $(y,\obj)$ 
	fulfills~\eqref{eq:loptimality}
	because it is feasible for this slightly modified instance. 
	Therefore~\eqref{eq:loptimality} is valid.
\end{proof}

Now we can use the valid inequalities~\eqref{eq:loptimality} to find a stronger 
LP-relaxation of \myref{PC2} by replacing~\eqref{eq:cuts} 
by~\eqref{eq:loptimality} in the following way.

\begin{theorem}
	\label{thm:newcutsgiveLB}
	Let $LB\geq 0$ be a lower bound on the optimal objective function value of 
	\myref{PC2}. Then 
\begin{subequations}	
\begin{alignat}{3}
\mytag{PCLB} \quad
\mathcal{L}(LB) = & \min & \obj  \label{eq:PCLBobj}  \\ 
& \st~ &  \sum_{j \in \loc} y_j &= p  \label{eq:PCLBsumy}\\       
&&  \obj &\geq d_{ij} -\sum_{j' \colon d_{ij'} < d_{ij}} 
\left(d_{ij} - \max\{LB,d_{ij'}\}\right)y_{j'} \qquad
&& \forall i \in \cus, \forall j \in \loc : d_{ij} > LB   
\label{eq:PCLBcutsLB}\\ 
&& \obj &\geq LB \label{eq:PCLBobjLB}\\ 
&& 0 &\leq y_{j} \leq 1  && \forall j \in \loc \label{eq:PCLBboundsy} \\
&& \obj & \in \mathbb{R}
\label{eq:PCLBz}
%\nonumber
\end{alignat}
\end{subequations}
	is an LP-relaxation of \myref{PC2}.
	In particular, $\mathcal{L}(LB)$
	is a lower bound on the optimal objective function value of \myref{PC2}
	with $\mathcal{L}(LB) \geq LB$.
\end{theorem}

\begin{proof}
	Obviously an LP-relaxation of \myref{PC2} is obtained by relaxing the 
	binary 
	constraints on $y$. Due to Lemma~\ref{lem:validCuts} this 
	LP-relaxation can be strengthen  by adding cuts of the 
	form~\eqref{eq:loptimality}. 
	Let us consider some $i \in \cus$ and $j \in \loc$. If $d_{ij} \leq LB$ 
	then 
	\eqref{eq:loptimality} simplifies to~\eqref{eq:PCLBobjLB}, 
	and if $d_{ij} > LB$ then~\eqref{eq:loptimality} is equivalent to
	\eqref{eq:PCLBcutsLB}. Furthermore we can disregard the 
	inequalities~\eqref{eq:cuts} because they are superseded 
	by~\eqref{eq:loptimality}. 
\end{proof}

In a nutshell, Theorem~\ref{thm:newcutsgiveLB} allows us to start with a 
lower bound $LB$ on \myref{PC2} and solve the LP \myref{PCLB} with $|\loc|+1$ 
variables and a 
maximum of $|\cus|\cdot|\loc|+\blue{2}|\loc|+2$ constraints, in order to obtain 
a new 
lower bound 
$\mathcal{L}(LB)$ that is as least as good as $LB$.
Our aim is to utilize Theorem~\ref{thm:newcutsgiveLB} in the following way: We 
start with a lower bound $LB$ on the optimal objective 
function value of 
\myref{PC2}, for example the minimum of all distance 
	values of $D$, so $LB = \min \{d: d\in D\}$, and then iteratively improve 
	this lower bound by solving 
\myref{PCLB} in order to obtain the new lower bound $\mathcal{L}(LB)$ and use 
this bound as $LB$ in the next iteration. 
We give a formal definition of this procedure 
\blue{at the end of Section~\ref{sec:bestLB}.}

\blue{
\subsection{Our best possible lower bound}\label{sec:bestLB}
}

\blue{The aim of this section is to give a combinatorial interpretation of the 
best possible lower bound that we can reach with this procedure, i.e., to 
consider its convergence. Furthermore, we detail how to compute this best 
possible lower bound.
To this end}
%In order to consider the convergence of this procedure 
we start with the following lemmata.

\begin{lemma}
	\label{lem:NoViolCutsImplySum}
	Let $LB\geq 0$ be a lower bound on the optimal objective function value of 
	\myref{PC2} such that $\mathcal{L}(LB) = LB$ and let 
	$(y^\ast,\obj^\ast=LB)$ 
	be an 
	optimal solution of \myref{PCLB}. Then
	\begin{align*}
	\sum_{j \colon d_{ij} \leq LB}y^\ast_{j} \geq 1
	\end{align*}
	holds for all $i \in \cus$. 
\end{lemma}
\begin{proof}
	For any $i \in \cus$  let $j_i \in \arg \min_{j\in \loc} \{d_{ij}: 
	d_{ij} > 
	LB 
	\}$, i.e., $j_i$ is the closest facility location to the 
	customer demand point $i$ that has 
	distance 
	greater than $LB$. Then~\eqref{eq:PCLBcutsLB} for $i$ and $j=j_i$ yields 
	that
	\begin{align*}
	LB &\geq d_{i j_i} -\sum_{j' \colon 
	d_{ij'} < d_{ij_i}} 
	\left(d_{i j_i} - \max\{LB,d_{ij'}\}\right)y^\ast_{j'}
	\end{align*}
	holds. By the choice of $j_i$ we have $\max\{LB,d_{ij'}\} = LB$ for 
	all $j'$, so this implies
	\begin{align*}
	\left(d_{i j_i} - LB\right)
	\sum_{j' \colon 
			d_{ij'} < d_{ij_i} 
		} y^\ast_{j'} 
	&\geq  d_{i j_i} - LB  
	\end{align*}
	and together with the fact that 
	$d_{ij'} < d_{ij_i}$ 
is equivalent to $d_{ij'} \leq LB $ for all $j'$   
 this yields the desired result
	\begin{align*}
\sum_{j' \colon 
		d_{ij'} 
		\leq LB} y^\ast_{j'} 
&\geq  1.
\end{align*}	
\end{proof}

\begin{lemma}
	\label{lem:sumImpliesNoViolCuts}
	Let $LB\geq 0$ be a lower bound on the optimal objective function value of 
	\myref{PC2}, $y$ such that $0 \leq y \leq 1$ and $\obj = LB$.
	 If
	\begin{align}
	\label{eq:SumyGeq1}
	\sum_{j' \colon d_{ij'} \leq LB}y_{j'} \geq 1
	\end{align}
	holds for some $i \in \cus$, then the 
	inequalities~\eqref{eq:PCLBcutsLB} hold 
	for $(y,\obj)$  
	for this $i$  and all $j \in \loc$.
\end{lemma}
\begin{proof}
	Let $i \in \cus$ such that~\eqref{eq:SumyGeq1}	
	holds. Consider an arbitrary but fixed $j \in \loc$ with $d_{ij} > LB$.
	Then clearly
	\begin{align*}
	d_{ij} -&\sum_{j' \colon d_{ij'} <  d_{ij}} 
	\left(d_{ij} - \max\{LB,d_{ij'}\}\right)y_{j'} 
	 \leq 
	 d_{ij} -\sum_{j' \colon  d_{ij'} \leq LB} 
	\left(d_{ij} - \max\{LB,d_{ij'}\}\right)y_{j'} \\
	& = 
	 d_{ij} -\left(d_{ij} - LB\right) \sum_{j' \colon 
	 	d_{ij'} \leq 
	 LB} 
	y_{j'} 
	\leq d_{ij} -\left(d_{ij} - LB\right) = LB = \obj 
	\end{align*}
	holds, and hence~\eqref{eq:PCLBcutsLB} is satisfied for $(y,\obj)$.
\end{proof}

Lemma~\ref{lem:NoViolCutsImplySum} and Lemma~\ref{lem:sumImpliesNoViolCuts} are 
the key ingredients of the following theorem.

\begin{theorem}
	\label{thm:convergence}
		Let $LB\geq 0$ be a lower bound on the optimal objective function value 
		of 
	\myref{PC2}. Then $\mathcal{L}(LB) = LB$ holds if and only if 
	there is a fractional set cover solution with radius 
	$LB$ that uses at most $p$ sets, i.e., if and 
	only if there 
	is a feasible solution $y^\ast$ for 
	\begin{subequations}
	\begin{alignat}{3}
	    &\min &  \sum_{j \in \loc} y_j \phantom{iiiii}&\label{eq:sco} \\       
		&\st~&\sum_{j \colon  d_{ij}  \leq LB}y_{j} &\geq 1 &&
		\forall i \in \cus \label{eq:sc1} \\
		&& 0 &\leq y_{j} \leq 1 \qquad && \forall j \in \loc \label{eq:sc2}
	\end{alignat}
	\end{subequations}
	with objective function value at most $p$.
\end{theorem}
\begin{proof}
	Assume $\mathcal{L}(LB) = LB$ holds, then by 
	Lemma~\ref{lem:NoViolCutsImplySum} an optimal solution 
	$(y^\ast,\obj^\ast=LB)$ of~\myref{PCLB} fulfills~\eqref{eq:sc1}, because 
	of~\eqref{eq:PCLBboundsy} it 
	satisfies~\eqref{eq:sc2} and due to~\eqref{eq:PCLBsumy} the optimal 
	objective function value~\eqref{eq:sco} is at most~$p$.
	
	Now let $\obj^\ast=LB$ and assume $y^\ast$ is feasible 
	for~\eqref{eq:sc1} 
	and~\eqref{eq:sc2} and has objective function value~\eqref{eq:sco} at most 
	$p$. Then due 
	to Lemma~\ref{lem:sumImpliesNoViolCuts} all 
	inequalities~\eqref{eq:PCLBcutsLB} are fulfilled for $(y^\ast,\obj^\ast)$. 
	It is easy to see that we can 
	obtain $(y^{\ast\ast},\obj^{\ast\ast}=\obj^\ast)$ by 
	arbitrarily choosing some $y^\ast_j$ in $y^\ast$ and increase them 
	such 
	that~\eqref{eq:PCLBsumy} and~\eqref{eq:PCLBboundsy} hold and all 
	other constraints of~\myref{PCLB} are still satisfied. Therefore 
	$\mathcal{L}(LB) 
	\leq \obj^{\ast\ast} = LB$ holds because 
	$(y^{\ast\ast},\obj^{\ast\ast})$ is feasible 
	for~\myref{PCLB}. As a result $\mathcal{L}(LB) = LB$, because  
	$\mathcal{L}(LB) \geq LB$ holds by Theorem~\ref{thm:newcutsgiveLB}.
\end{proof}

Theorem~\ref{thm:convergence} implies that if there is a fractional set cover 
solution with radius 
$LB$ that uses at most $p$ sets, then the new bound we obtain by solving 
\myref{PCLB} will not be better than $LB$. Also the following corollary is an 
easy 
consequence of Theorem~\ref{thm:convergence}.

\begin{corollary}
	\label{cor:noImprovementOfLowerBound}
	Let $LB\geq 0$ be a lower bound on the optimal objective function value 
	of 
	\myref{PC2} with $\mathcal{L}(LB) = LB$. Then $\mathcal{L}(LB') = LB'$ 
	holds for all $LB' \geq LB$.
\end{corollary}
\begin{proof}
	By Theorem~\ref{thm:convergence} there is a fractional set cover solution 
	with radius $LB$ that uses at most $p$ sets for \myref{PC2}. Then for any 
	$LB' \geq LB$ it follows that there is also a fractional set cover solution 
	with radius $LB'$ that uses at most $p$ sets for \myref{PC2}, so by 
	Theorem~\ref{thm:convergence} it holds that $\mathcal{L}(LB') = LB'$.
\end{proof}

Corollary~\ref{cor:noImprovementOfLowerBound} implies that whenever a lower 
bound $LB$ does not yield a better bound by solving \myref{PCLB}, this is also 
true for all values larger than $LB$. This gives rise to the following 
definition.

\begin{definition}
	\label{def:LBGS}
	Let $\LBGS = \min \{LB \in \mathbb{R}: \mathcal{L}(LB) = LB\}$.
\end{definition}

Due to Corollary~\ref{cor:noImprovementOfLowerBound} $\LBGS$ is not only the 
smallest lower bound, for which solving \myref{PCLB} does not improve the lower 
bound, but $\LBGS$  is also the supremum over all lower bounds, for which 
solving \myref{PCLB} improves the bound. Next we reduce the possible values for 
$\LBGS$. 

\begin{lemma}
	\label{lem:LBGSinD}
	It holds that $\LBGS \in D$.
\end{lemma}
\begin{proof}
It is easy to see that for any $k \in \{1, \dots, K-1\}$ the 
constraint~\eqref{eq:sc1} is the same for all values of $LB$ with 
$d_{k} \leq LB < d_{k+1}$, because whether a variable $y_j$ is present in the 
sum only depends on whether its distance is less or equal to $LB$, and $d_k$ 
and $d_{k+1}$ are two consecutive distances of all distances. 
As a consequence, the optimal objective function value of \myref{PCLB} is the 
same for all $LB$ with $d_{k} \leq LB < d_{k+1}$. This together with 
Theorem~\ref{thm:convergence} implies that whether $\mathcal{L}(LB) = LB$ holds 
is the same for all $LB$ with $d_{k} \leq LB < d_{k+1}$. As $\LBGS$ is defined 
as the 
minimum value of $LB$ such that $\mathcal{L}(LB) = LB$, clearly $\LBGS$ 
coincides with $d_k$ for some $k \in \{1, \dots, K-1\}$ or with $d_{K}$, 
therefore $\LBGS \in D$ holds. 
\end{proof}

In order to compute $\LBGS$ we can use the following approach: 
%Keep in mind 
%that $D = \{d_1, \dots, d_K\}$. 
We 
start with a lower bound $LB$ on the optimal objective 
function value of 
\myref{PC2}, for example the minimum of all distance 
	values in $D$, so $LB = d_1$.
Then we iteratively improve 
this lower bound by 
first solving 
\myref{PCLB} in order to obtain the new lower bound $\mathcal{L}(LB)$, 
and in a second step we set 
$LB = \min \{d_k \in D: d_k \geq \mathcal{L}(LB)\}$.  
As soon as we do not improve $LB$ in an iteration anymore, we set $\LBGS = LB$. 

It is straight forward to see that the output $\LBGS$ of this procedure  
gives $\LBGS$ as defined in Definition~\ref{def:LBGS}. 
Note that an 
infinite series of increases without convergence is not possible since in the 
procedure $LB$ only takes values of $D$ and there are at most 
$|\cus| \cdot 
|\loc|$ values in $D$.

%As a result of Theorem~\ref{thm:convergence} it follows that when iteratively 
%solving \myref{PCLB} and using $\mathcal{L}(LB)$ as new lower bound $LB$, 
%then  
%we will ultimately converge to a lower bound $\LBGS$ for which there 
%exists a 
%feasible 
%fractional set cover solution with radius 
%$LB$ that uses at most $p$ sets. Note that there are only finitely many values 
%that we have to consider as $LB$, because any $LB$ can always be increased to 
%the value of 
%the smallest $d_{ij}$ which is greater or equal to $LB$ because we 
%know that the optimal solution of $(PC)$ coincides with a $d_{ij}$. Hence an 
%infinite series of increases without convergence is not possible.

\blue{
\subsection{Comparison of our best possible lower bound with the best lower 
bound from the 
literature}\label{sec:comparisonLiterature}
}

Next we want to compare our best bound $\LBGS$ to the currently best lower 
bound known from the literature $LB^*$, which was introduced by  
\citep{elloumi2004}.  
They obtain $LB^*$ by relaxing the integrality of the variables $y_j$ 
while keeping the integrality of the variables $u_k$ of 
their IP formulation \myref{PCE}, so  $LB^*$ is the optimal objective function 
value of the following MIP. 
\begin{alignat}{3}
\mytag{PCE-R} \qquad
& LB^* = \min  &\eqref{pcE:ojb} \nonumber \\ 
& \st~ & \eqref{pcE:sumyp}&-\eqref{pcE:ubin} \nonumber \\       
&& 0 \leq &y_{j} \leq  1 \qquad && \forall j \in \loc \nonumber
\end{alignat}

It turns out that our best bound $\LBGS$ and the best known bound from the 
literature $LB^*$ coincide.

\begin{theorem}
	$\LBGS = LB^*$  
\end{theorem}
\begin{proof}
	First, we show that $\LBGS \geq LB^*$ holds. Clearly $\mathcal{L}(\LBGS) = 
	\LBGS$ holds, so by Theorem~\ref{thm:convergence}
	there is a fractional set cover solution $y^*$ with radius 
	$\LBGS$ that uses at most $p$ sets. Furthermore by Lemma~\ref{lem:LBGSinD} 
	we have that  
	$\LBGS \in 
	D$, so there is a $k^*$ such that $\LBGS = d_{k^*}$. 
	We set $u^*_{k} = 1$ for all $k \leq k^*$ and 
	$u^*_{k} = 0$ for all $ k > k^*$. 
	It is easy to see that $(u^*, y^*)$ is a feasible solution of 
	\myref{PCE-R}, whose objective function value equals $d_{k^*} = \LBGS$, 
	thus $\LBGS \geq LB^*$ holds.
	
	Next, we show that $LB^* \geq \LBGS$ holds. Towards this end let 
	$(u^\circ, y^\circ)$ be an optimal solution of \myref{PCE-R}.
	\citep{elloumi2004} already observed that due to the structure of the 
	problem there is an index $k^\circ \in \{2, \ldots, K\}$ such that 
	$u^\circ_{k} = 1$ for all $k \leq k^\circ$ and 
	$u^\circ_{k} = 0$  for all $ k > k^\circ$. Then 
	constraint~\eqref{pcE:sumyu} for $k=k^\circ+1$ implies that 
	$\sum_{j\colon d_{ij} < d_{k^\circ+1}} y^\circ_{j} \geq  1$ holds $\forall 
	i \in 
	\cus$, and hence also 
	$\sum_{j\colon d_{ij} \leq d_{k^\circ}} y^\circ_{j} \geq  1$ holds 
	$\forall i \in \cus$. As a consequence, $y^\circ$ is a fractional set cover 
	solution with radius $d_{k^\circ}$ that uses at most $p$ sets. 
	Furthermore $d_{k^\circ}$ coincides with $LB^*$, because $LB^*$ 
	is the optimal objective function value of \myref{PCE-R} and the optimal 
	solution $(u^\circ, y^\circ)$ has objective function value 
	$d_{k^\circ}$. 
	Thus, $y^\circ$ is a fractional set cover 
	solution with radius $LB^*$ that uses at most $p$ sets, so by 
	Theorem~\ref{thm:convergence} it holds that
	$\mathcal{L}(LB^*) = LB^*$, therefore  $LB^* \geq \LBGS$.
\end{proof}

\blue{
\subsection{Inequalities needed for obtaining our best possible lower bound}
\label{sec:notAllInequalitiesNecessary}
}

Now we want to draw our attention to the question of whether all 
inequalities~\eqref{eq:PCLBcutsLB} are necessary for the convergence result. In 
fact, the following holds. 
\begin{theorem}
	\label{thm:newcutsgiveLBLessInequalities}
	Let $LB\geq 0$ be a lower bound on the optimal objective function value of 
	\myref{PC2}. For any $i \in \cus$ let $j_i \in \arg \min_{j\in \loc} 
	\{d_{ij}: 
	d_{ij} > LB \}$, i.e., $j_i$ is the closest facility location to 
	the customer demand point~$i$ that has 
	distance more than $LB$. 
	Then
	\begin{alignat}{3}
	\mathcal{L'}(LB) = & \min & \eqref{eq:PCLBobj} %\obj 
	\nonumber \\ 
	& \st~ &  %\sum_{j \in \loc} y_j &= p  
	%\label{eq:PCLBsumyLessInequalities}       
	%&&  
	\obj &\geq d_{ij_i} -\sum_{j' \colon d_{ij'} <  d_{ij_i}} 
	\left(d_{ij_i} - \max\{LB,d_{ij'}\}\right)y_{j'} \qquad
	&& 	\forall i \in \cus 
	\label{eq:PCLBcutsLBLessInequalities}
	\\ 
	&&\eqref{eq:PCLBsumy} &, \eqref{eq:PCLBobjLB} - \eqref{eq:PCLBz}
	\nonumber%\\ 
	%&& \obj &\geq LB 
	%\label{eq:PCLBobjLBLessInequalities}
	%\nonumber\\ 
	%&& 0 &\leq y_{j} \leq 1  && \forall j \in \loc 
	%\label{eq:PCLBboundsyLessInequalities} 
	%\nonumber\\
	%&& \obj & \in \mathbb{R} 
	%\nonumber
	\end{alignat}
	is an LP-relaxation of \myref{PC2}.
	In particular, $\mathcal{L'}(LB)$ gives a lower bound on the optimal 
	objective function value of \myref{PC2}
	with $\mathcal{L}(LB)  \geq \mathcal{L'}(LB)  \geq LB$.
	Furthermore $\mathcal{L'}(LB) = LB$ holds if and only if 
	there is a fractional set cover solution with radius 
	$LB$ that uses at most $p$ sets.
\end{theorem}
\begin{proof}
	Clearly $\mathcal{L'}(LB)$ is obtained from $\mathcal{L}(LB)$ by removing 
	some inequalities, so $\mathcal{L}(LB)  \geq \mathcal{L'}(LB)$ holds. 
	Analogous arguments as the ones in the proof of 
	Theorem~\ref{thm:newcutsgiveLB} prove that $\mathcal{L'}(LB)$ is a lower 
	bound on the optimal  objective function value of \myref{PC2}. 
	The proof of the remaining statements can be done analogously to the proof 
	of Theorem~\ref{thm:convergence}, i.e., with the help of analogous results 
	as Lemma~\ref{lem:NoViolCutsImplySum} and 
	Lemma~\ref{lem:sumImpliesNoViolCuts}.
\end{proof}

In other words Theorem~\ref{thm:newcutsgiveLBLessInequalities} implies that 
including only one of the  inequalities~\eqref{eq:PCLBcutsLB} for every 
customer 
demand point~$i$, namely~\eqref{eq:PCLBcutsLBLessInequalities} for the closest 
facility location to $i$ 
that has distance more than $LB$, is enough to ensure convergence to a 
fractional set cover solution $\LBGS$. This reduces the number of needed 
\blue{constraints} 
in $\mathcal{L'}(LB)$ to $|\cus|+\blue{2}|\loc|+\blue{2}$ instead of 
potentially 
$|\cus|\cdot|\loc| + \blue{2}|\loc| + \blue{2}$ 
\blue{constraints}  
needed in $\mathcal{L}(LB)$. However, it can be beneficial to 
include~\eqref{eq:PCLBcutsLB} for a customer demand point $i$ for several 
locations 
$j$ in order 
to improve the speed of convergence.

After deriving theoretical properties of our best bound $\LBGS$ and showing 
that it is as good as the best bound \blue{from the literature}, we next exploit  $\LBGS$ 
computationally.

\section{Implementation details \label{sec:implementation}}

We now present our solution algorithm to solve the \PCP based on formulation  
\myref{PC2}, where we replace~\eqref{eq:cuts} with the lifted optimality cuts 
\eqref{eq:loptimality}.  
In a nutshell, in our branch-and-cut algorithm 
we initialize the model with only a subset of inequalities 
\eqref{eq:loptimality} (i.e., only for a subset of customers $i\in \cus$ and 
potential facility locations $j\in \loc$) for a given valid lower bound $LB$. 
We 
then iteratively 
separate 
inequalities \eqref{eq:loptimality} while also updating the current lower bound 
$LB$ 
used in defining \eqref{eq:loptimality} by exploiting the objective value of 
the current LP-relaxation. 
\blue{In the 
update of $LB$ we also exploit that in instances \blue{from the literature} the 
distances 
are integral.}

We have implemented different strategies for doing 
the overall separation scheme (e.g., which inequalities to add, \ldots), these 
strategies are detailed in Section \ref{sec:details}. Before we discuss these 
strategies, 
we describe how a violated inequality \eqref{eq:loptimality} 
can be separated efficiently in Section~\ref{sec:sep}. Finally, 
Section~\ref{sec:heur} describes a primal heuristic which is called during the 
\blue{\BC} and driven by the optimal solution of the current LP-relaxation.

Our solution algorithm was implemented in C++
and it is available online under 
\url{https://msinnl.github.io/pages/instancescodes.html}.
IBM ILOG 
CPLEX 12.10 with 
default settings was used as \blue{\BC} framework. 
We apply the described separation schemes in the 
\texttt{UserCutCallback} of CPLEX, which gets called for fractional optimal 
solutions 
of the LP-relaxations within the \blue{\BC} tree. In the 
\texttt{LazyConstraintCallback}, which gets called for optimal integer 
solutions, we 
simply add a single violated inequality \eqref{eq:loptimality} if any exits. 
This is done, as most of the integer solutions encountered %during the 
%algorithm 
are produced by internal CPLEX heuristics (and are not just integral 
LP-relaxations) and are very different to the current optimal solution of the 
LP-relaxations. As a consequence, adding multiple inequalities 
\eqref{eq:loptimality} for such solutions is usually not useful to help to 
improve the optimal objective value the LP-relaxation.

%Thus, aside from ensuring correctness of the solution algorithm 
%by forcing the correct solution value for the currently produced heuristic 
%integer solution, the inequalities obtained when separation for a heuristic 
%solution usually do not help with the current lower bound.
%\todo{I really do not understand this sentence. Please rewrite! :)}

\subsection{Separation of a violated inequality
	\texorpdfstring{\eqref{eq:loptimality}}{(L-OPT)}}
\label{sec:sep}

Let $(y^*,z^*)$ be a (fractional) optimal solution of the LP-relaxation at a 
\blue{\BC} node (the LP-relaxation of \myref{PC2}, where we 
replace~\eqref{eq:cuts} with the lifted optimality cuts 
\eqref{eq:loptimality} for a subset of all customers $i \in \cus$ and potential 
facility locations $j \in \loc$), $LB$ a given lower bound on the objective 
value 
and $i \in 
\cus$ a given customer. In order to obtain an efficient separation procedure, 
we can leverage that the cuts for a fixed customer are similar to the Benders 
cuts for the uncapacitated facility location problem as discussed in Section 
\ref{sec:benders}. Thus we can use an adapted separation procedure of the one  
presented 
in \cite{fischetti2017redesigning} to find the most violated inequality, 
if there is a violated inequality for the customer $i \in \cus$ and the lower 
bound $LB$. 

The procedure works as follows: Let $d'_{ij}=d_{ij}$ if $d_{ij}>LB$ 
and $d'_{ij}=LB$ otherwise. Sort the facilities $j \in \loc$ in ascending order 
according to $d'_{ij}$. \blue{Note that we only need to consider facilities 
with $y^*_j>0$ for this sorting as the other facilities do not contribute to 
the potential violation.} In the following, we assume that the locations are 
ordered in this way, i.e., we have $d'_{i1}\leq \ldots \leq d'_{i|\loc|}$.
Let the \emph{critical location} $j_i$ be the index such that 
$\sum_{j=1}^{j_i-1}y^*_j < 1 \leq \sum_{j=1}^{\blue{j_i}}y^*_j$. 
The maximum violation (i.e., the largest right-hand side value for the current 
$y^*$)
is then given by
\begin{equation*}
%- \obj^* +
 d_{ij_i} -\sum_{\substack{j\colon y^*_j = 1 \text{ and }  d_{ij} <  
		d_{ij_i}}} 
\left(d_{ij_i} - \max\{LB,d_{ij}\}\right)y^*_{j}
\end{equation*}
and the 
inequality 
\eqref{eq:loptimality} with maximum violation is
\begin{equation*}
\obj \geq d_{ij_i} -\sum_{j \colon d_{ij} <  d_{ij_i}} 
\left(d_{ij_i} - \max\{LB,d_{ij}\}\right)y_{j}.
\end{equation*}
For more details on why this procedure gives the largest right-hand side value  
we refer to \cite{fischetti2017redesigning}.

We note that this procedure allows us to calculate $j_i$ and hence the maximal 
violation for each customer~$i$ in an efficient way without the need of knowing 
all the distances $d_{ij}$. \blue{As mentioned above}, we only need to know the distances 
$d_{ij}$ for each $j \in \loc$ with $y^*_j>0$ in order to determine $j_i$, and 
the number of such $j$ is usually rather small compared to $|\loc|$. 
 This is important for large-scale instances, where 
the distance matrix cannot be stored in memory and the distances need to be 
computed on-the-fly when needed.
Note however, that for adding such an inequality to the model for a customer 
$i$, all the distances from $i$ to any $j$ need to be known.

\subsection{Details of the overall separation schemes}
\label{sec:details}

We have implemented two different separation schemes, which differ in the way 
the violated inequalities that are added at a separation round are selected. 
Note that due to the size of the encountered instances, always adding all the 
violated inequalities becomes prohibitive. Moreover, previously added violated 
inequalities \eqref{eq:loptimality} can also become redundant, because during 
the course 
of the algorithm, we iteratively obtain larger values for $LB$, and thus new 
inequalities \eqref{eq:loptimality} which are strictly stronger than previously 
added ones can be derived. Moreover, as detailed in Section~\ref{sec:sep}, the 
maximal violation 
of any cut of the form \eqref{eq:loptimality} can be calculated efficiently for 
each customer, 
while for adding an inequality, all the distances need to be known, which is 
computationally much more costly for large-scale instances. Hence, a carefully 
engineered separation scheme is needed in order to obtain good performance. 

\newcommand{\maxViolated}{\texttt{maxViolated}\xspace}
\newcommand{\fixedCustomer}{\texttt{fixedCustomer}\xspace}
\newcommand{\fixedCustomerOwnRoot}{\texttt{fixedCustomerOwnRoot}\xspace}

Our two schemes are denoted by \maxViolated\ and \fixedCustomer. In both 
schemes, we use the option \texttt{purgeable} of CPLEX when adding a violated 
inequality in the root-node of the \blue{\BC} tree. This setting allows 
CPLEX to automatically remove a previously violated inequality if CPLEX decides 
(by an internal mechanism) that this inequality is ``useless'' for the 
subsequent solution process. To add a violated inequality at any other 
node of the \blue{\BC} tree, we use the method \texttt{addLocal} of CPLEX, 
which adds the inequality not in a global 
fashion, but only for the current subtree of the \blue{\BC} tree. This 
allows us to use the lower bound of the subtree as $LB$ within the inequalities 
\eqref{eq:loptimality} instead of using the global lower 
bound\footnote{Unfortunately, CPLEX does not allow to use the 
\texttt{purgeable} option when using \texttt{addLocal}.}.

% For the scheme \fixedCustomer we also considered a variant, where we 
%implement our own root-cut-loop instead of letting CPLEX handling it. This 
%way, 
%we can directly control which inequalities are kept in the current 
%LP-relaxation. The variant is denoted by \fixedCustomerOwnRoot.

\paragraph{Separation scheme \maxViolated} In this scheme, we simply add the 
most violated 
inequality \eqref{eq:loptimality} for the 
\texttt{maxNumCutsRoot} (\texttt{maxNumCutsTree}) customers that have the 
largest maximum violation at each separation round. We use at most 
\texttt{maxNumSepRoot} 
(\texttt{maxNumSepTree}) separation rounds. If the lower bound value does 
not improve more than $\varepsilon=1e-5$ for \texttt{maxNoImprovements} 
separation rounds at a node in the \blue{\BC} tree (including the 
root-node), we stop the separation at this node. In this scheme, we do not add 
any 
inequalities \eqref{eq:loptimality} at initialization, and set $LB$ to zero 
\blue{(the best deducible lower bound without additional computational effort)}
in 
the beginning. Algorithm~\ref{alg:sepmaxviolated} summarizes the separation 
scheme which is implemented in the \texttt{UserCutCallback}.

\begin{algorithm}[h!tb]   
	\small
	\DontPrintSemicolon                 
	\caption{Separation scheme \maxViolated}
	\label{alg:sepmaxviolated}
	\KwIn{instance $(\cus=\loc,d,p)$ of the \PCP, LP-relaxation $(y^*,z^*)$ of 
	current \blue{\BC} node,  ID \texttt{nodeID} of current 
	\blue{\BC} node, ID \texttt{prevNodeID} of previous \blue{\BC} 
	node, counter \texttt{nodeIterations}, counter 
	\texttt{boundNotImprovedCount}, LP-relaxation bound $z^*_{prev}$ of the 
	previous iteration}
%	\KwOut{set \texttt{VIneqs} of violated inequalities \eqref{eq:loptimality}, initial lower bound $LB$}

	\If{$\texttt{nodeID}=\texttt{prevNodeID}$}
	{	
		\If{$z^*-z^*_{prev}<\varepsilon$}
		{
			$\texttt{boundNotImprovedCount}\gets \texttt{boundNotImprovedCount}+1$\;
			\If{$\texttt{boundNotImprovedCount}=\texttt{maxNoImprovements}$}
			{
				\Return\;
			}
		}
	}
	\Else
	{
		$\texttt{boundNotImprovedCount}\gets 0$\;
		$\texttt{nodeIterations}\gets 0$\;
	}

	$\texttt{maxNumCuts} \gets \texttt{maxNumCutsRoot}$\;
	$\texttt{maxNumSep} \gets \texttt{maxNumSepRoot}$\;
	\If{$\texttt{nodeID}\neq \texttt{rootNode}$}
	{
			$\texttt{maxNumCuts} \gets \texttt{maxNumCutsTree}$\;
			$\texttt{maxNumSep} \gets \texttt{maxNumSepTree}$\;
	}
	\If{$\texttt{nodeIterations}>\texttt{maxNumSep}$}
	{
		\Return\;
	}

 	$z^*_{prev} \gets z^*$\;
	$LB \gets \lceil z^* \rceil$\;

	$\texttt{customerAndViolation} \gets \emptyset$ \;
	\For{$i \in \cus$}
	{
		calculate maximal \texttt{violation} of inequalities \eqref{eq:loptimality} for customer $i$ and given lower bound $LB$ as described in Section \ref{sec:sep}\;
		\If{$\texttt{violation}>0$}
		{
			$\texttt{customerAndViolation} \gets 
			\texttt{customerAndViolation}\xspace \cup 
			\{(i,\texttt{violation})\}$\;	
		}
	}
	sort $\texttt{customerAndViolation}$ in descending order according to 
	\texttt{violation}\;
	\For{$\ell \in \{ 1,\ldots, 
	\min\{\texttt{maxNumCuts},size(\texttt{customerAndViolation})\} \}$}
	{
		$i \gets$ \texttt{customer} from $\texttt{customerAndViolation}[\ell]$\;
		\If{$\texttt{nodeID}= \texttt{rootNode}$}
		{
			add the maximal violated inequalitiy \eqref{eq:loptimality} for customer $i$ and lower bound $LB$ with the option \texttt{purgeable} \;
		}
		\Else
		{
				add the maximal violated inequalitiy \eqref{eq:loptimality} for customer $i$ and lower bound $LB$ as locally valid cut with using \texttt{addLocal}\;
		}
	}
	
\end{algorithm}

\paragraph{Separation scheme \fixedCustomer}
 In this scheme, we restrict the 
separation to a subset $\hat \cus \subseteq \cus$ of the customers, which we 
then iteratively 
grow during the course of the algorithm. The idea behind this separation scheme 
is that for solving the problem to optimality, it can be enough to focus on a 
subset of the customers due to the min-max structure of the objective function. 
This is also used in the set-cover-based approaches 
\cite{chen2009,contardo2019scalable}, where the set cover problem is solved (to 
optimality) on only a subset of the customers, and then it is checked, if any 
of the not-considered customers would change the objective function value. If 
yes, the set cover problem is then adapted with new customers and resolved. 
\blue{Furthermore, note that in  
	this scheme, we sometimes use that $\cus=\loc$ as is customary in the 
	instances \blue{from the literature}. We mention whenever we use that 
	condition 
	\blue{and also detail how to proceed if $\cus \neq \loc$}. }

Our 
separation scheme tries to follow a similar avenue within our \blue{\BC} 
framework. 
In particular, we add the most violated 
inequality \eqref{eq:loptimality} for all customers in $\hat \cus\subseteq 
\cus$ in each 
separation round and use at most 
\texttt{maxNumSepRoot} 
(\texttt{maxNumSepTree}) separation rounds. If the lower bound value does 
not improve more than $\varepsilon=1e-5$ for \texttt{maxNoImprovements} 
separation rounds at a node in the \blue{\BC} tree (including the 
root-node), we stop the separation at this node. 

The procedure to initialize the set  $\hat \cus$ 
%\sout{exploits the assumption $\cus=\loc$.}
\blue{is given as follows.} 
We select a sample $\hat \cus$ of $p+1$ 
%\sout{locations} 
\blue{customers} using the 
following greedy algorithm: We randomly choose an initial customer $i \in 
\cus$ and add it to the empty set $\hat \cus$. We then iteratively grow $\hat 
\cus$ by adding 
the customer $i \in \cus \setminus \hat \cus$ which has the maximum distance 
to its closest customer in 
$\hat \cus$. This is done until $|\hat \cus|=p+1$.

\blue{This procedure to initialize the set  $\hat \cus$ exploits the assumption 
that $\cus=\loc$ by using the distances between pairs of 
customers, which are not necessary available if $\cus \neq \loc$.
In the case that $\cus \neq \loc$ and these distances are available, this 
procedure can be used in the same way. If $\cus \neq \loc$ and these distances 
are not available, 
then one can iteratively add to $\hat \cus$ the customer $i \in \cus \setminus 
\hat \cus$ 
which has 
the maximum value of $ \min_{\hat i \in \hat \cus} \min_{j \in J} \{ d_{i j} + 
d_{\hat i  j} \}$, 
or 
$|\hat \cus|$ can be initialized with $p+1$ random customers.}

As initial value for $LB$ we use 
%\sout{$\min_{i \in \cus, i' \in 
%\blue{\hat \cus}: i\neq i'} \{d_{ii'} \}$}
\blue{$\min_{i \in \hat \cus, j \in \blue{J}: j\neq i} \{d_{ij} \}$}. 
This is a valid lower bound, since 
%\sout{we assume that $\cus=\loc$ and} 
initially $|\hat \cus|=p+1$. Thus at least one 
%\sout{location} 
\blue{customer} 
%\sout{$i' \in \blue{\hat\cus}$}
\blue{$i \in \blue{\hat\cus}$}
cannot be chosen as open facility 
\blue{(which is in any case only possible if $i \in \loc$)}
in a feasible solution, and so its 
minimum distance must be the one to another location 
%\sout{$i\neq i'$ with $i\in \cus = \loc$}
\blue{$j\neq i$ with $j\in \loc$}. 
\blue{Note that in the case $\cus \neq \loc$ the condition $j \neq i$ can 
	become redundant.}
\blue{Furthermore observe that as in the scheme \maxViolated, we initialize 
$LB$ with the best 
deducible lower bound without additional computational
effort.}
%We add 
%inequalities \eqref{eq:loptimality} for all $i \in \bar \cus$ as 
%initialization. 
%Algorithm~\ref{alg:init} summarizes the initialization procedure.

%\begin{algorithm}[h!tb]   
%	\DontPrintSemicolon                 
%	\caption{Algorithm to get an initial set of inequalities 
%\eqref{eq:loptimality} \label{alg:init}}
%	\KwIn{instance $(\cus=\loc,d,p)$ of the \PCP}
%	\KwOut{set \texttt{VIneqs} of violated inequalities \eqref{eq:loptimality}, 
%initial lower bound $LB$}
%	$\bar \cus \gets $ randomly select a customer in $\cus$\;
%	\For{$i \in 1,\ldots,p$}{
%		$i^*=\arg \max_{i \in \cus \setminus \bar \cus} \min_{i' \in \bar \cus} 
%		\{d_{ii'} \}$ \;
%		$\bar \cus \gets \bar \cus \cup \{ i^* \}$
%	}
%	$LB \gets \min_{i \in \cus, i' \in \bar \cus: i\neq i'} \{ d_{ii'} \}$ \;
%	%\texttt{VIneqs} $\gets$ inequalities \eqref{eq:loptimality} defined for 
%	%all customers in $\bar \cus$ considering $LB$\;
%	%\Return{\texttt{VIneqs}, $LB$}		
%\end{algorithm}

The customers to grow $\hat \cus$ are selected as follows: At each separation 
iteration, take the customer $i \in \cus \setminus \hat \cus$ which induces the 
most violated inequality \eqref{eq:loptimality} (if any exist) and add it to 
$\hat \cus$. Moreover, if there are more than \texttt{maxNoImprovementsFixed} 
consecutive iterations, in which the lower bound did not improve, we increase 
$\hat \cus$ more 
aggressively by adding more 
customers. The selection of these additional customers is again driven by 
maximum violation of \eqref{eq:loptimality}, but in order to find a diverse set 
of customers does not 
consider all $i \in \cus \setminus \hat \cus$ as candidates, but a subset $\bar 
\cus \subseteq \cus \setminus \hat \cus$. To build $\bar \cus$, initially we 
set $\bar \cus=\cus \setminus \hat \cus$. 
%\sout{Whenever 
%some $j 
%\in \loc$ with $d_{ij}\leq LB$ occurs
%in 
%\sout{a}
%\blue{the maximum}
%violated inequality 
%\eqref{eq:loptimality} for a customer 
%$i$ we just added to $\hat \cus$, we remove $j$ 
%from $\bar \cus$. }
\blue{Whenever
$d_{ij}\leq LB$ 
holds for some $j \in \cus = \loc$
for a customer 
$i$ we just added to $\hat \cus$,
we remove $j$ from $\bar \cus$.}
\blue{Note that in this case $j$ occurs in the maximum violated inequality 
\eqref{eq:loptimality} for customer~$i$.}
The idea behind this is that if $i$ and $j$ are close to 
each other, then the already added inequality for customer~$i$ may be quite 
similar to the inequality we would add for $j$, thus we prevent adding the 
inequality for $j$. 

\blue{Note that for constructing $\bar \cus$ we exploit that 
$\cus=\loc$, again by using the distances between pairs of customers. If these 
are available for $\cus \neq \loc$, our approach can still be used, only $j$ 
does not appear in the maximum violated inequality 
\eqref{eq:loptimality} for customer~$i$ in this case.
 If $\cus \neq \loc$ and these distances 
are not available, then
one could either delete all 
customers $\hat i$ for which there is a location $\hat j$ such that $d_{i \hat 
j} + d_{\hat i \hat j} \leq LB$, 
or omit the deletion of customers of $\bar \cus$ completely.
}

Algorithm~\ref{alg:sepfixedCustomer}, Part 1
(continued in Algorithm~\ref{alg:sepfixedCustomer2}, Part 2)  summarizes the 
separation 
scheme which is implemented in the \texttt{UserCutCallback}.

\begin{algorithm}[h!tb]  
	\small
	\DontPrintSemicolon                 
	\caption{(Part 1 of 2) Separation scheme \fixedCustomer}
	\label{alg:sepfixedCustomer}
	\KwIn{instance $(\cus=\loc,d,p)$ of the \PCP, LP-relaxation $(y^*,z^*)$ of 
	current \blue{\BC} node, ID \texttt{nodeID} of current 
	\blue{\BC} node, ID \texttt{prevNodeID} of previous \blue{\BC} 
	node, counter \texttt{nodeIterations}, counter 
	\texttt{boundNotImprovedCount}, counter 
	\texttt{boundNotImprovedCountFixed}, LP-relaxation bound $z^*_{prev}$ of 
	the previous iteration, current subset of customers $\hat I$}
	%	\KwOut{set \texttt{VIneqs} of violated inequalities \eqref{eq:loptimality}, initial lower bound $LB$}
	
	\If{$\texttt{nodeID}=\texttt{prevNodeID}$}
	{	
		\If{$z^*-z^*_{prev}<\varepsilon$}
		{
			$\texttt{boundNotImprovedCount}\gets \texttt{boundNotImprovedCount}+1$\;
			$\texttt{boundNotImprovedCountFixed}\gets \texttt{boundNotImprovedCountFixed}+1$\;
			\If{$\texttt{boundNotImprovedCount}=\texttt{maxNoImprovements}$}
			{
				\Return\;
			}
		}
	}
	\Else
	{
		$\texttt{boundNotImprovedCount}\gets 0$\;
		$\texttt{boundNotImprovedCountFixed}\gets 0$\;
		$\texttt{nodeIterations}\gets 0$\;
	}

	%$\texttt{maxNumCuts} \gets \texttt{maxNumCutsRoot}$\;
	$\texttt{maxNumSep} \gets \texttt{maxNumSepRoot}$\;
	\If{$\texttt{nodeID}\neq \texttt{rootNode}$}
	{
		%$\texttt{maxNumCuts} \gets \texttt{maxNumCutsTree}$\;
		$\texttt{maxNumSep} \gets \texttt{maxNumSepTree}$\;
	}
	\If{$\texttt{nodeIterations}>\texttt{maxNumSep}$}
	{
		\Return\;
	}

	$z^*_{prev} \gets z^*$\;
	$LB \gets \lceil z^* \rceil$\;
	
	$\texttt{customerAndViolation} \gets \emptyset$ \;
	\For{$i \in \hat \cus$}
	{
		calculate maximal \texttt{violation} of inequalities \eqref{eq:loptimality} for customer $i$ and given lower bound $LB$ as described in Section \ref{sec:sep}\;
		\If{$\texttt{violation}>0$}
		{
			$\texttt{customerAndViolation} \gets 
			\texttt{customerAndViolation}\xspace \cup \{(i,\texttt{violation}) 
			\}$\;	
		}
	}
	sort $\texttt{customerAndViolation}$ in descending order according to  
	\texttt{violation}\;
	\For{$\ell \in \{ 1,\ldots, 
	size(\texttt{customerAndViolation})\}$}
	{
		$i \gets$ customer from $\texttt{customerAndViolation}[\ell]$\;
		\If{$\texttt{nodeID}= \texttt{rootNode}$}
		{
			add the maximal violated inequalitiy \eqref{eq:loptimality} for customer $i$ and lower bound $LB$ with the option \texttt{purgeable} \;
		}
		\Else
		{
			add the maximal violated inequalitiy \eqref{eq:loptimality} for customer $i$ and lower bound $LB$ as locally valid cut with using \texttt{addLocal}\;
		}
	}
	\nonl continued in Algorithm \ref{alg:sepfixedCustomer}, Part 2 of 2 \;

\end{algorithm}

\addtocounter{algocf}{-1}
\begin{algorithm}[h!tb]   
	\small
	\DontPrintSemicolon     
	\setcounter{AlgoLine}{29}            
	\caption{(Part 2 of 2) Separation scheme \fixedCustomer}
	\label{alg:sepfixedCustomer2}
	\nonl continued from Algorithm \ref{alg:sepfixedCustomer}, Part 1 of 2 \;
	
	$\texttt{customerAndViolationOutside} \gets \emptyset$ \;
	\For{$i \in \cus \setminus \hat \cus$}
	{
		calculate maximal \texttt{violation} of inequalities 
		\eqref{eq:loptimality} for customer $i$ and given lower bound $LB$ as 
		described in Section \ref{sec:sep}\;
		\If{$\texttt{violation}>0$}
		{
			$\texttt{customerAndViolationOutside} \gets 
			\texttt{customerAndViolationOutside}\xspace \cup 
			\{(i,\texttt{violation}) \}$\;	
		}
	}	
	sort $\texttt{customerAndViolationOutside}$ in descending order according 
	to the 
	violation\;
	\If{$\texttt{boundNotImprovedCountFixed}<\texttt{maxNoImprovementsFixed} $}
	{
		\If{size(\texttt{customerAndViolationOutside}) $\geq 1$ }
		{
			$\texttt{mostViolatedCustomer} \gets$ customer from 
			$\texttt{customerAndViolationOutside}[1]$\;
			$\hat \cus \gets \hat \cus \cup  \{\texttt{mostViolatedCustomer}\}$\;
			\If{$\texttt{nodeID}= \texttt{rootNode}$}
			{
				add the maximal violated inequalitiy \eqref{eq:loptimality} for 
				customer \texttt{mostViolatedCustomer} and lower bound $LB$ 
				with the option \texttt{purgeable} \;
			}
			\Else
			{
				add the maximal violated inequalitiy \eqref{eq:loptimality} for 
				customer \texttt{mostViolatedCustomer} and lower bound $LB$ as 
				locally valid cut with using \texttt{addLocal}\;
			}
			
		}
	}
	\Else
	{
		$\texttt{boundNotImprovedCountFixed} \gets 0$\;
		$\bar \cus \gets \cus \setminus \hat \cus$\;
		\For{$\ell \in \{1,\ldots, 
		size(\texttt{customerAndViolationOutside})\}$}
		{
			$i \gets$ customer from 
			$\texttt{customerAndViolationOutside}[\ell]$\;
			\If{\blue{$i \in \bar \cus$}}
			{
				$\hat \cus \gets \hat \cus \cup \{i\}$\;
				\If{$\texttt{nodeID}= \texttt{rootNode}$}
				{
					add the maximal violated inequalitiy \eqref{eq:loptimality} for customer $i$ and lower bound $LB$ with the option \texttt{purgeable} \;
				}
				\Else
				{
					add the maximal violated inequalitiy \eqref{eq:loptimality} for customer $i$ and lower bound $LB$ as locally valid cut with using \texttt{addLocal}\;
				}
				\For{$j$ with $d_{ij}\leq LB$ 
					\blue{ (this 
				implies that $j$ occurs 
				in} the maximal 
				violated 
				inequalitiy \eqref{eq:loptimality} for customer $i$ and lower 
				bound 
				$LB$\blue{)} }
				{
					$\bar \cus \gets \bar \cus \setminus \{j\} $
				}	
			}	
		}

	}	
\end{algorithm}

%
%\paragraph{Separation scheme variant \fixedCustomerOwnRoot} In the previous 
%schemes, we relied on the \texttt{purgeable} option of CPLEX to remove 
%redundant inequalities \eqref{eq:loptimality}. In the scheme 
%\fixedCustomerOwnRoot we implemented our own cut-loop for the root-node, i.e., 
%before we start the \blue{\BC} of CPLEX, we iteratively solve the 
%LP-relaxation 
%of our model using CPLEX, and add violated inequalities 
%\eqref{eq:loptimality}. 
%This is done in as described in Algorithm \ref{alg:sepfixedCustomerOwnRoot} 
%and 
%closely resembles the scheme \fixedCustomer. Compared to \fixedCustomer, we 
%add 
%two inequalities \eqref{eq:loptimality} for each customer in $\hat \cus$ in a 
%separation round, namely the most violated one and also the one for facility 
%$j_i \in \arg \min_{j \in \loc} \{d_{ij} :d_{ij}>LB\}$ (if these two 
%inequalities coincide, we of course only add it once). This is done for faster 
%convergence. In fact, we do not add the inequalities to the previous LP-model, 
%but rebuild the model again from scratch. This is done as for our problem, in 
%case the bound $LB$ has changed, this allows replacing the previous 
%inequalities \eqref{eq:loptimality} with stronger ones. When there are no more 
%violated inequalities, we use the inequalities in our final LP-model to 
%initialize the \blue{\BC}, which then uses separation scheme \fixedCustomer 
%(where the initial  $\hat \cus$ is the customer subset we have at the end of 
%our cut-loop).
%

\subsection{Primal heuristic \label{sec:heur}}

In order to obtain primal solutions, we have implemented a greedy heuristic
in the \texttt{HeuristicCallback} of CPLEX. It uses the optimal solution 
$(y^*,z^*)$ of 
the current LP-relaxation at a \blue{\BC} node. Let~$S^H$ be the indices of 
the facilities to open in the solution produced by our heuristic. To construct  
$S^H$, we sort all the locations 
with $y_j^*>0$ in descending order and then iterate through the sorted list. 
Whenever adding a location $j$ from the list to $S^H$ would improve the 
objective function value of the heuristic solution (which initially is set to 
$\infty$), 
we add $j$ to $S^H$. This is done until $|S^H|=p$. If needed, we make multiple 
passes through the list.

\section{Computational results} 
\label{sec:results}

The runs were made on a single core of an Intel Xeon E5-2670v2 machine with 2.5 
GHz and 32GB of RAM, and all CPLEX settings were left on their default values. 
The timelimit for the runs was set to 1800 seconds. The following parameter 
values were used for our algorithm, they were determined in preliminary 
computations: 
\ifArXiV
\begin{multicols}{2}
\fi
\begin{itemize}
	\setlength\itemsep{0em}
	\item \texttt{maxNumCutsRoot}: 100
	\item \texttt{maxNumCutsTree}: 50
	\item \texttt{maxNumSepRoot}: 1000
	\item \texttt{maxNumSepTree}: 1
	\item \texttt{maxNoImprovements}: 100
	\item \texttt{maxNoImprovementsFixed}: 5
\end{itemize}
\ifArXiV
\end{multicols}
\fi
	
\subsection{Instances}

We used two sets of instances \blue{from the literature} to evaluate the performance of 
our algorithm, these sets are denoted by \texttt{pmedian} and \texttt{TSPLIB} 
and details are given below.
\begin{itemize}
\item \texttt{pmedian}: This is an instance set used in \cite{calik2013double, chen2009, contardo2019scalable}. The set contains 40 instances. All instances 
have all customer demand points as potential facility locations, so $\cus 
= \loc  = V$ holds. The number of customer demand points and hence also the 
number of potential facility locations $|\cus| = |\loc| = |V|$ is between 
$100$ and $900$, and $p$ is between $5$ and $200$.
For the concrete values of $|V|$ and $p$ for each instance see 
Table~\ref{ta:pmed}. 
\item  \texttt{TSPLIB}: This set contains larger instances and is based on the TSP-library \citep{reinelt1991tsplib}.  The number of customer demand points and hence also the number of potential facility locations $|\cus| = |\loc| = |V|$ is between 
$1621$ and $744710$, for the concrete values of $|V|$ see 
\blue{the tables in the appendix}.
 The complete set of 44 instances was used in \cite{contardo2019scalable}. In 
 other works \citep{elloumi2004, chen2009, calik2013double} used only two or 
 three 
 (smaller) instances (with up to 3038 vertices) to test the 
 scalability of the respective algorithms\footnote{In \cite{chen2009} also 
 other smaller TSPlib instances with less than 1000 vertices were used, but 
 they 
 are not included in the set used in \cite{contardo2019scalable}.}. In 
 \cite{contardo2019scalable} the instances were used with 
 $p=2,3,5,10,15,20,25,30$, while in the other works also larger values of $p$ 
 were considered.
 In the instances all customer demand points are given as 
 two-dimensional coordinates, and 
 the Euclidean distance rounded down to the nearest integer is used as 
 distance. This follows
 previous literature.
\end{itemize}

\subsection{Results}

We first give a comparison of our approaches with the lifted inequalities 
\eqref{eq:loptimality} with a \blue{\BC} using our new formulation 
\myref{PC2} without the lifting on instance set \texttt{pmedian}. Then we 
describe the results obtained by using our approaches with the lifted 
inequalities \eqref{eq:loptimality} for the larger instance set \texttt{TSPLIB}.

\paragraph{Results for the instance set \texttt{pmedian}}

\newcommand{\mVnoL}{\texttt{mVnoL}\xspace}
\newcommand{\fCnoL}{\texttt{fCnoL}\xspace}
\newcommand{\mVL}{\texttt{mVL}\xspace}
\newcommand{\fCL}{\texttt{fCL}\xspace}
\newcommand{\mVLH}{\texttt{mVLH}\xspace}
\newcommand{\fCLH}{\texttt{fCLH}\xspace}

In Table \ref{ta:pmed}, we compare the following six different configurations:
\begin{itemize}
	\setlength\itemsep{0em}
	\item \mVnoL: separation scheme \maxViolated and inequalities \eqref{eq:cuts}
	\item \fCnoL: separation scheme \fixedCustomer and inequalities 
	\eqref{eq:cuts}
	\item \mVL: separation scheme \maxViolated and lifted inequalities \eqref{eq:loptimality}
	\item \fCL: separation scheme \fixedCustomer and lifted inequalities \eqref{eq:loptimality}
	\item \mVLH: configuration \mVL and the primal heuristic described in Section \ref{sec:heur}
	\item \fCLH: configuration \fCL and the primal heuristic described in Section \ref{sec:heur}
\end{itemize}

In Table \ref{ta:pmed}, we see the huge effect of the lifted inequalities 
\eqref{eq:loptimality}. We report the runtime in seconds ($t[s]$), the 
optimality gap ($g[\%]$, which is computed as $\frac{UB-LB}{UB}\cdot 100$, 
where $LB$ is the obtained lower bound and $UB$ is the obtained upper bound) 
and the number of \blue{\BC} nodes ($\#BC$). Without lifting, only 11 of 
the 40 instances are solvable within the timelimit of 1800 seconds, while with 
lifting all the instances can be solved to optimality within four seconds 
runtime. 
Thus, our lifting has an immense positive effect.
For these instances, the primal heuristic has no discernible effect, 
as they already become solvable very fast when using the lifting. 
Also the choice of the separation scheme (\maxViolated or \fixedCustomer) 
does not influence the quality of the results a lot.
We also note 
that most of the instances can be solved within the root-node of the 
\blue{\BC} algorithm when using our lifting. 

\paragraph{Results for the instance set \texttt{TSPlib}}

In Figures \ref{fig:plotsa}-\ref{fig:plotsb}, we show plots of the runtime and 
optimality gap aggregated by $p$ for configurations \mVL, \fCL, \mVLH and 
\fCLH, i.e., the by far best configurations for the instance set 
\texttt{pmedian} that contains smaller instances.
We used the same values  for $p$ as \cite{contardo2019scalable}. Detailed 
results for the setting \fCLH can be found in the appendix in the Tables 
\ref{ta:tsplib2}-\ref{ta:tsplib30}, were we also include the results obtained 
by \cite{contardo2019scalable}. From the figures, we see that the instances are 
becoming harder to solve when $p$ becomes larger, 
which is in tune with the computational experiments of 
\cite{contardo2019scalable}.
We note that for the settings 
\mVL and \mVLH the code did not terminate for all instances, as the available 
memory was not enough. In general, the settings \fCL and \fCLH perform better 
than \mVL and \mVLH. An explanation for this could be that by focusing the 
separation of inequalities \eqref{eq:loptimality} on the subset $\hat \cus$, 
the separation procedure becomes more stable, which could allow CPLEX to easier 
decide \blue{which} inequalities to remove for further iterations. Hence, the 
LP-relaxations becomes smaller and easier to solve, and also not so 
memory-intensive. 

\blue{To further illustrate the different behavior of both 
separation schemes, Figure \ref{fig:bound} shows the obtained lower bound 
values at the root node of the \blue{\BC} tree against the runtime for 
settings \mVLH and \fCLH for the \texttt{TSPlib} instance \texttt{rl11849} for 
$p=5$. While both 
settings manage to solve this instance to opimality in the root node, we see 
that \fCLH converges much faster. The behavior of the lower bound in this 
instance is a typical behavior we observed for both strategies. 

Figures \ref{fig:numcutsa} and \ref{fig:numcutsb} show the number of added 
inequalities \eqref{eq:loptimality} for all \texttt{TSPlib} instances with 
$p=2$ and $p=5$ 
which could be solved by all four considered settings. Theses figures further 
support that with the scheme \fixedCustomer much less inequalities need to be 
added. We also see that for larger $p$ much more inequalities need to be 
added.  For the scheme \fixedCustomer the maximum number of added inequalities 
for 
$p=2$ is slightly over 500, while for $p=5$ it is around 10000. For the 
scheme 
\maxViolated this number grows from around 4000 for $p=2$ to around 80000 for 
$p=5$. We note that the number of added inequalities is not necessarily the 
same as the number of inequalities which are in the final LP-relaxations, as 
the 
CPLEX cut management may delete added violated inequalities during the course 
of the \blue{\BC}.}

From the Figures \ref{fig:plotsa}-\ref{fig:plotsb}, we also see that there is 
some effect 
of using the primal heuristic, in particular regarding the optimality gap. This 
is consistent when comparing our obtained lower and upper bound values with the 
values obtained by \cite{contardo2019scalable}, i.e., from the comparison we 
see that our approach provides good lower bounds (sometimes even improving on 
the lower bound of \cite{contardo2019scalable}, which they obtained within a 
timelimit of 24 hours), while the upper bounds for some instances could still 
be improved. This again shows the important contribution of our lifting 
procedure.

\blue{From the results in Tables 
	\ref{ta:tsplib2}-\ref{ta:tsplib30} we can also see that in some instances 
	for larger $p$ our approach scales worse compared to the approach of 
	\cite{contardo2019scalable}. We believe that this could be due to the 
	combination of the following factors: i) our method is a ``classical'' 
	\blue{\BC} approach, thus the strength of the LP-relaxation is of 
	course crucial. For larger $p$ the feasible region of the LP-relaxation is 
	getting larger, and thus the LP-relaxation may also become weaker; ii) 
	state-of-the-art methods such as \cite{contardo2019scalable} work by 
	iteratively resolving set cover problems. For these problems, the used 
	MIP-solvers may be able to use additional preprocessing routines, 
	general-purpose cuts or other speed-up tricks, which are nowadays built in 
	such solvers; iii) \cite{contardo2019scalable} also uses a specific 
	preprocessing which is applied before each set cover problem. The procedure is not 
	described, thus it is also clear, if such a procedure could be directly 
	transferable into our \blue{\BC} approach; iv) in some cases the 
	obtained primal bound is the bottleneck for our algorithm, even though we 
	already implemented a primal heuristic.}

%\begin{landscape}
%\begin{table}[ht]
\begin{sidewaystable}[h!tp]	
	\setlength{\tabcolsep}{5pt}
	\centering
	\caption{Detailed results for the \texttt{pmedian} instances\label{ta:pmed}} 
	\begingroup\footnotesize
	\begin{tabular}{lll|rrr|rrr|rrr|rrr|rrr|rrr}
		\toprule
		& & & \multicolumn{3}{|c}{$\mVnoL$}  & \multicolumn{3}{|c}{$\fCnoL$} & 
		\multicolumn{3}{|c}{$\mVL$} & \multicolumn{3}{|c}{$\fCL$} & 
		\multicolumn{3}{|c}{$\mVLH$} & \multicolumn{3}{|c}{$\fCLH$} \\ name & 
		$|V|$ & $p$ & $t[s]$ & $g[\%]$ & \#BC & $t[s]$ & $g[\%]$ & \#BC  & 
		$t[s]$ & $g[\%]$ & \#BC  & $t[s]$ & $g[\%]$ & \#BC  & $t[s]$ & $g[\%]$ 
		& \#BC  & $t[s]$ & $g[\%]$ & \#BC\\ \midrule
		pmed1 & 100 & 5 &   57.72 & 0.00 & 115902 &   52.96 & 0.00 & 114523 & 0.14 & 0.00 & 45 & 0.11 & 0.00 & 46 & 0.15 & 0.00 & 71 & 0.10 & 0.00 & 27 \\ 
		pmed2 & 100 & 10 & 1198.51 & 0.00 & 2384887 & 1695.92 & 0.00 & 3342277 & 0.04 & 0.00 & 0 & 0.05 & 0.00 & 0 & 0.05 & 0.00 & 0 & 0.05 & 0.00 & 0 \\ 
		pmed3 & 100 & 10 & 1163.21 & 0.00 & 1862933 &  729.00 & 0.00 & 1250156 & 0.04 & 0.00 & 0 & 0.04 & 0.00 & 0 & 0.05 & 0.00 & 0 & 0.05 & 0.00 & 0 \\ 
		pmed4 & 100 & 20 & TL & 10.81 & 1616178 & TL & 10.81 & 1877802 & 0.05 & 0.00 & 0 & 0.05 & 0.00 & 0 & 0.05 & 0.00 & 0 & 0.05 & 0.00 & 0 \\ 
		pmed5 & 100 & 33 & TL & 8.33 & 2454807 & TL & 8.33 & 2367363 & 0.05 & 0.00 & 0 & 0.05 & 0.00 & 0 & 0.05 & 0.00 & 0 & 0.05 & 0.00 & 0 \\ 
		pmed6 & 200 & 5 &   52.98 & 0.00 & 60870 &   61.15 & 0.00 & 79822 & 0.11 & 0.00 & 3 & 0.11 & 0.00 & 3 & 0.12 & 0.00 & 3 & 0.10 & 0.00 & 3 \\ 
		pmed7 & 200 & 10 & TL & 9.38 & 1139346 & TL & 7.81 & 1294446 & 0.09 & 0.00 & 0 & 0.08 & 0.00 & 0 & 0.09 & 0.00 & 0 & 0.08 & 0.00 & 0 \\ 
		pmed8 & 200 & 20 & TL & 20.00 & 778085 & TL & 20.00 & 794555 & 0.11 & 0.00 & 0 & 0.10 & 0.00 & 0 & 0.12 & 0.00 & 0 & 0.10 & 0.00 & 0 \\ 
		pmed9 & 200 & 40 & TL & 18.92 & 543618 & TL & 18.92 & 581900 & 0.14 & 0.00 & 0 & 0.15 & 0.00 & 0 & 0.14 & 0.00 & 0 & 0.16 & 0.00 & 0 \\ 
		pmed10 & 200 & 67 & TL & 25.00 & 488041 & TL & 25.00 & 560078 & 0.13 & 0.00 & 0 & 0.15 & 0.00 & 0 & 0.13 & 0.00 & 0 & 0.16 & 0.00 & 0 \\ 
		pmed11 & 300 & 5 &   73.09 & 0.00 & 109946 &   60.72 & 0.00 & 89820 & 0.16 & 0.00 & 2 & 0.13 & 0.00 & 0 & 0.16 & 0.00 & 2 & 0.13 & 0.00 & 0 \\ 
		pmed12 & 300 & 10 & TL & 5.88 & 979558 & TL & 3.92 & 1110223 & 0.17 & 0.00 & 0 & 0.16 & 0.00 & 0 & 0.16 & 0.00 & 0 & 0.15 & 0.00 & 0 \\ 
		pmed13 & 300 & 30 & TL & 16.67 & 402537 & TL & 22.22 & 383634 & 0.21 & 0.00 & 0 & 0.24 & 0.00 & 0 & 0.22 & 0.00 & 0 & 0.24 & 0.00 & 0 \\ 
		pmed14 & 300 & 60 & TL & 23.08 & 260750 & TL & 23.08 & 252826 & 0.36 & 0.00 & 2 & 0.36 & 0.00 & 0 & 0.37 & 0.00 & 2 & 0.37 & 0.00 & 0 \\ 
		pmed15 & 300 & 100 & TL & 27.78 & 221453 & TL & 27.78 & 216826 & 0.33 & 0.00 & 0 & 0.36 & 0.00 & 0 & 0.34 & 0.00 & 0 & 0.38 & 0.00 & 0 \\ 
		pmed16 & 400 & 5 &   24.67 & 0.00 & 23852 &   22.56 & 0.00 & 22601 & 0.27 & 0.00 & 0 & 0.25 & 0.00 & 0 & 0.28 & 0.00 & 0 & 0.26 & 0.00 & 0 \\ 
		pmed17 & 400 & 10 & TL & 7.69 & 618267 & TL & 7.69 & 672741 & 0.32 & 0.00 & 0 & 0.30 & 0.00 & 0 & 0.32 & 0.00 & 0 & 0.33 & 0.00 & 0 \\ 
		pmed18 & 400 & 40 & TL & 21.43 & 247230 & TL & 21.43 & 280508 & 0.49 & 0.00 & 0 & 0.62 & 0.00 & 9 & 0.50 & 0.00 & 0 & 0.69 & 0.00 & 15 \\ 
		pmed19 & 400 & 80 & TL & 22.22 & 155689 & TL & 22.22 & 158464 & 0.72 & 0.00 & 0 & 0.63 & 0.00 & 0 & 0.74 & 0.00 & 0 & 0.65 & 0.00 & 0 \\ 
		pmed20 & 400 & 133 & TL & 23.08 & 130236 & TL & 23.08 & 135130 & 0.57 & 0.00 & 0 & 0.61 & 0.00 & 0 & 0.57 & 0.00 & 0 & 0.64 & 0.00 & 0 \\ 
		pmed21 & 500 & 5 &  381.19 & 0.00 & 337602 &  377.07 & 0.00 & 334501 & 0.48 & 0.00 & 0 & 0.47 & 0.00 & 0 & 0.48 & 0.00 & 0 & 0.47 & 0.00 & 0 \\ 
		pmed22 & 500 & 10 & TL & 7.89 & 501977 & TL & 7.89 & 512749 & 0.67 & 0.00 & 17 & 0.86 & 0.00 & 89 & 0.75 & 0.00 & 61 & 0.97 & 0.00 & 130 \\ 
		pmed23 & 500 & 50 & TL & 22.73 & 182841 & TL & 22.73 & 159898 & 0.79 & 0.00 & 0 & 0.85 & 0.00 & 6 & 0.81 & 0.00 & 0 & 0.87 & 0.00 & 6 \\ 
		pmed24 & 500 & 100 & TL & 26.67 & 102612 & TL & 26.67 & 102117 & 0.77 & 0.00 & 0 & 0.83 & 0.00 & 0 & 0.78 & 0.00 & 0 & 0.84 & 0.00 & 0 \\ 
		pmed25 & 500 & 167 & TL & 27.27 & 75362 & TL & 27.27 & 79665 & 0.92 & 0.00 & 0 & 0.98 & 0.00 & 0 & 0.95 & 0.00 & 0 & 1.02 & 0.00 & 0 \\ 
		pmed26 & 600 & 5 &  180.59 & 0.00 & 115119 &  183.90 & 0.00 & 119306 & 0.84 & 0.00 & 2 & 0.80 & 0.00 & 1 & 0.86 & 0.00 & 15 & 0.80 & 0.00 & 1 \\ 
		pmed27 & 600 & 10 & TL & 9.38 & 533068 & TL & 9.38 & 496703 & 0.82 & 0.00 & 0 & 0.82 & 0.00 & 0 & 0.82 & 0.00 & 0 & 0.83 & 0.00 & 0 \\ 
		pmed28 & 600 & 60 & TL & 22.22 & 132018 & TL & 22.22 & 144764 & 1.04 & 0.00 & 0 & 1.05 & 0.00 & 0 & 1.06 & 0.00 & 0 & 1.08 & 0.00 & 0 \\ 
		pmed29 & 600 & 120 & TL & 23.08 & 75833 & TL & 23.08 & 78090 & 1.17 & 0.00 & 0 & 1.23 & 0.00 & 0 & 1.20 & 0.00 & 0 & 1.26 & 0.00 & 0 \\ 
		pmed30 & 600 & 200 & TL & 22.22 & 51632 & TL & 22.22 & 54290 & 1.34 & 0.00 & 0 & 1.42 & 0.00 & 0 & 1.37 & 0.00 & 0 & 1.47 & 0.00 & 0 \\ 
		pmed31 & 700 & 5 &  746.29 & 0.00 & 401934 &  479.78 & 0.00 & 234253 & 1.24 & 0.00 & 0 & 1.20 & 0.00 & 0 & 1.22 & 0.00 & 0 & 1.19 & 0.00 & 0 \\ 
		pmed32 & 700 & 10 & TL & 10.34 & 347742 & TL & 10.34 & 368273 & 1.39 & 0.00 & 10 & 1.59 & 0.00 & 60 & 1.40 & 0.00 & 10 & 1.53 & 0.00 & 36 \\ 
		pmed33 & 700 & 70 & TL & 20.00 & 86212 & TL & 20.00 & 92720 & 1.63 & 0.00 & 0 & 1.78 & 0.00 & 0 & 1.53 & 0.00 & 0 & 1.77 & 0.00 & 0 \\ 
		pmed34 & 700 & 140 & TL & 27.27 & 57991 & TL & 27.27 & 55186 & 1.75 & 0.00 & 0 & 1.75 & 0.00 & 0 & 1.78 & 0.00 & 0 & 1.74 & 0.00 & 0 \\ 
		pmed35 & 800 & 5 &   66.68 & 0.00 & 41432 &   69.04 & 0.00 & 41664 & 1.76 & 0.00 & 3 & 1.74 & 0.00 & 0 & 1.72 & 0.00 & 0 & 1.74 & 0.00 & 0 \\ 
		pmed36 & 800 & 10 & TL & 7.41 & 339494 & TL & 7.41 & 380960 & 1.92 & 0.00 & 9 & 1.92 & 0.00 & 5 & 1.93 & 0.00 & 9 & 1.95 & 0.00 & 5 \\ 
		pmed37 & 800 & 80 & TL & 26.67 & 79136 & TL & 26.67 & 67520 & 2.44 & 0.00 & 0 & 2.26 & 0.00 & 0 & 2.47 & 0.00 & 0 & 2.27 & 0.00 & 0 \\ 
		pmed38 & 900 & 5 &   43.10 & 0.00 & 23546 &   31.83 & 0.00 & 16249 & 2.50 & 0.00 & 0 & 2.47 & 0.00 & 0 & 2.50 & 0.00 & 0 & 2.48 & 0.00 & 0 \\ 
		pmed39 & 900 & 10 & TL & 4.35 & 364938 & TL & 4.35 & 338705 & 2.70 & 0.00 & 28 & 2.53 & 0.00 & 0 & 2.56 & 0.00 & 0 & 2.55 & 0.00 & 0 \\ 
		pmed40 & 900 & 90 & TL & 23.08 & 55589 & TL & 23.08 & 54009 & 3.44 & 0.00 & 0 & 3.26 & 0.00 & 0 & 3.45 & 0.00 & 0 & 3.30 & 0.00 & 0 \\ 
		\bottomrule
	\end{tabular}
	\endgroup
\end{sidewaystable}	
%\end{table}
%\end{landscape}

\begin{figure}[h!tb]
	\begin{subfigure}[b]{.49\linewidth}
		\centering
		\includegraphics[width=\textwidth]{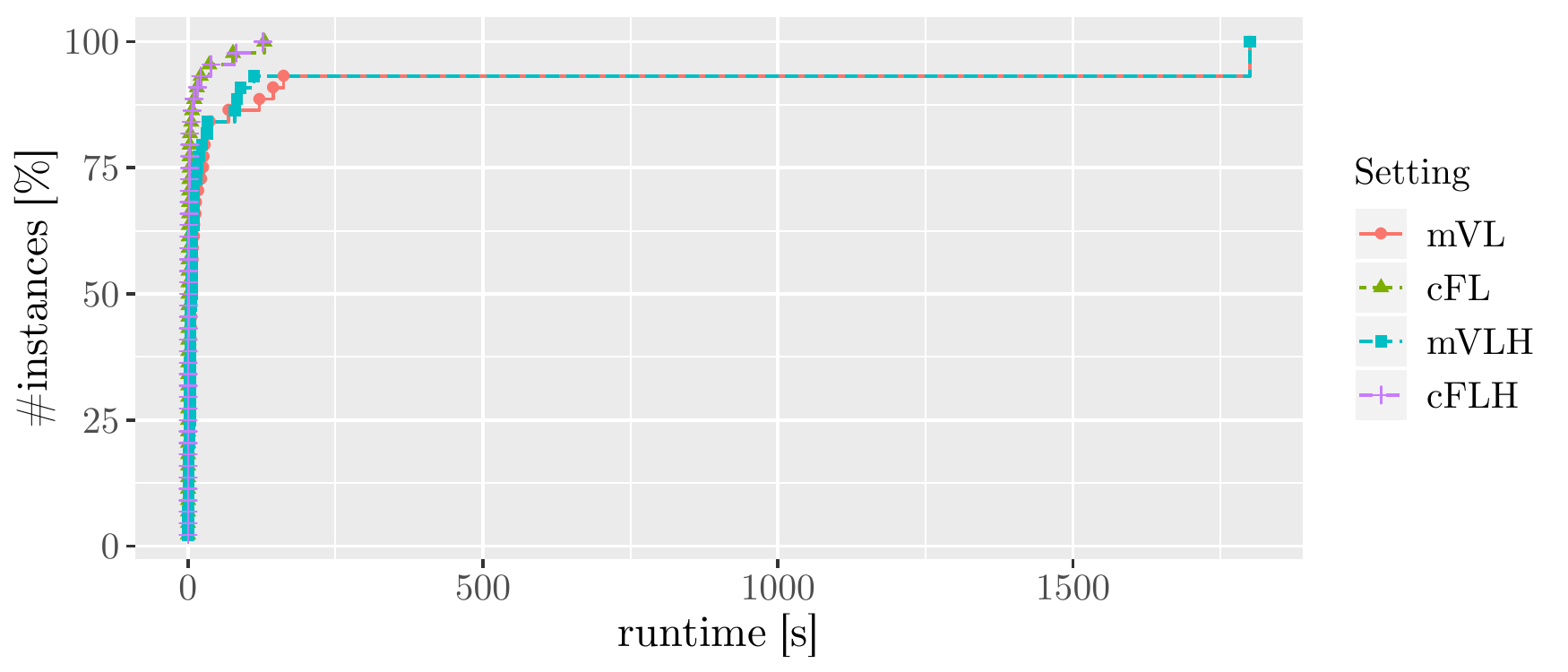}
		\caption{$p=2$, runtime}
	\end{subfigure}
	\begin{subfigure}[b]{.49\linewidth}
	\centering
	\includegraphics[width=\textwidth]{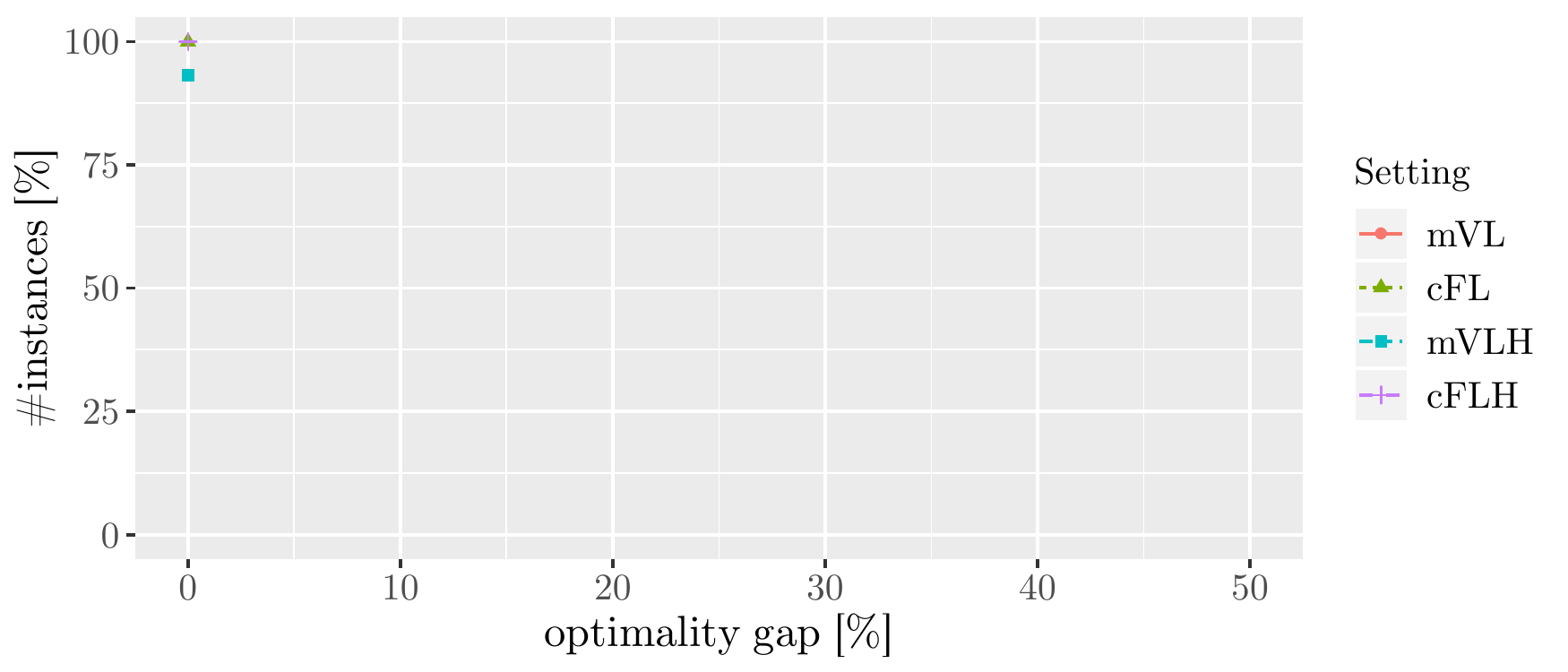}
	\caption{$p=2$, optimality gap}
\end{subfigure}

	\begin{subfigure}[b]{.49\linewidth}
	\centering
	\includegraphics[width=\textwidth]{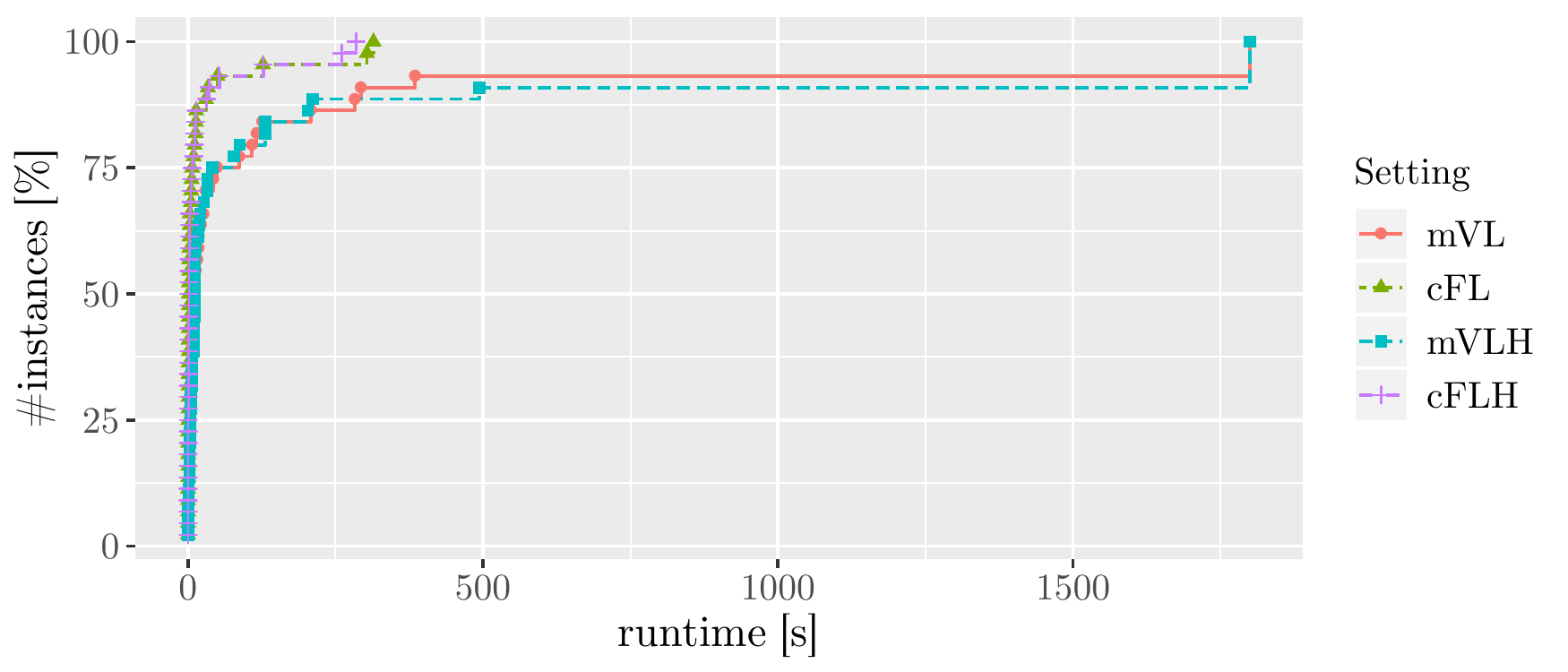}
	\caption{$p=3$, runtime}
\end{subfigure}
\begin{subfigure}[b]{.49\linewidth}
	\centering
	\includegraphics[width=\textwidth]{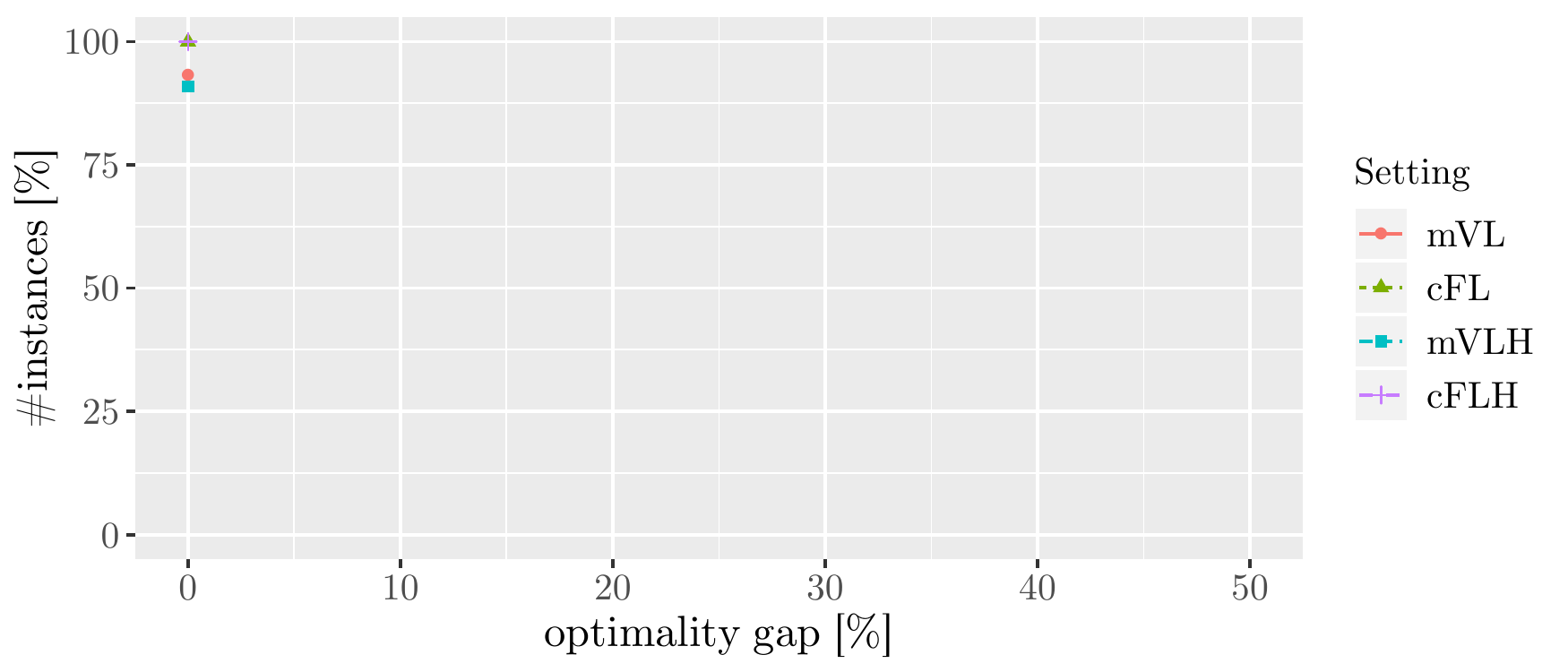}
	\caption{$p=3$, optimality gap}
\end{subfigure}

	\begin{subfigure}[b]{.49\linewidth}
	\centering
	\includegraphics[width=\textwidth]{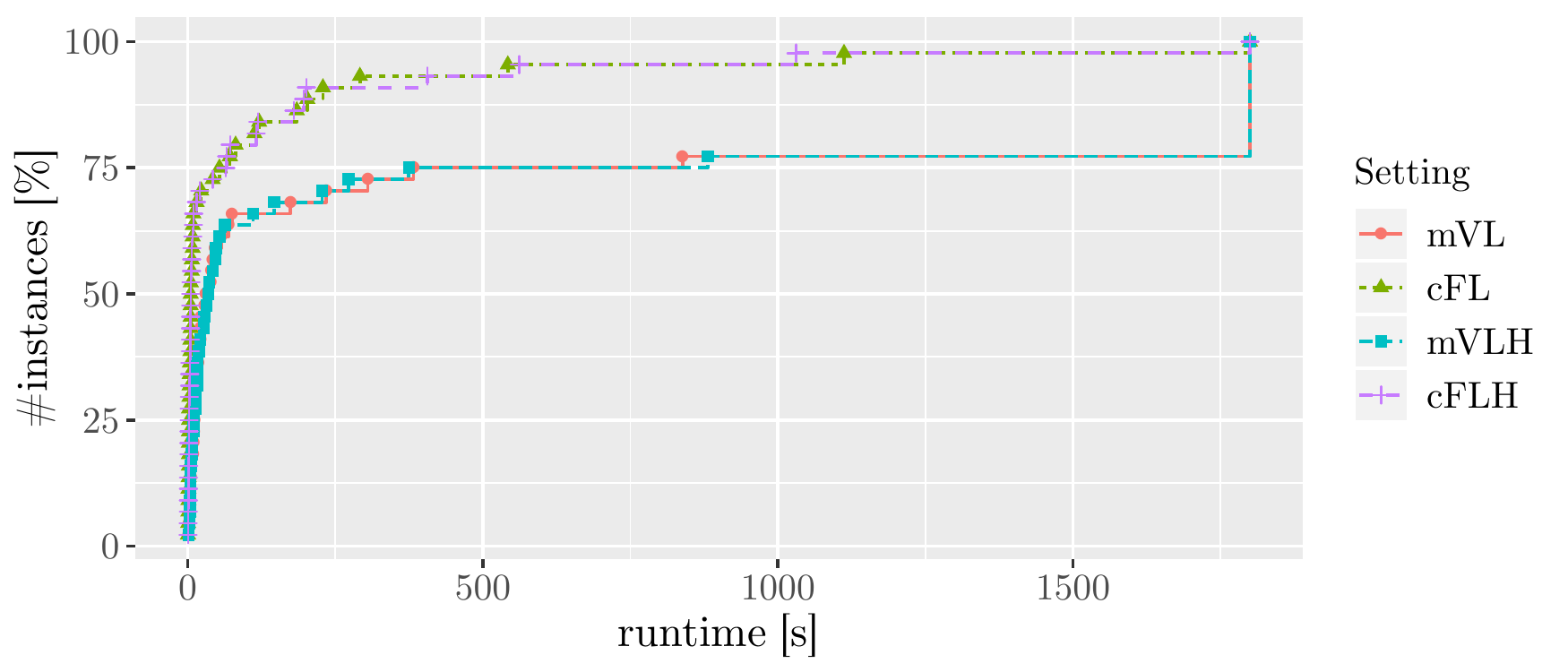}
	\caption{$p=5$, runtime}
\end{subfigure}
\begin{subfigure}[b]{.49\linewidth}
	\centering
	\includegraphics[width=\textwidth]{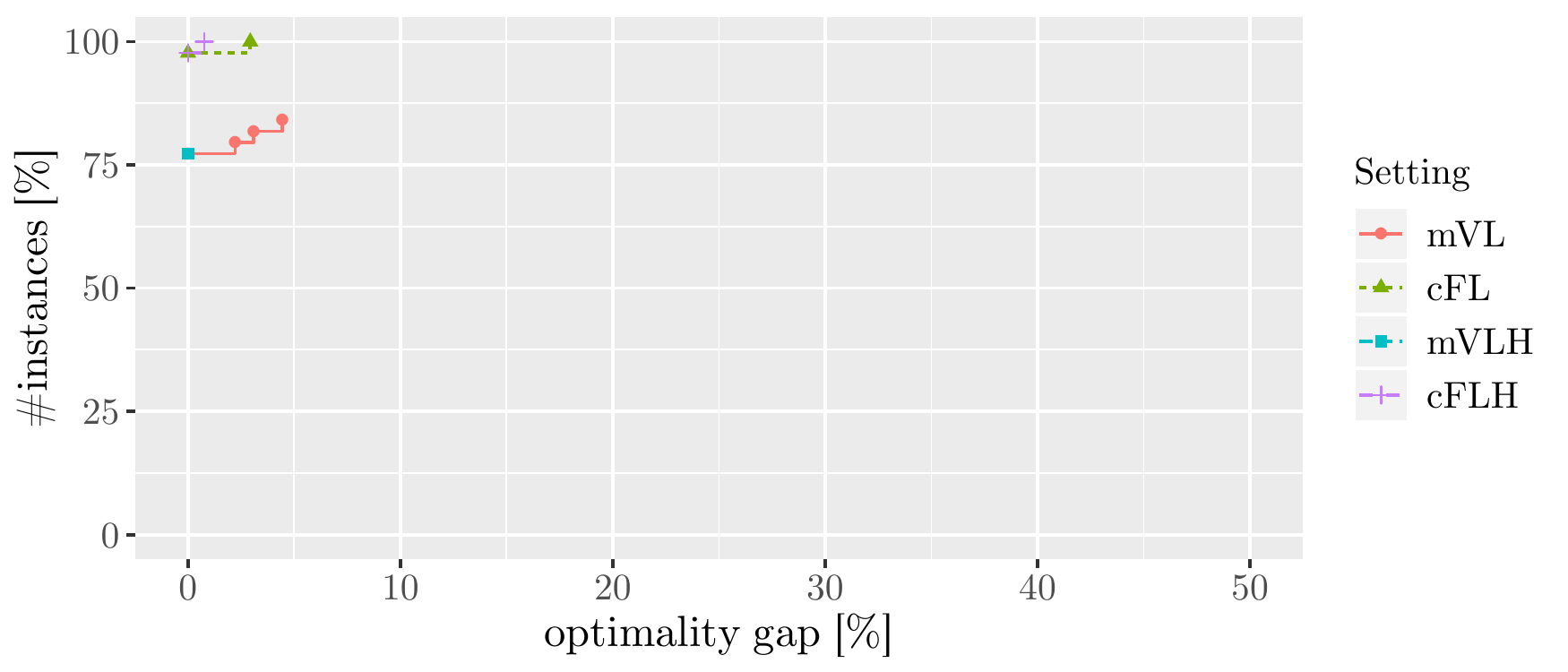}
	\caption{$p=5$, optimality gap}
\end{subfigure}

\begin{subfigure}[b]{.49\linewidth}
	\centering
	\includegraphics[width=\textwidth]{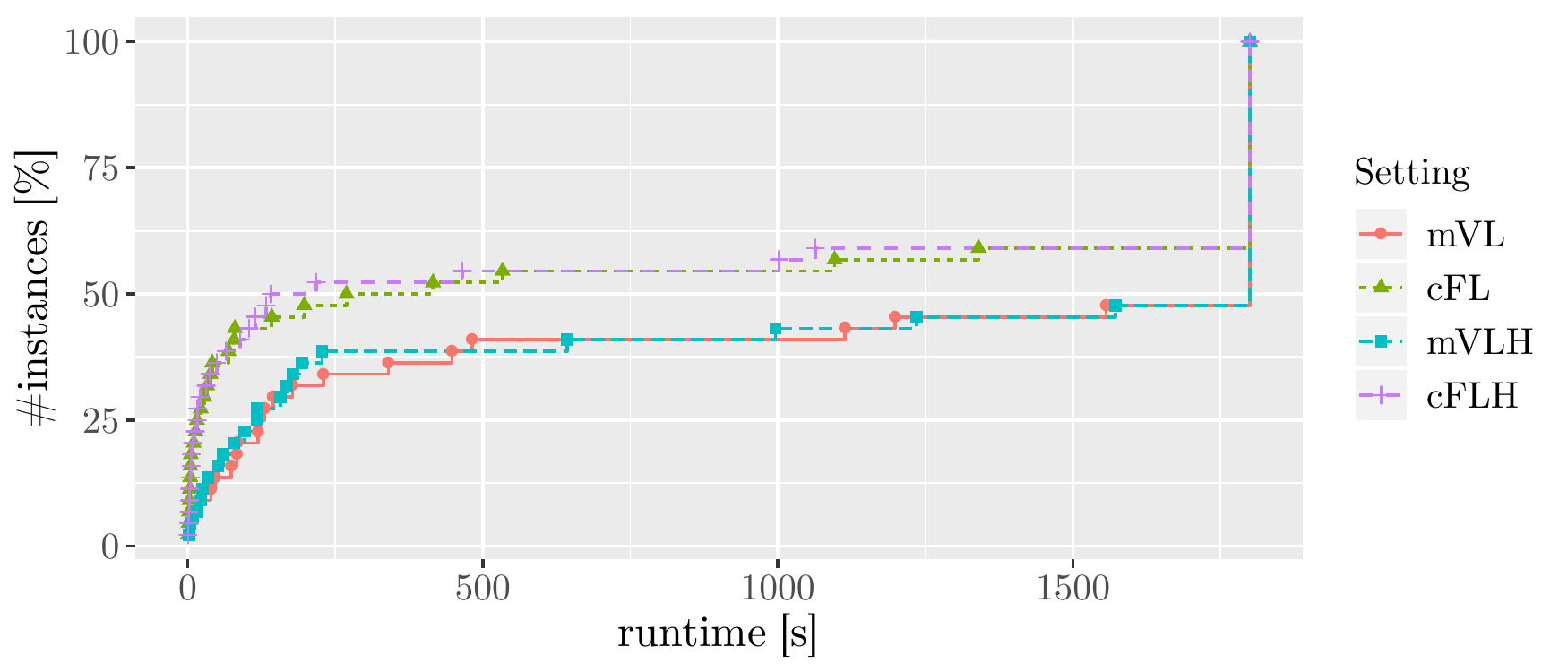}
	\caption{$p=10$, runtime}
\end{subfigure}
\begin{subfigure}[b]{.49\linewidth}
	\centering
	\includegraphics[width=\textwidth]{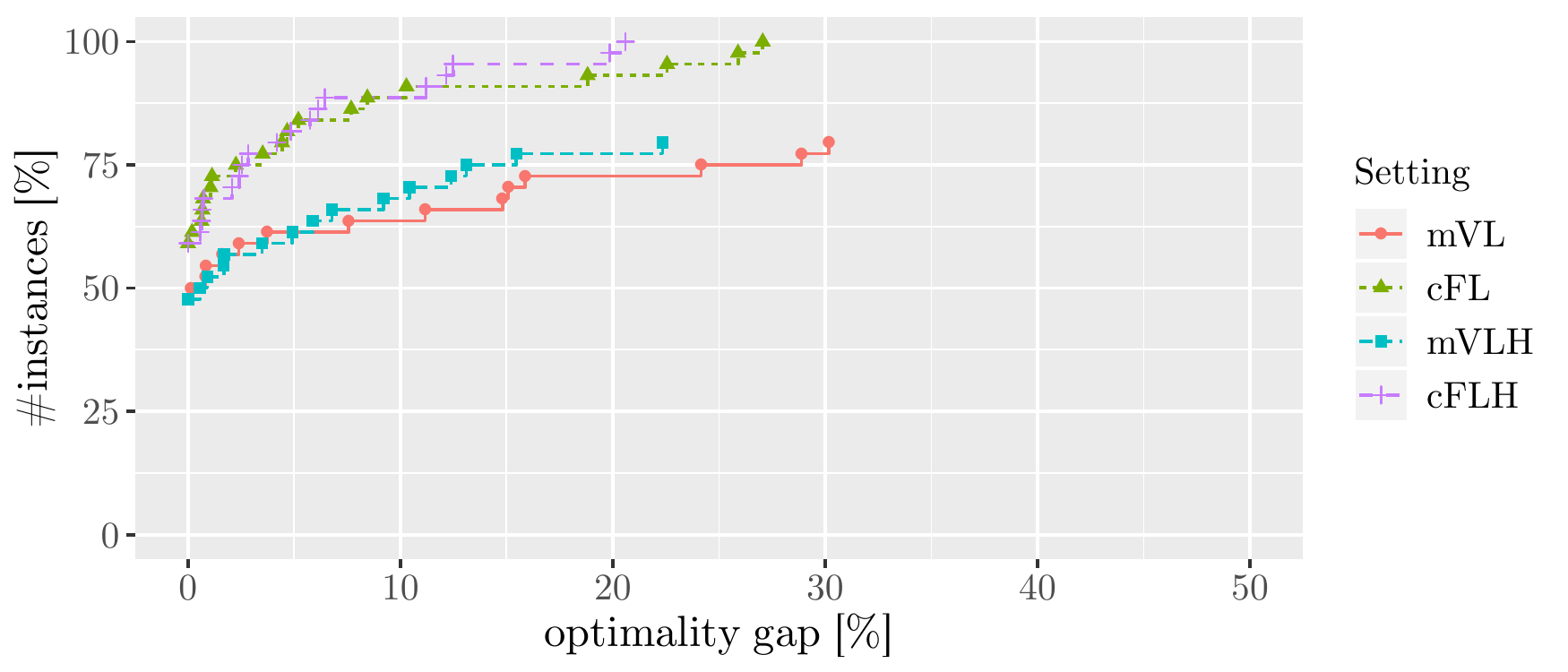}
	\caption{$p=10$, optimality gap}
\end{subfigure}
\caption{Plots for runtime and optimality gap for $p=2,3,5,10$ \blue{for the 
\texttt{TSPLIB} instances}. If 
\texttt{\#instances[\%]} in the optimality gap plot does not sum up to 100\%, 
this means that some runs did not finish due to memory issues. 
\label{fig:plotsa}}

\end{figure}

\begin{figure}[h!tb]

	\begin{subfigure}[b]{.49\linewidth}
		\centering
		\includegraphics[width=\textwidth]{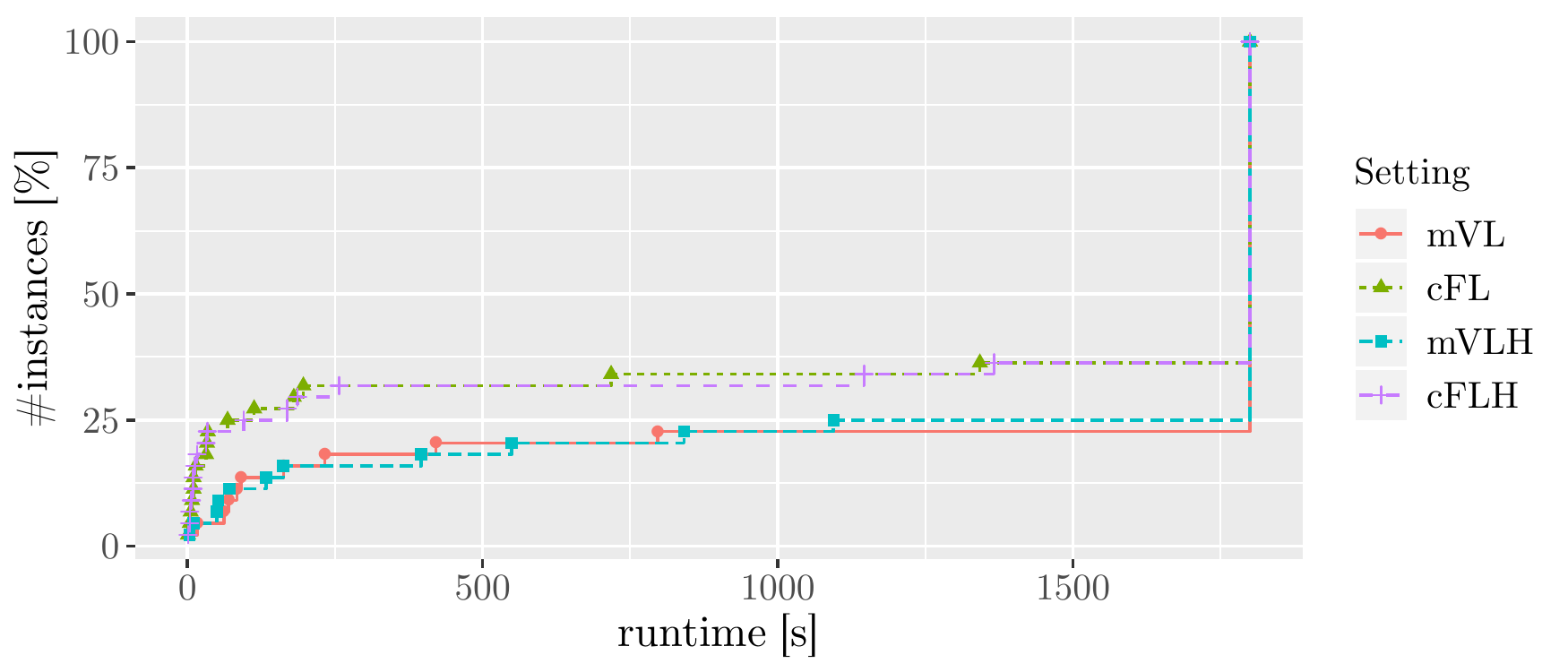}
		\caption{$p=15$, runtime}
	\end{subfigure}
	\begin{subfigure}[b]{.49\linewidth}
		\centering
		\includegraphics[width=\textwidth]{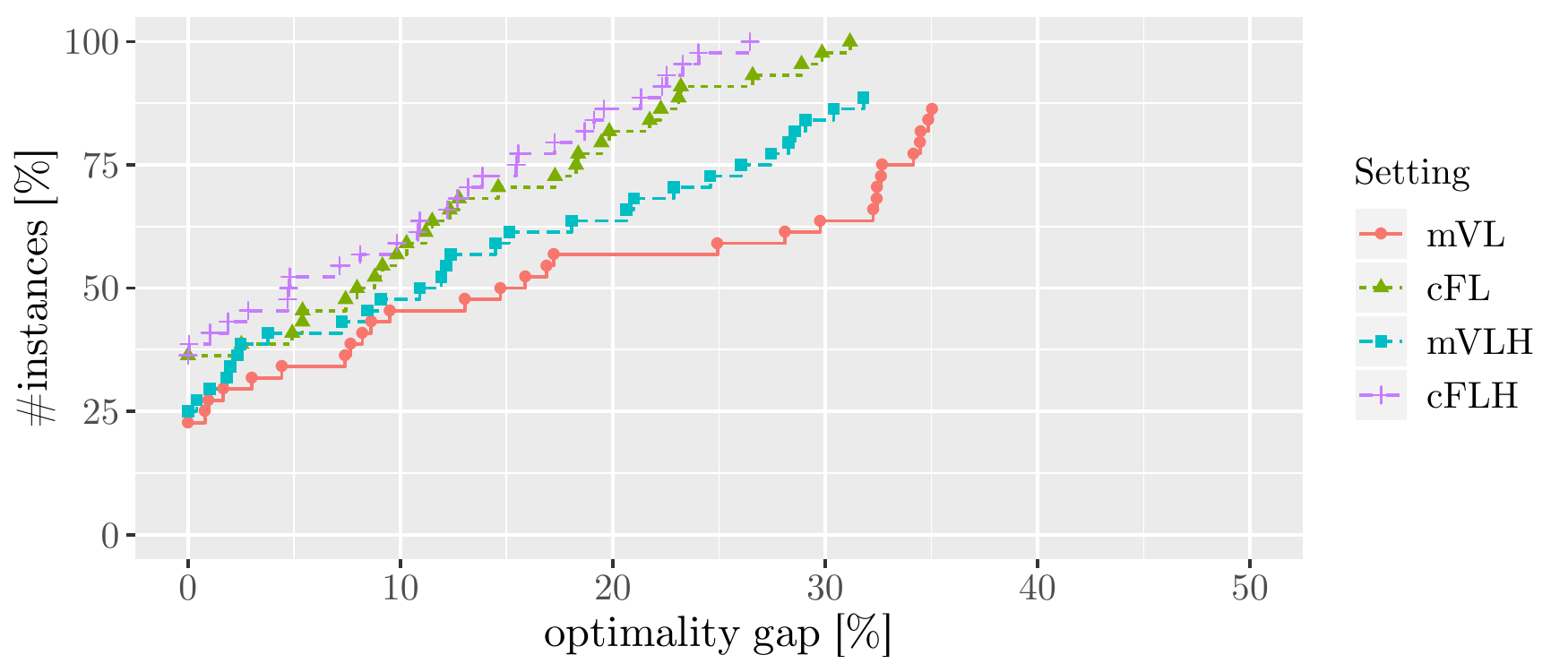}
		\caption{$p=15$, optimality gap}
	\end{subfigure}
	
	\begin{subfigure}[b]{.49\linewidth}
		\centering
		\includegraphics[width=\textwidth]{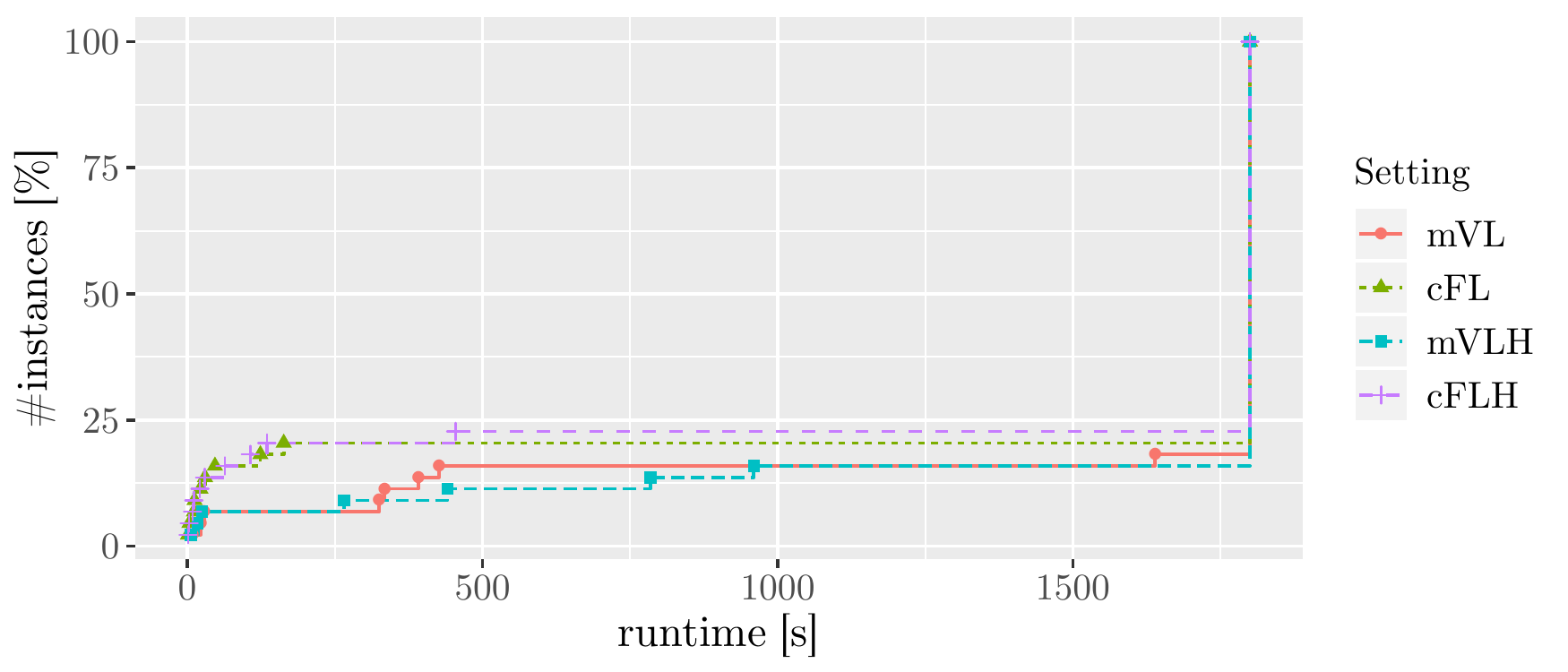}
		\caption{$p=20$, runtime}
	\end{subfigure}
	\begin{subfigure}[b]{.49\linewidth}
		\centering
		\includegraphics[width=\textwidth]{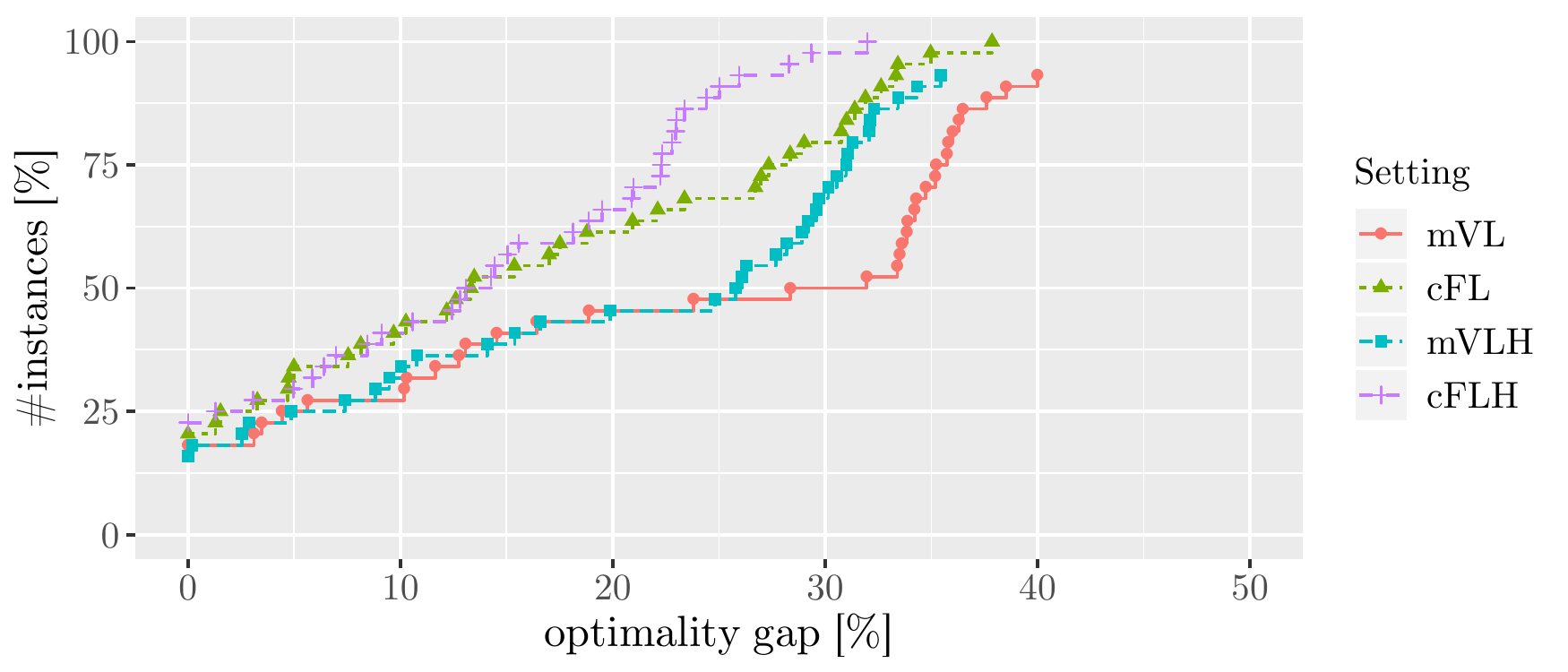}
		\caption{$p=20$, optimality gap}
	\end{subfigure}
	
	\begin{subfigure}[b]{.49\linewidth}
		\centering
		\includegraphics[width=\textwidth]{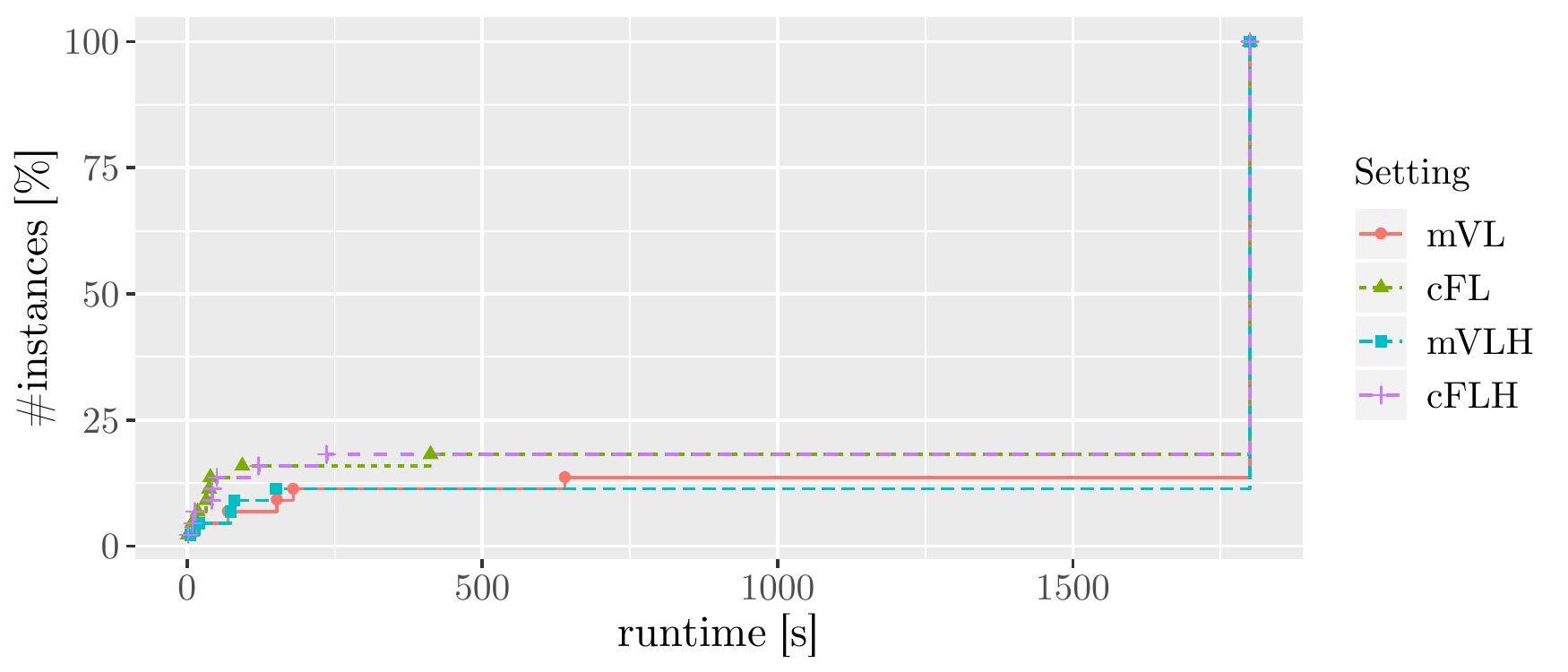}
		\caption{$p=25$, runtime}
	\end{subfigure}
	\begin{subfigure}[b]{.49\linewidth}
		\centering
		\includegraphics[width=\textwidth]{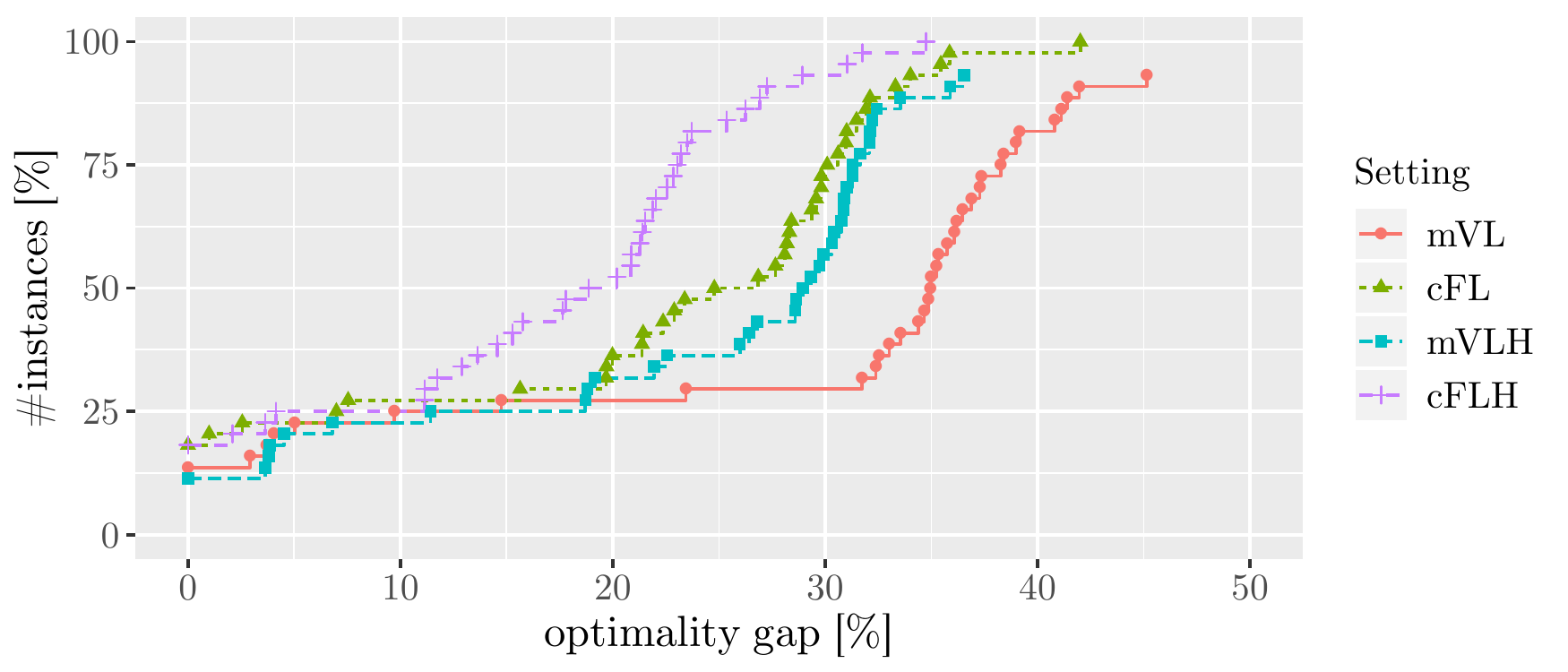}
		\caption{$p=25$, optimality gap}
	\end{subfigure}
	
	\begin{subfigure}[b]{.49\linewidth}
		\centering
		\includegraphics[width=\textwidth]{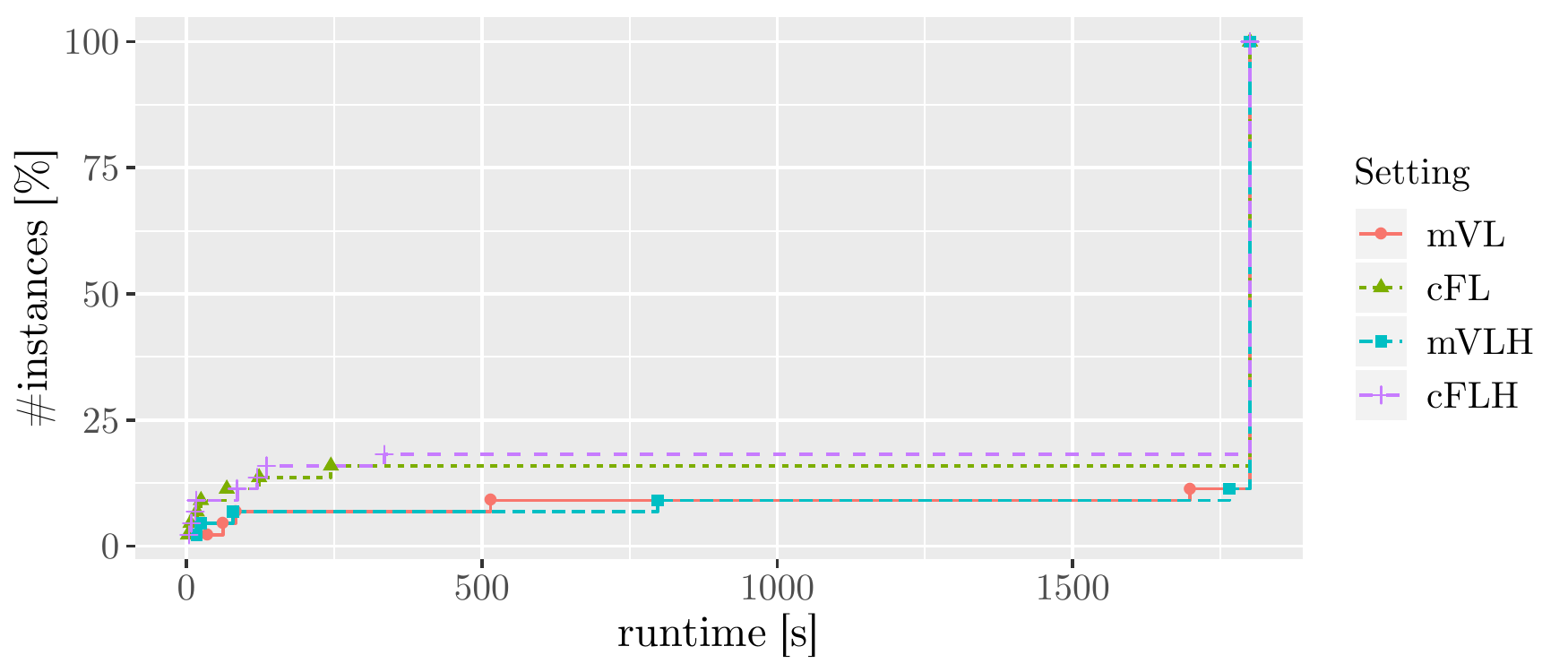}
		\caption{$p=30$, runtime}
	\end{subfigure}
	\begin{subfigure}[b]{.49\linewidth}
		\centering
		\includegraphics[width=\textwidth]{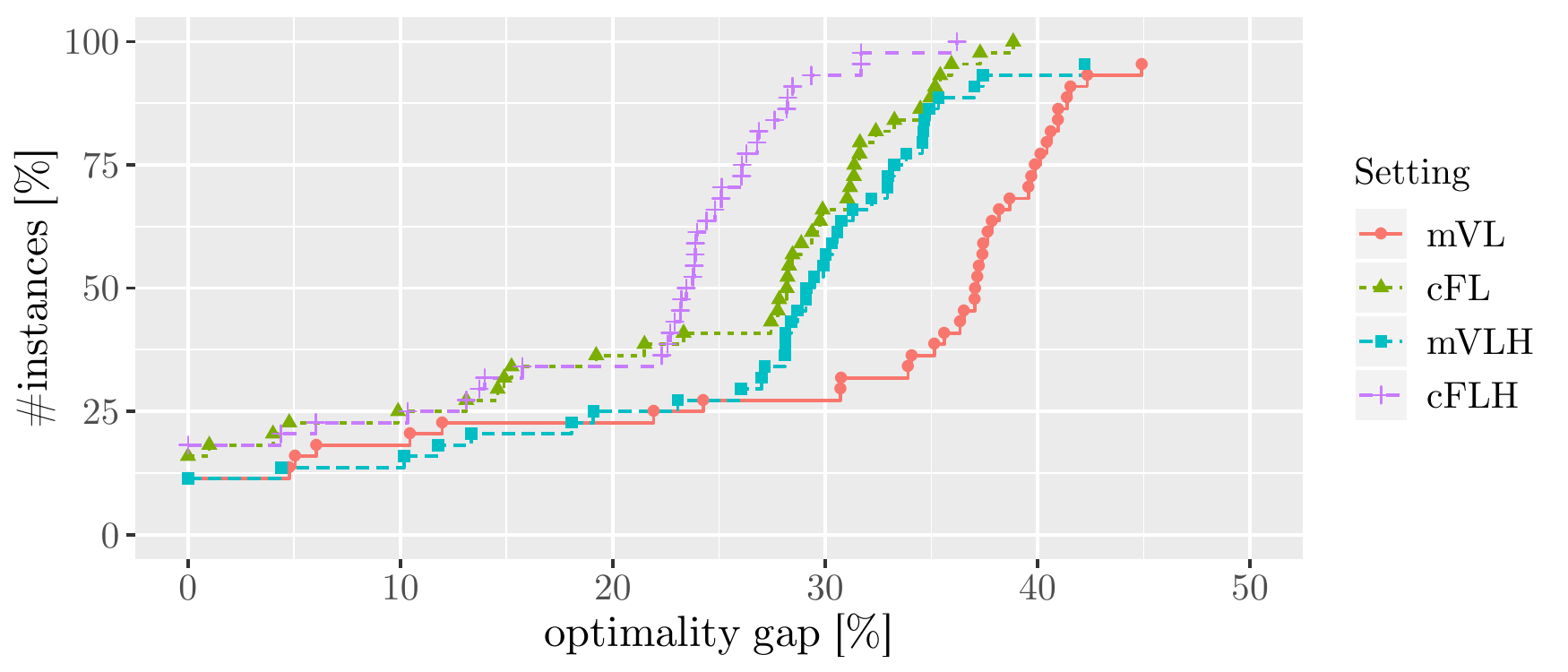}
		\caption{$p=30$, optimality gap}
	\end{subfigure}
	\caption{Plots for runtime and optimality gap for $p=15,20,25,30$ \blue{for 
	the \texttt{TSPLIB} instances}. If 
	\texttt{\#instances[\%]} in the optimality gap plot does not sum up to 
	100\%, this means that some runs did not finish due to memory issues. 
	\label{fig:plotsb}}
\end{figure}

\begin{figure}[h!tb]
		\centering
		\includegraphics[width=0.7\textwidth]{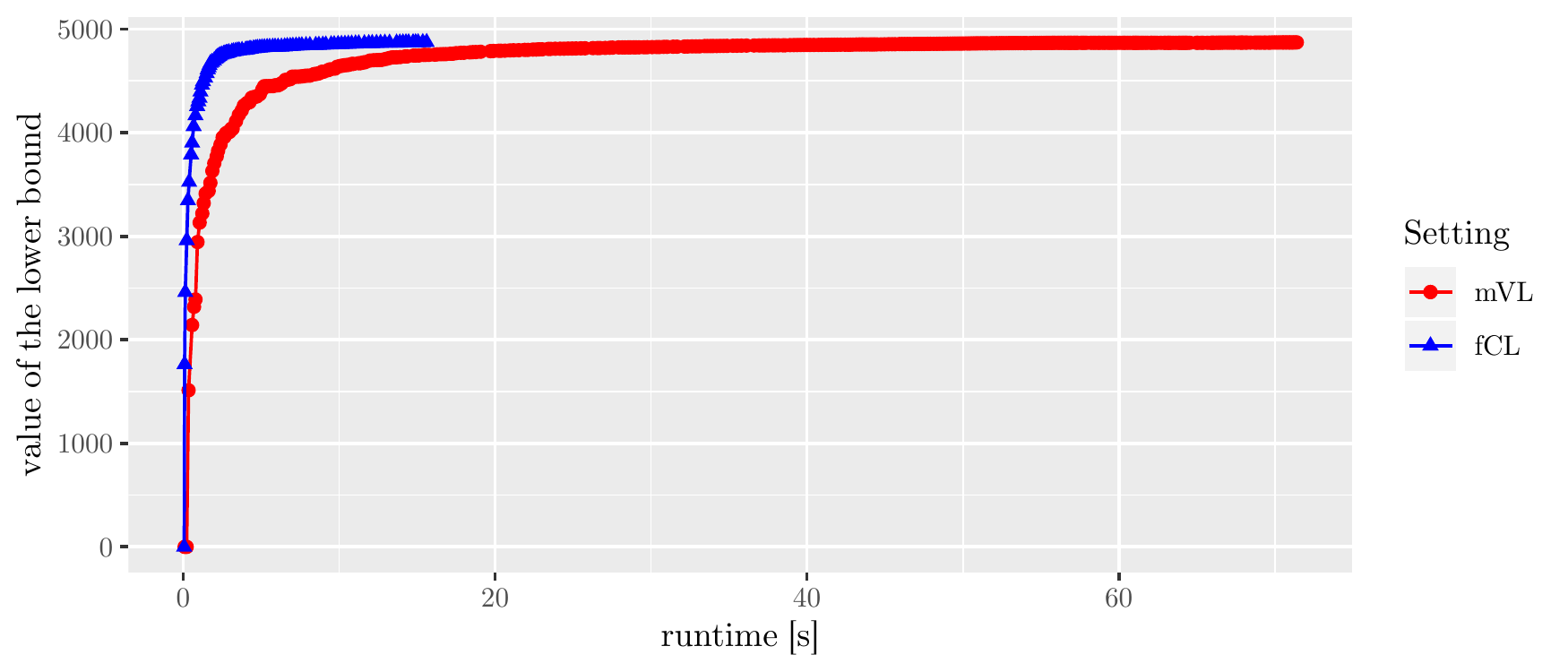}
		\caption{\blue{Behavior of the lower bound at the root node for the 
		instance 
		\texttt{rl11849} with $p=5$ and our two separation schemes} 
		\label{fig:bound}}
\end{figure}

\begin{figure}[h!tb]
	
	\begin{subfigure}[b]{.49\linewidth}
		\centering
		\includegraphics[width=\textwidth]{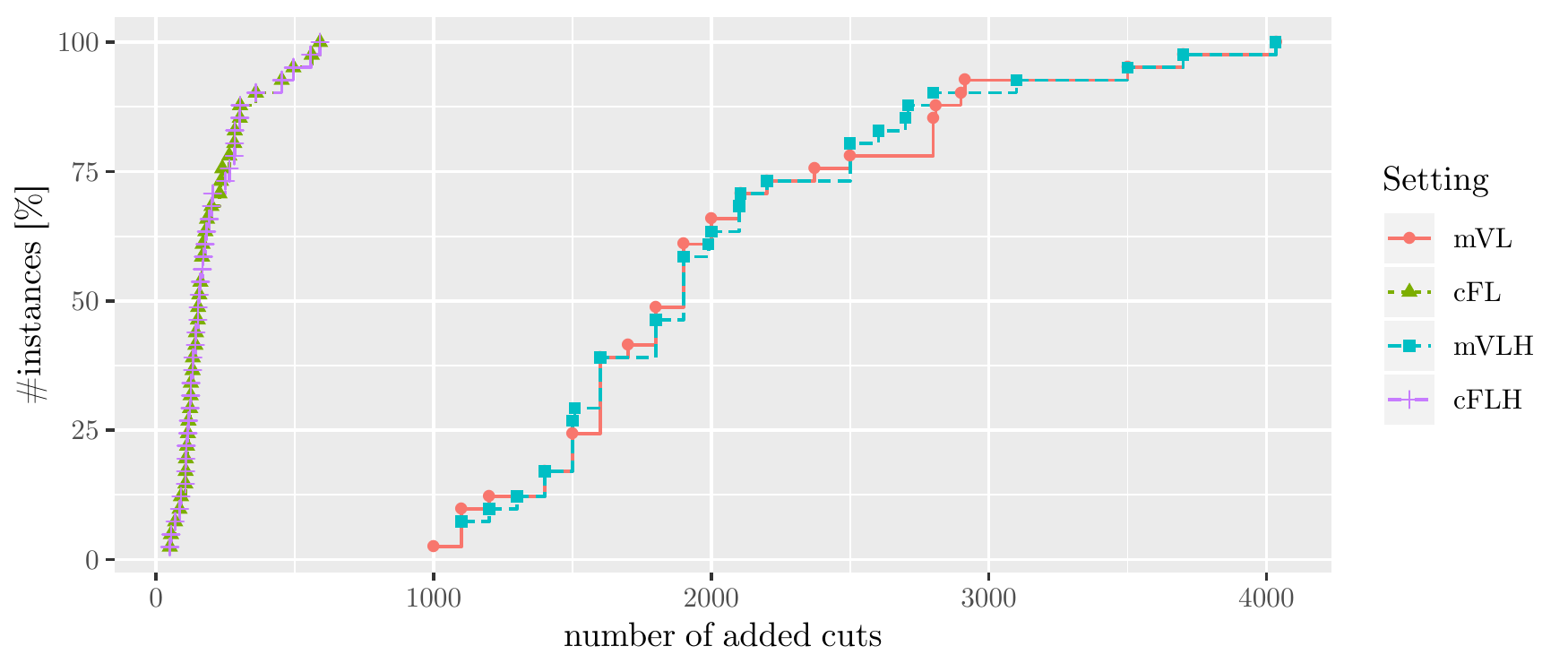}
		\caption{\blue{$p=2$, 41 instances \label{fig:numcutsa}}}
	\end{subfigure}
	\begin{subfigure}[b]{.49\linewidth}
		\centering
		\includegraphics[width=\textwidth]{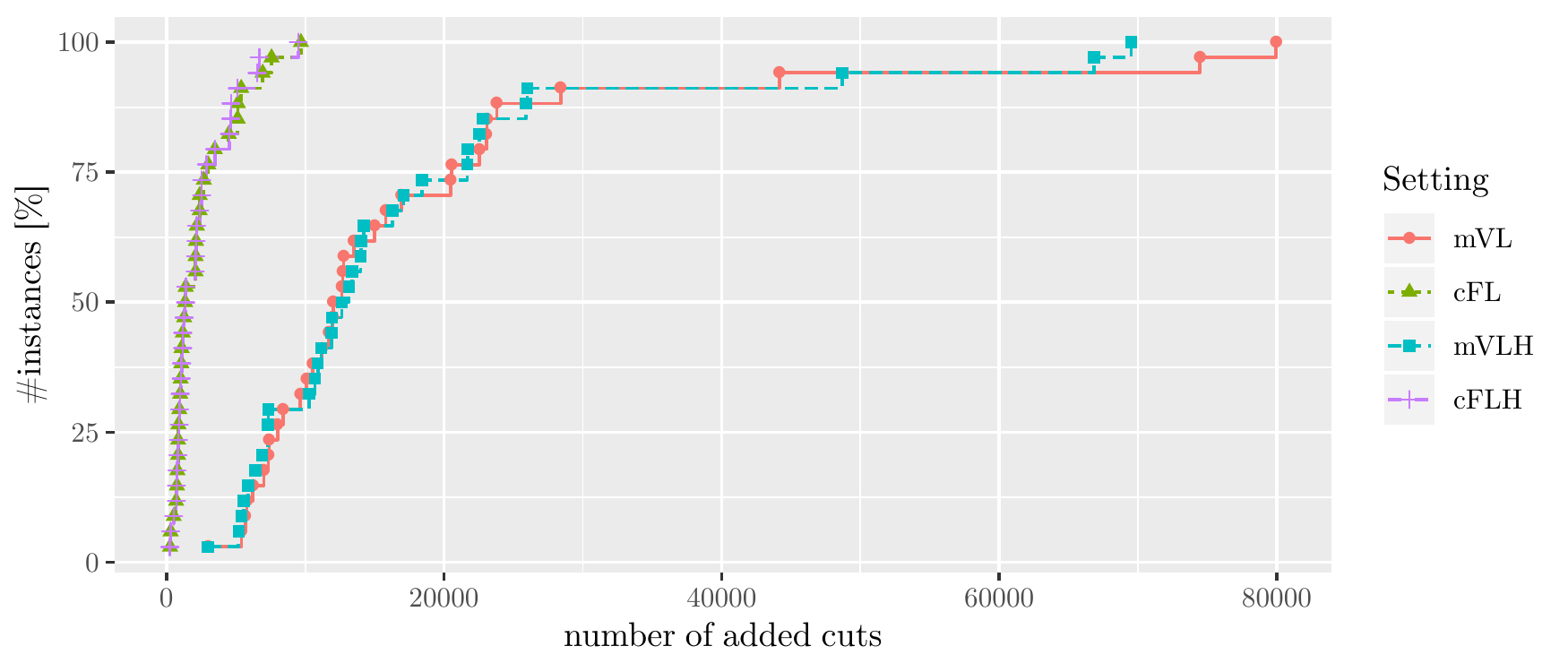}
		\caption{\blue{$p=5$, 34 instances \label{fig:numcutsb}}}
	\end{subfigure}

	\caption{\blue{Plots for the number of added inequalities 
	\eqref{eq:loptimality} for all \texttt{TSPLIB} instances which were  
	solved to optimality with all considered settings}}
\end{figure}

%\clearpage
\section{Conclusions} 
\label{sec:conclusion}

In this paper, we present a novel solution approach for the  the $p$-center problem, which is a fundamental problem in location science. The solution approach works on a new integer programming formulation that can be obtained by a projection from the classical 
formulation. The formulation is solved by means of branch-and-cut, where cuts 
for customer demand points are iteratively generated. This makes the method 
suitable for large scale instances. Moreover, the formulation can be 
strengthened with combinatorial information to obtain a much tighter 
LP-relaxation. In particular, we present a novel way to use lower bound 
information to obtain stronger cuts. We show that the obtained formulation has 
the same lower bound as the best known relaxations from the literature 
\citep{elloumi2004, calik2013double}. In contrast to these bounds \blue{from 
the literature}, which are based on semi-relaxations, i.e., some variables are 
kept 
integer in the relaxation, our relaxation only contains continuous variables. 
Moreover, our approach can be easily implemented in a single branch-and-cut 
algorithm, while existing state-of-the-art approaches for the $p$-center 
problem consist of iterative solution algorithms, where set cover problems are 
solved repeatedly \citep{chen2009,contardo2019scalable}. We describe different 
strategies on how to separate our lifted cuts, and also present a primal 
heuristic which uses the information provided by the linear programming 
relaxation. The solution algorithm is made available online.

In order to assess the efficiency of our proposed approach, we conducted a 
computational study on instances \blue{from the literature} with up to more than 700000 
customers and locations. The results show that for many instances, our solution 
algorithm is competitive with highly sophisticated solution frameworks 
\blue{from the literature}.

There are various avenues for further work. \blue{Further research regarding potential additional valid inequalities to strengthen the LP-relaxation could be interesting.} \blue{Another} direction could be the 
development of better primal heuristics to be used within the branch-and-cut 
algorithm, as for many instances our lower bound was quite good, while the 
primal bound still had some room for improvement. However, the large-scale 
nature of the considered instances may make this a challenging task, as even 
applying simple local search operator could become quite time consuming. A 
hybridization of our approach with existing set cover-based approaches could 
also be interesting, because the results show that for some instances our 
approach 
works better, while for other instances set-cover-based approaches work 
better. \blue{The investigation if preprocessing procedures from set-cover-based approaches can be transferred to our approach could also be a fruitful topic.}
It could also be worthwhile to investigate, if there are certain 
instance-characteristics which may favor a particular approach. Moreover, as 
our solution approach directly models the problem with optimality-type cuts, 
our approach could be more suitable when considering stochastic or robust 
versions of the \PCP, compared to the set cover-based approaches, which may no 
be easily adaptable to such settings.

\section*{Acknowledgments}

The research was supported by the Linz Institute of Technology (Project 
LIT-2019-7-YOU-211) and the JKU Business School.

\FloatBarrier

\ifArXiV
%\section*{References} %commented out to avoid "References" two times; 
%%commented 
%in again, as I was not getting them at all ;)
\bibliographystyle{elsarticle-harv}
\bibliography{biblio}

\else
%\section*{References} %commented out to avoid "References" two times; commented in again, as I was not getting them at all ;)

\fi

\clearpage

\appendix

\section{Detailed results for the instance set \texttt{TSPlib}}

In the following Tables \ref{ta:tsplib2}-\ref{ta:tsplib30} we provide a 
comparison with the results obtained by \cite{contardo2019scalable} using an 
Intel Xeon E5462 with 2.8 Ghz and 16GB RAM. In their runs, the 
timelimit was set to 24 hours (indicated by TL2 in the tables), while we have a 
timelimit of 30 minutes (indicated by TL in the tables).
Whenever at least one of the two versions finished, the faster runtime is 
bold. If both versions did not finish, the better upper and the better lower 
bound is bold.

\begin{table}[ht]
	\centering
	\caption{Detailed results for the \texttt{TSPlib} instances with 
	$p=2$\label{ta:tsplib2}} 
	\begingroup\footnotesize
	\begin{tabular}{lr|rrrrr|rrr}
		\toprule
		& & \multicolumn{5}{|c|}{$\fCLH$} & 
		\multicolumn{3}{|c}{\cite{contardo2019scalable}} \\ name & $|V|$  & LB 
		& UB & $g[\%]$  & \#BC & $t[s]$ & LB & UB & $t[s]$ \\ \midrule
		rw1621 & 1621 & 714 & 714 & 0.00 & 0 & \textbf{  0.11} & 714 & 714 &   9.50 \\ 
		u1817 & 1817 & 1061 & 1061 & 0.00 & 0 & \textbf{  0.20} & 1061 & 1061 &  12.10 \\ 
		rl1889 & 1889 & 6931 & 6931 & 0.00 & 0 & \textbf{  0.16} & 6931 & 6931 &  10.20 \\ 
		mu1979 & 1979 & 3284 & 3284 & 0.00 & 0 & \textbf{  0.09} & 3284 & 3284 &  10.50 \\ 
		pr2392 & 2392 & 6060 & 6060 & 0.00 & 0 & \textbf{  0.28} & 6060 & 6060 &  13.10 \\ 
		d15112-modif-2500 & 2500 & 9234 & 9234 & 0.00 & 0 & \textbf{  0.30} & 9234 & 9234 &  12.70 \\ 
		pcb3038 & 3038 & 1734 & 1734 & 0.00 & 0 & \textbf{  0.27} & 1734 & 1734 &  10.00 \\ 
		nu3496 & 3496 & 2140 & 2140 & 0.00 & 0 & \textbf{  0.21} & 2140 & 2140 &  10.20 \\ 
		ca4663 & 4663 & 25781 & 25781 & 0.00 & 0 & \textbf{  0.53} & 25781 & 25781 &  17.90 \\ 
		rl5915 & 5915 & 7385 & 7385 & 0.00 & 0 & \textbf{  0.52} & 7385 & 7385 &  17.00 \\ 
		rl5934 & 5934 & 7004 & 7004 & 0.00 & 0 & \textbf{  0.97} & 7004 & 7004 &  20.50 \\ 
		tz6117 & 6117 & 4582 & 4582 & 0.00 & 0 & \textbf{  0.85} & 4582 & 4582 &  15.80 \\ 
		eg7146 & 7146 & 4643 & 4643 & 0.00 & 0 & \textbf{  0.49} & 4643 & 4643 &  14.20 \\ 
		pla7397 & 7397 & 310664 & 310664 & 0.00 & 0 & \textbf{  0.43} & 310664 & 310664 &  15.30 \\ 
		ym7663 & 7663 & 3807 & 3807 & 0.00 & 0 & \textbf{  0.25} & 3807 & 3807 &  12.50 \\ 
		pm8079 & 8079 & 1686 & 1686 & 0.00 & 0 & \textbf{  0.45} & 1686 & 1686 &  13.40 \\ 
		ei8246 & 8246 & 1813 & 1813 & 0.00 & 0 & \textbf{  0.57} & 1813 & 1813 &  18.30 \\ 
		ar9152 & 9152 & 9725 & 9725 & 0.00 & 0 & \textbf{  0.68} & 9725 & 9725 &  20.50 \\ 
		ja9847 & 9847 & 10099 & 10099 & 0.00 & 0 & \textbf{  0.30} & 10099 & 10099 &  18.70 \\ 
		gr9882 & 9882 & 3757 & 3757 & 0.00 & 0 & \textbf{  0.71} & 3757 & 3757 &  13.70 \\ 
		kz9976 & 9976 & 11325 & 11325 & 0.00 & 0 & \textbf{  1.91} & 11325 & 11325 &  17.50 \\ 
		fi10639 & 10639 & 4922 & 4922 & 0.00 & 0 & \textbf{  0.79} & 4922 & 4922 &  16.70 \\ 
		rl11849 & 11849 & 7298 & 7298 & 0.00 & 0 & \textbf{  1.11} & 7298 & 7298 &  21.20 \\ 
		usa13509 & 13509 & 175750 & 175750 & 0.00 & 0 & \textbf{  0.86} & 175750 & 175750 &  17.90 \\ 
		brd14051 & 14051 & 2970 & 2970 & 0.00 & 0 & \textbf{  1.38} & 2970 & 2970 &  15.20 \\ 
		mo14185 & 14185 & 3698 & 3698 & 0.00 & 0 & \textbf{  1.27} & 3698 & 3698 &  16.20 \\ 
		ho14473 & 14473 & 1834 & 1834 & 0.00 & 0 & \textbf{  0.96} & 1834 & 1834 &  15.80 \\ 
		d15112 & 15112 & 9406 & 9406 & 0.00 & 0 & \textbf{  2.05} & 9406 & 9406 &  17.40 \\ 
		it16862 & 16862 & 4677 & 4677 & 0.00 & 0 & \textbf{  1.01} & 4677 & 4677 &  16.80 \\ 
		d18512 & 18512 & 3301 & 3301 & 0.00 & 0 & \textbf{  2.51} & 3301 & 3301 &  16.40 \\ 
		vm22775 & 22775 & 4266 & 4266 & 0.00 & 0 & \textbf{  1.59} & 4266 & 4266 &  20.50 \\ 
		sw24978 & 24978 & 4960 & 4960 & 0.00 & 0 & \textbf{  1.91} & 4960 & 4960 &  25.10 \\ 
		fyg28534 & 28534 & 448 & 448 & 0.00 & 0 & \textbf{  3.00} & 448 & 448 &  18.50 \\ 
		bm33708 & 33708 & 5167 & 5167 & 0.00 & 0 & \textbf{  2.92} & 5167 & 5167 &  21.60 \\ 
		pla33810 & 33810 & 330462 & 330462 & 0.00 & 0 & \textbf{  6.95} & 330462 & 330462 &  64.30 \\ 
		bby34656 & 34656 & 474 & 474 & 0.00 & 0 & \textbf{  2.99} & 474 & 474 &  15.10 \\ 
		pba38478 & 38478 & 469 & 469 & 0.00 & 0 & \textbf{  2.90} & 469 & 469 &  25.50 \\ 
		ch71009 & 71009 & 19988 & 19988 & 0.00 & 0 & \textbf{  5.63} & 19988 & 19988 &  27.10 \\ 
		pla85900 & 85900 & 436008 & 436008 & 0.00 & 0 & \textbf{ 21.56} & 436008 & 436008 & 136.10 \\ 
		sra104815 & 104814 & 908 & 908 & 0.00 & 0 & \textbf{ 16.21} & 908 & 908 &  51.10 \\ 
		usa115475 & 115475 & 17745 & 17745 & 0.00 & 0 & \textbf{ 10.22} & 17745 & 17745 &  47.90 \\ 
		ara238025 & 238025 & 1484 & 1484 & 0.00 & 0 & \textbf{ 39.25} & 1484 & 1484 &  89.80 \\ 
		lra498378 & 498378 & 5888 & 5888 & 0.00 & 0 & \textbf{ 81.50} & 5888 & 5888 & 423.70 \\ 
		lrb744710 & 744710 & 1930 & 1930 & 0.00 & 0 & \textbf{127.01} & 1930 & 1930 & 241.00 \\ 
		\bottomrule
	\end{tabular}
	\endgroup
\end{table}
% latex table generated in R 3.4.4 by xtable 1.8-3 package
% Fri Jun 25 13:30:55 2021
\begin{table}[ht]
	\centering
	\caption{Detailed results for the \texttt{TSPlib} instances with 
	$p=3$\label{ta:tsplib3}} 
	\begingroup\footnotesize
	\begin{tabular}{lr|rrrrr|rrr}
		\toprule
		& & \multicolumn{5}{|c|}{$\fCLH$} & 
		\multicolumn{3}{|c}{\cite{contardo2019scalable}} \\ name & $|V|$  & LB 
		& UB & $g[\%]$  & \#BC & $t[s]$ & LB & UB & $t[s]$ \\ \midrule
		rw1621 & 1621 & 624 & 624 & 0.00 & 0 & \textbf{  0.18} & 624 & 624 &  10.80 \\ 
		u1817 & 1817 & 895 & 895 & 0.00 & 0 & \textbf{  0.23} & 895 & 895 &   9.90 \\ 
		rl1889 & 1889 & 6066 & 6066 & 0.00 & 0 & \textbf{  0.35} & 6066 & 6066 &  12.30 \\ 
		mu1979 & 1979 & 2326 & 2326 & 0.00 & 0 & \textbf{  0.14} & 2326 & 2326 &  10.10 \\ 
		pr2392 & 2392 & 5413 & 5413 & 0.00 & 0 & \textbf{  1.36} & 5413 & 5413 &  17.10 \\ 
		d15112-modif-2500 & 2500 & 7952 & 7952 & 0.00 & 0 & \textbf{  0.38} & 7952 & 7952 &  13.10 \\ 
		pcb3038 & 3038 & 1519 & 1519 & 0.00 & 0 & \textbf{  0.61} & 1519 & 1519 &  12.40 \\ 
		nu3496 & 3496 & 1513 & 1513 & 0.00 & 0 & \textbf{  0.62} & 1513 & 1513 &  12.60 \\ 
		ca4663 & 4663 & 22680 & 22680 & 0.00 & 2 & \textbf{  1.56} & 22680 & 22680 &  20.70 \\ 
		rl5915 & 5915 & 6377 & 6377 & 0.00 & 0 & \textbf{  1.62} & 6377 & 6377 &  18.50 \\ 
		rl5934 & 5934 & 6005 & 6005 & 0.00 & 0 & \textbf{  1.29} & 6005 & 6005 &  16.80 \\ 
		tz6117 & 6117 & 4171 & 4171 & 0.00 & 0 & \textbf{  3.06} & 4171 & 4171 &  21.60 \\ 
		eg7146 & 7146 & 3702 & 3702 & 0.00 & 0 & \textbf{  0.55} & 3702 & 3702 &  15.30 \\ 
		pla7397 & 7397 & 279242 & 279242 & 0.00 & 0 & \textbf{  2.36} & 279242 & 279242 &  32.50 \\ 
		ym7663 & 7663 & 3107 & 3107 & 0.00 & 0 & \textbf{  0.67} & 3107 & 3107 &  17.10 \\ 
		pm8079 & 8079 & 1382 & 1382 & 0.00 & 0 & \textbf{  0.62} & 1382 & 1382 &  14.50 \\ 
		ei8246 & 8246 & 1522 & 1522 & 0.00 & 0 & \textbf{  1.34} & 1522 & 1522 &  15.90 \\ 
		ar9152 & 9152 & 8196 & 8196 & 0.00 & 0 & \textbf{  1.23} & 8196 & 8196 &  24.80 \\ 
		ja9847 & 9847 & 7872 & 7872 & 0.00 & 0 & \textbf{  0.64} & 7872 & 7872 &  13.50 \\ 
		gr9882 & 9882 & 3170 & 3170 & 0.00 & 0 & \textbf{  1.11} & 3170 & 3170 &  14.90 \\ 
		kz9976 & 9976 & 8244 & 8244 & 0.00 & 0 & \textbf{  1.02} & 8244 & 8244 &  16.10 \\ 
		fi10639 & 10639 & 3742 & 3742 & 0.00 & 0 & \textbf{  1.69} & 3742 & 3742 &  13.50 \\ 
		rl11849 & 11849 & 6452 & 6452 & 0.00 & 0 & \textbf{  5.10} & 6452 & 6452 &  24.20 \\ 
		usa13509 & 13509 & 134489 & 134489 & 0.00 & 0 & \textbf{  1.18} & 134489 & 134489 &  23.50 \\ 
		brd14051 & 14051 & 2426 & 2426 & 0.00 & 0 & \textbf{  2.14} & 2426 & 2426 &  17.90 \\ 
		mo14185 & 14185 & 2992 & 2992 & 0.00 & 0 & \textbf{  1.71} & 2992 & 2992 &  16.70 \\ 
		ho14473 & 14473 & 1534 & 1534 & 0.00 & 0 & \textbf{  1.81} & 1534 & 1534 &  19.00 \\ 
		d15112 & 15112 & 8154 & 8154 & 0.00 & 0 & \textbf{  5.08} & 8154 & 8154 &  32.60 \\ 
		it16862 & 16862 & 3724 & 3724 & 0.00 & 0 & \textbf{  1.09} & 3724 & 3724 &  20.50 \\ 
		d18512 & 18512 & 2914 & 2914 & 0.00 & 0 & \textbf{  6.41} & 2914 & 2914 &  30.20 \\ 
		vm22775 & 22775 & 3135 & 3135 & 0.00 & 0 & \textbf{  2.86} & 3135 & 3135 &  21.80 \\ 
		sw24978 & 24978 & 4120 & 4120 & 0.00 & 0 & \textbf{  2.99} & 4120 & 4120 &  24.20 \\ 
		fyg28534 & 28534 & 400 & 400 & 0.00 & 0 & \textbf{ 10.90} & 400 & 400 &  34.60 \\ 
		bm33708 & 33708 & 4213 & 4213 & 0.00 & 0 & \textbf{  7.25} & 4213 & 4213 &  35.10 \\ 
		pla33810 & 33810 & 302160 & 302160 & 0.00 & 0 & \textbf{ 31.67} & 302160 & 302160 & 154.70 \\ 
		bby34656 & 34656 & 420 & 420 & 0.00 & 0 & \textbf{ 11.43} & 420 & 420 &  43.30 \\ 
		pba38478 & 38478 & 413 & 413 & 0.00 & 0 & \textbf{  9.88} & 413 & 413 &  33.20 \\ 
		ch71009 & 71009 & 13292 & 13292 & 0.00 & 0 & \textbf{ 13.96} & 13292 & 13292 &  59.90 \\ 
		pla85900 & 85900 & 399677 & 399677 & 0.00 & 0 & \textbf{127.69} & 399677 & 399677 & 577.30 \\ 
		sra104815 & 104814 & 688 & 688 & 0.00 & 0 & \textbf{ 35.93} & 688 & 688 &  97.70 \\ 
		usa115475 & 115475 & 13530 & 13530 & 0.00 & 0 & \textbf{ 13.33} & 13530 & 13530 &  64.30 \\ 
		ara238025 & 238025 & 1166 & 1166 & 0.00 & 0 & \textbf{ 53.05} & 1166 & 1166 & 153.30 \\ 
		lra498378 & 498378 & 4995 & 4995 & 0.00 & 0 & \textbf{260.76} & 4995 & 4995 & 689.40 \\ 
		lrb744710 & 744710 & 1702 & 1702 & 0.00 & 0 & \textbf{284.90} & 1702 & 1702 & 567.10 \\ 
		\bottomrule
	\end{tabular}
	\endgroup
\end{table}
% latex table generated in R 3.4.4 by xtable 1.8-3 package
% Fri Jun 25 13:30:55 2021
\begin{table}[ht]
	\centering
	\caption{Detailed results for the \texttt{TSPlib} instances with 
	$p=5$\label{ta:tsplib5}} 
	\begingroup\footnotesize
	\begin{tabular}{lr|rrrrr|rrr}
		\toprule
		& & \multicolumn{5}{|c|}{$\fCLH$} & 
		\multicolumn{3}{|c}{\cite{contardo2019scalable}} \\ name & $|V|$  & LB 
		& UB & $g[\%]$  & \#BC & $t[s]$ & LB & UB & $t[s]$ \\ \midrule
		rw1621 & 1621 & 414 & 414 & 0.00 & 0 & \textbf{     0.33} & 414 & 414 &   12.50 \\ 
		u1817 & 1817 & 715 & 715 & 0.00 & 0 & \textbf{     1.22} & 715 & 715 &   15.90 \\ 
		rl1889 & 1889 & 4792 & 4792 & 0.00 & 0 & \textbf{     1.74} & 4792 & 4792 &   16.30 \\ 
		mu1979 & 1979 & 1877 & 1877 & 0.00 & 0 & \textbf{     0.56} & 1877 & 1877 &   14.90 \\ 
		pr2392 & 2392 & 3827 & 3827 & 0.00 & 0 & \textbf{     4.51} & 3827 & 3827 &   33.90 \\ 
		d15112-modif-2500 & 2500 & 5856 & 5856 & 0.00 & 0 & \textbf{     3.44} & 5856 & 5856 &   16.90 \\ 
		pcb3038 & 3038 & 1064 & 1064 & 0.00 & 0 & \textbf{     1.79} & 1064 & 1064 &   21.20 \\ 
		nu3496 & 3496 & 1123 & 1123 & 0.00 & 0 & \textbf{     1.11} & 1123 & 1123 &   14.60 \\ 
		ca4663 & 4663 & 16837 & 16837 & 0.00 & 0 & \textbf{     1.51} & 16837 & 16837 &   16.10 \\ 
		rl5915 & 5915 & 4554 & 4554 & 0.00 & 0 & \textbf{     5.30} & 4554 & 4554 &   27.00 \\ 
		rl5934 & 5934 & 4792 & 4792 & 0.00 & 0 & \textbf{     4.85} & 4792 & 4792 &   33.40 \\ 
		tz6117 & 6117 & 2918 & 2918 & 0.00 & 0 & \textbf{     8.45} & 2918 & 2918 &   37.10 \\ 
		eg7146 & 7146 & 2590 & 2590 & 0.00 & 0 & \textbf{     2.27} & 2590 & 2590 &   21.90 \\ 
		pla7397 & 7397 & 174542 & 174542 & 0.00 & 0 & \textbf{     2.22} & 174542 & 174542 &   32.70 \\ 
		ym7663 & 7663 & 2043 & 2043 & 0.00 & 0 & \textbf{     0.95} & 2043 & 2043 &   14.00 \\ 
		pm8079 & 8079 & 938 & 938 & 0.00 & 0 & \textbf{     2.04} & 938 & 938 &   18.40 \\ 
		ei8246 & 8246 & 1042 & 1042 & 0.00 & 0 & \textbf{     2.71} & 1042 & 1042 &   18.00 \\ 
		ar9152 & 9152 & 6752 & 6752 & 0.00 & 0 & \textbf{     3.29} & 6752 & 6752 &   21.40 \\ 
		ja9847 & 9847 & 4503 & 4503 & 0.00 & 0 & \textbf{     0.98} & 4503 & 4503 &   18.60 \\ 
		gr9882 & 9882 & 2151 & 2151 & 0.00 & 0 & \textbf{     6.48} & 2151 & 2151 &   37.00 \\ 
		kz9976 & 9976 & 5995 & 5995 & 0.00 & 0 & \textbf{     4.70} & 5995 & 5995 &   29.70 \\ 
		fi10639 & 10639 & 2739 & 2739 & 0.00 & 0 & \textbf{     6.88} & 2739 & 2739 &   23.30 \\ 
		rl11849 & 11849 & 4873 & 4873 & 0.00 & 0 & \textbf{    20.29} & 4873 & 4873 &   53.70 \\ 
		usa13509 & 13509 & 103671 & 103671 & 0.00 & 0 & \textbf{     9.52} & 103671 & 103671 &   53.40 \\ 
		brd14051 & 14051 & 1822 & 1822 & 0.00 & 0 & \textbf{     4.61} & 1822 & 1822 &   22.70 \\ 
		mo14185 & 14185 & 2143 & 2143 & 0.00 & 0 & \textbf{     4.86} & 2143 & 2143 &   30.30 \\ 
		ho14473 & 14473 & 1112 & 1112 & 0.00 & 0 & \textbf{     4.39} & 1112 & 1112 &   20.70 \\ 
		d15112 & 15112 & 5890 & 5890 & 0.00 & 0 & \textbf{    42.21} & 5890 & 5890 &   77.40 \\ 
		it16862 & 16862 & 2811 & 2811 & 0.00 & 0 & \textbf{     3.13} & 2811 & 2811 &   33.10 \\ 
		d18512 & 18512 & 2073 & 2073 & 0.00 & 0 & \textbf{    65.95} & 2073 & 2073 &   68.40 \\ 
		vm22775 & 22775 & 2269 & 2269 & 0.00 & 0 & \textbf{     6.20} & 2269 & 2269 &   41.70 \\ 
		sw24978 & 24978 & 3022 & 3022 & 0.00 & 0 & \textbf{    14.66} & 3022 & 3022 &   45.90 \\ 
		fyg28534 & 28534 & 261 & 261 & 0.00 & 0 & \textbf{    64.21} & 261 & 261 &  164.00 \\ 
		bm33708 & 33708 & 2747 & 2747 & 0.00 & 0 & \textbf{     9.95} & 2747 & 2747 &   53.80 \\ 
		pla33810 & 33810 & 203597 & 203597 & 0.00 & 0 & \textbf{   179.93} & 203597 & 203597 &  769.50 \\ 
		bby34656 & 34656 & 286 & 286 & 0.00 & 0 & \textbf{   118.60} & 286 & 286 &  256.70 \\ 
		pba38478 & 38478 & 313 & 313 & 0.00 & 0 & \textbf{    71.79} & 313 & 313 &  174.10 \\ 
		ch71009 & 71009 & 10944 & 10944 & 0.00 & 11 & \textbf{   200.89} & 10944 & 10944 &  269.30 \\ 
		pla85900 & 85900 & 269544 & 269544 & 0.00 & 0 & \textbf{  1030.73} & 269544 & 269544 & 3212.00 \\ 
		sra104815 & 104814 & 508 & 508 & 0.00 & 0 & \textbf{   115.84} & 508 & 508 &  237.20 \\ 
		usa115475 & 115475 & 10414 & 10414 & 0.00 & 0 & \textbf{   196.31} & 10414 & 10414 &  562.20 \\ 
		ara238025 & 238025 & 855 & 855 & 0.00 & 5 & \textbf{   406.24} & 855 & 855 &  578.20 \\ 
		lra498378 & 498378 & 3259 & 3259 & 0.00 & 0 & \textbf{   561.74} & 3259 & 3259 & 2250.30 \\ 
		lrb744710 & 744710 & 1167 & 1176 & 0.77 & 0 & TL & 1170 & 1170 & \textbf{4781.70} \\ 
		\bottomrule
	\end{tabular}
	\endgroup
\end{table}
% latex table generated in R 3.4.4 by xtable 1.8-3 package
% Fri Jun 25 13:30:55 2021
\begin{table}[ht]
	\centering
	\caption{Detailed results for the \texttt{TSPlib} instances with 
	$p=10$\label{ta:tsplib10}} 
	\begingroup\footnotesize
	\begin{tabular}{lr|rrrrr|rrr}
		\toprule
		& & \multicolumn{5}{|c|}{$\fCLH$} & 
		\multicolumn{3}{|c}{\cite{contardo2019scalable}} \\ name & $|V|$  & LB 
		& UB & $g[\%]$  & \#BC & $t[s]$ & LB & UB & $t[s]$ \\ \midrule
		rw1621 & 1621 & 285 & 285 & 0.00 & 0 & \textbf{     1.21} & 285 & 285 &     15.80 \\ 
		u1817 & 1817 & 458 & 458 & 0.00 & 6 & \textbf{    16.71} & 458 & 458 &     42.30 \\ 
		rl1889 & 1889 & 3078 & 3101 & 0.74 & 584 & TL & 3101 & 3101 & \textbf{   151.80} \\ 
		mu1979 & 1979 & 1161 & 1161 & 0.00 & 0 & \textbf{     0.61} & 1161 & 1161 &     14.60 \\ 
		pr2392 & 2392 & 2581 & 2581 & 0.00 & 550 &   1001.96 & 2581 & 2581 & \textbf{   166.10} \\ 
		d15112-modif-2500 & 2500 & 3705 & 3705 & 0.00 & 45 &     49.04 & 3705 & 3705 & \textbf{    45.90} \\ 
		pcb3038 & 3038 & 729 & 729 & 0.00 & 169 & \textbf{   133.16} & 729 & 729 &    175.40 \\ 
		nu3496 & 3496 & 757 & 757 & 0.00 & 0 & \textbf{     4.59} & 757 & 757 &     20.70 \\ 
		ca4663 & 4663 & 10499 & 10499 & 0.00 & 0 & \textbf{     2.15} & 10499 & 10499 &     27.00 \\ 
		rl5915 & 5915 & 3104 & 3262 & 4.84 & 52 & TL & 3137 & 3137 & \textbf{   361.80} \\ 
		rl5934 & 5934 & 3078 & 3143 & 2.07 & 105 & TL & 3092 & 3092 & \textbf{   337.50} \\ 
		tz6117 & 6117 & 1902 & 1902 & 0.00 & 0 & \textbf{    89.14} & 1902 & 1902 &    180.30 \\ 
		eg7146 & 7146 & 1826 & 1826 & 0.00 & 0 & \textbf{     3.35} & 1826 & 1826 &     35.30 \\ 
		pla7397 & 7397 & 121968 & 121968 & 0.00 & 0 & \textbf{    21.07} & 121968 & 121968 &    104.80 \\ 
		ym7663 & 7663 & 1447 & 1447 & 0.00 & 0 & \textbf{     5.87} & 1447 & 1447 &     36.60 \\ 
		pm8079 & 8079 & 653 & 653 & 0.00 & 0 & \textbf{     6.36} & 653 & 653 &     27.40 \\ 
		ei8246 & 8246 & 722 & 722 & 0.00 & 5 & \textbf{   114.22} & 722 & 722 &    225.20 \\ 
		ar9152 & 9152 & 4273 & 4273 & 0.00 & 54 &    218.28 & 4273 & 4273 & \textbf{   117.70} \\ 
		ja9847 & 9847 & 2766 & 2766 & 0.00 & 0 & \textbf{     3.48} & 2766 & 2766 &     28.50 \\ 
		gr9882 & 9882 & 1372 & 1372 & 0.00 & 0 & \textbf{     8.71} & 1372 & 1372 &     54.40 \\ 
		kz9976 & 9976 & 4155 & 4155 & 0.00 & 0 & \textbf{    13.12} & 4155 & 4155 &     61.80 \\ 
		fi10639 & 10639 & 1855 & 1855 & 0.00 & 0 & \textbf{    37.92} & 1855 & 1855 &     97.40 \\ 
		rl11849 & 11849 & 3158 & 3179 & 0.66 & 31 & TL & 3164 & 3164 & \textbf{   385.00} \\ 
		usa13509 & 13509 & 67075 & 67075 & 0.00 & 30 & \textbf{   465.62} & 67075 & 67075 &    503.80 \\ 
		brd14051 & 14051 & 1257 & 1265 & 0.63 & 30 & TL & 1265 & 1265 & \textbf{   181.20} \\ 
		mo14185 & 14185 & 1405 & 1405 & 0.00 & 0 & \textbf{    65.60} & 1405 & 1405 &    211.60 \\ 
		ho14473 & 14473 & 738 & 738 & 0.00 & 6 &    140.97 & 738 & 738 & \textbf{   119.50} \\ 
		d15112 & 15112 & 3760 & 4005 & 6.12 & 5 & TL & 3785 & 3785 & \textbf{   599.70} \\ 
		it16862 & 16862 & 1574 & 1574 & 0.00 & 0 & \textbf{    16.45} & 1574 & 1574 &     87.80 \\ 
		d18512 & 18512 & 1332 & 1371 & 2.84 & 94 & TL & 1340 & 1340 & \textbf{  2060.10} \\ 
		vm22775 & 22775 & 1315 & 1315 & 0.00 & 0 & \textbf{    31.50} & 1315 & 1315 &    134.30 \\ 
		sw24978 & 24978 & 2061 & 2061 & 0.00 & 0 & \textbf{   103.94} & 2061 & 2061 &    155.90 \\ 
		fyg28534 & 28534 & 175 & 176 & 0.57 & 36 & TL & 176 & 176 & \textbf{  3137.30} \\ 
		bm33708 & 33708 & 1907 & 1907 & 0.00 & 0 &   1063.64 & 1907 & 1907 & \textbf{  1015.40} \\ 
		pla33810 & 33810 & 135441 & 168992 & 19.85 & 0 & TL & 136517 & 136517 & \textbf{  6825.70} \\ 
		bby34656 & 34656 & 192 & 197 & 2.54 & 2 & TL & 192 & 192 & \textbf{  3449.20} \\ 
		pba38478 & 38478 & 206 & 232 & 11.21 & 2 & TL & 207 & 207 & \textbf{  6831.70} \\ 
		ch71009 & 71009 & 7162 & 7599 & 5.75 & 0 & TL & 7183 & 7183 & \textbf{  1595.20} \\ 
		pla85900 & 85900 & 176333 & 200715 & 12.15 & 0 & TL & 180497 & 180497 & \textbf{ 19575.70} \\ 
		sra104815 & 104814 & 349 & 373 & 6.43 & 0 & TL & 354 & 354 & \textbf{  7059.50} \\ 
		usa115475 & 115475 & 6700 & 6866 & 2.42 & 0 & TL & 6736 & 6736 & \textbf{  4784.00} \\ 
		ara238025 & 238025 & 549 & 573 & 4.19 & 0 & TL & 552 & 552 & \textbf{  7953.80} \\ 
		lra498378 & 498378 & 2160 & 2468 & 12.48 & 0 & TL & 2232 & 2232 & \textbf{ 16623.10} \\ 
		lrb744710 & 744710 & 748 & 942 & 20.59 & 0 & TL & \textbf{795} & \textbf{810} & TL2 \\ 
		\bottomrule
	\end{tabular}
	\endgroup
\end{table}
% latex table generated in R 3.4.4 by xtable 1.8-3 package
% Fri Jun 25 13:30:55 2021
\begin{table}[ht]
	\centering
	\caption{Detailed results for the \texttt{TSPlib} instances with 
	$p=15$\label{ta:tsplib15}} 
	\begingroup\footnotesize
	\begin{tabular}{lr|rrrrr|rrr}
		\toprule
		& & \multicolumn{5}{|c|}{$\fCLH$} & 
		\multicolumn{3}{|c}{\cite{contardo2019scalable}} \\ name & $|V|$  & LB 
		& UB & $g[\%]$  & \#BC & $t[s]$ & LB & UB & $t[s]$ \\ \midrule
		rw1621 & 1621 & 217 & 217 & 0.00 & 8 & \textbf{     6.62} & 217 & 217 &     27.50 \\ 
		u1817 & 1817 & 359 & 359 & 0.00 & 737 &   1366.66 & 359 & 359 & \textbf{   386.20} \\ 
		rl1889 & 1889 & 2384 & 2384 & 0.00 & 264 &   1146.25 & 2384 & 2384 & \textbf{   181.60} \\ 
		mu1979 & 1979 & 868 & 868 & 0.00 & 0 & \textbf{     0.92} & 868 & 868 &     17.50 \\ 
		pr2392 & 2392 & 2011 & 2110 & 4.69 & 394 & TL & 2039 & 2039 & \textbf{  1581.30} \\ 
		d15112-modif-2500 & 2500 & 2972 & 2972 & 0.00 & 30 & \textbf{    95.50} & 2972 & 2972 &    110.70 \\ 
		pcb3038 & 3038 & 567 & 701 & 19.12 & 276 & TL & 578 & 578 & \textbf{  3474.90} \\ 
		nu3496 & 3496 & 604 & 604 & 0.00 & 12 & \textbf{    31.38} & 604 & 604 &     54.40 \\ 
		ca4663 & 4663 & 8296 & 8296 & 0.00 & 0 & \textbf{     3.80} & 8296 & 8296 &     36.20 \\ 
		rl5915 & 5915 & 2403 & 2768 & 13.19 & 186 & TL & 2441 & 2441 & \textbf{  6266.20} \\ 
		rl5934 & 5934 & 2392 & 2393 & 0.04 & 83 & TL & 2393 & 2393 & \textbf{   821.20} \\ 
		tz6117 & 6117 & 1508 & 1727 & 12.68 & 459 & TL & 1528 & 1528 & \textbf{   764.70} \\ 
		eg7146 & 7146 & 1308 & 1308 & 0.00 & 0 & \textbf{     3.81} & 1308 & 1308 &     44.70 \\ 
		pla7397 & 7397 & 95270 & 95270 & 0.00 & 0 & \textbf{   186.73} & 95270 & 95270 &    217.50 \\ 
		ym7663 & 7663 & 1054 & 1054 & 0.00 & 0 & \textbf{    16.26} & 1054 & 1054 &     67.60 \\ 
		pm8079 & 8079 & 517 & 517 & 0.00 & 0 & \textbf{     9.59} & 517 & 517 &     52.30 \\ 
		ei8246 & 8246 & 575 & 655 & 12.21 & 265 & TL & 579 & 579 & \textbf{  2268.70} \\ 
		ar9152 & 9152 & 3250 & 3284 & 1.04 & 9 & TL & 3250 & 3250 & \textbf{   397.70} \\ 
		ja9847 & 9847 & 1990 & 1990 & 0.00 & 0 & \textbf{     9.36} & 1990 & 1990 &     65.10 \\ 
		gr9882 & 9882 & 1038 & 1038 & 0.00 & 0 & \textbf{    12.58} & 1038 & 1038 &     74.40 \\ 
		kz9976 & 9976 & 3261 & 3261 & 0.00 & 0 & \textbf{   169.08} & 3261 & 3261 &    433.10 \\ 
		fi10639 & 10639 & 1456 & 1484 & 1.89 & 182 & TL & 1461 & 1461 & \textbf{  1030.10} \\ 
		rl11849 & 11849 & 2424 & 3129 & 22.53 & 10 & TL & 2462 & 2462 & \textbf{ 14808.20} \\ 
		usa13509 & 13509 & 51476 & 57780 & 10.91 & 6 & TL & 52178 & 52178 & \textbf{  4740.90} \\ 
		brd14051 & 14051 & 954 & 1058 & 9.83 & 248 & TL & 966 & 966 & \textbf{  1832.50} \\ 
		mo14185 & 14185 & 1110 & 1166 & 4.80 & 166 & TL & 1118 & 1118 & \textbf{   692.60} \\ 
		ho14473 & 14473 & 577 & 647 & 10.82 & 105 & TL & 585 & 585 & \textbf{   826.10} \\ 
		d15112 & 15112 & 3008 & 3239 & 7.13 & 0 & TL & 3037 & 3037 & \textbf{ 27897.80} \\ 
		it16862 & 16862 & 1237 & 1237 & 0.00 & 20 &    257.02 & 1237 & 1237 & \textbf{   142.10} \\ 
		d18512 & 18512 & 1066 & 1119 & 4.74 & 0 & TL & 1075 & 1075 & \textbf{ 12403.90} \\ 
		vm22775 & 22775 & 1099 & 1099 & 0.00 & 0 & \textbf{    33.74} & 1099 & 1099 &    265.60 \\ 
		sw24978 & 24978 & 1631 & 1775 & 8.11 & 19 & TL & 1637 & 1637 & \textbf{  1751.50} \\ 
		fyg28534 & 28534 & 139 & 168 & 17.26 & 0 & TL & \textbf{142} & \textbf{151} & TL2 \\ 
		bm33708 & 33708 & 1469 & 1512 & 2.84 & 1 & TL & 1475 & 1475 & \textbf{  3952.00} \\ 
		pla33810 & 33810 & 106684 & 137325 & 22.31 & 0 & TL & \textbf{110019} & \textbf{126933} & TL2 \\ 
		bby34656 & 34656 & 152 & 189 & 19.58 & 0 & TL & \textbf{155} & \textbf{164} & TL2 \\ 
		pba38478 & 38478 & 158 & 208 & 24.04 & 0 & TL & \textbf{160} & \textbf{163} & TL2 \\ 
		ch71009 & 71009 & 5564 & 6459 & 13.86 & 0 & TL & 5664 & 5665 & \textbf{ 13621.50} \\ 
		pla85900 & 85900 & 136667 & 185853 & 26.47 & 0 & TL & \textbf{142938} & \textbf{159484} & TL2 \\ 
		sra104815 & 104814 & 273 & 347 & 21.33 & 0 & TL & 277 & 277 & \textbf{ 15204.60} \\ 
		usa115475 & 115475 & 5090 & 6259 & 18.68 & 0 & TL & \textbf{5288} & \textbf{5419} & TL2 \\ 
		ara238025 & 238025 & 432 & 511 & 15.46 & 0 & TL & \textbf{441} & \textbf{458} & TL2 \\ 
		lra498378 & 498378 & 1647 & 2147 & 23.29 & 0 & TL & 1755 & 1755 & \textbf{ 44562.60} \\ 
		lrb744710 & 744710 & 586 & 694 & 15.56 & 0 & TL & \textbf{626} & \textbf{653} & TL2 \\ 
		\bottomrule
	\end{tabular}
	\endgroup
\end{table}
% latex table generated in R 3.4.4 by xtable 1.8-3 package
% Fri Jun 25 13:30:55 2021
\begin{table}[ht]
	\centering
	\caption{Detailed results for the \texttt{TSPlib} instances with 
	$p=20$\label{ta:tsplib20}} 
	\begingroup\footnotesize
	\begin{tabular}{lr|rrrrr|rrr}
		\toprule
		& & \multicolumn{5}{|c|}{$\fCLH$} & 
		\multicolumn{3}{|c}{\cite{contardo2019scalable}} \\ name & $|V|$  & LB 
		& UB & $g[\%]$  & \#BC & $t[s]$ & LB & UB & $t[s]$ \\ \midrule
		rw1621 & 1621 & 186 & 186 & 0.00 & 376 & \textbf{    21.39} & 186 & 186 &     34.80 \\ 
		u1817 & 1817 & 306 & 310 & 1.29 & 6425 & TL & 309 & 309 & \textbf{  2171.80} \\ 
		rl1889 & 1889 & 2060 & 2201 & 6.41 & 632 & TL & 2089 & 2089 & \textbf{   405.00} \\ 
		mu1979 & 1979 & 751 & 751 & 0.00 & 0 & \textbf{     1.15} & 751 & 751 &     26.20 \\ 
		pr2392 & 2392 & 1694 & 1864 & 9.12 & 383 & TL & 1736 & 1736 & \textbf{  7500.10} \\ 
		d15112-modif-2500 & 2500 & 2519 & 2965 & 15.04 & 227 & TL & 2573 & 2573 & \textbf{  1795.90} \\ 
		pcb3038 & 3038 & 480 & 618 & 22.33 & 124 & TL & 493 & 493 & \textbf{ 54802.10} \\ 
		nu3496 & 3496 & 519 & 519 & 0.00 & 358 &    454.56 & 519 & 519 & \textbf{   128.30} \\ 
		ca4663 & 4663 & 7024 & 7024 & 0.00 & 0 & \textbf{    11.04} & 7024 & 7024 &     57.00 \\ 
		rl5915 & 5915 & 2036 & 2509 & 18.85 & 11 & TL & 2083 & 2083 & \textbf{ 76684.10} \\ 
		rl5934 & 5934 & 2068 & 2412 & 14.26 & 68 & TL & 2100 & 2100 & \textbf{  4367.20} \\ 
		tz6117 & 6117 & 1266 & 1306 & 3.06 & 155 & TL & 1278 & 1278 & \textbf{  1988.70} \\ 
		eg7146 & 7146 & 972 & 972 & 0.00 & 0 & \textbf{     3.16} & 972 & 972 &     51.30 \\ 
		pla7397 & 7397 & 78014 & 83860 & 6.97 & 19 & TL & 78817 & 78817 & \textbf{   608.20} \\ 
		ym7663 & 7663 & 899 & 899 & 0.00 & 5 & \textbf{    63.54} & 899 & 899 &    152.10 \\ 
		pm8079 & 8079 & 430 & 430 & 0.00 & 5 & \textbf{    29.19} & 430 & 430 &     91.40 \\ 
		ei8246 & 8246 & 488 & 533 & 8.44 & 295 & TL & 497 & 497 & \textbf{ 33150.50} \\ 
		ar9152 & 9152 & 2653 & 2818 & 5.86 & 80 & TL & 2695 & 2695 & \textbf{  6067.80} \\ 
		ja9847 & 9847 & 1724 & 1724 & 0.00 & 0 & \textbf{     8.76} & 1724 & 1724 &     66.70 \\ 
		gr9882 & 9882 & 877 & 877 & 0.00 & 4 & \textbf{   106.71} & 877 & 877 &    235.60 \\ 
		kz9976 & 9976 & 2762 & 3168 & 12.82 & 45 & TL & 2790 & 2790 & \textbf{  1540.20} \\ 
		fi10639 & 10639 & 1228 & 1500 & 18.13 & 16 & TL & 1245 & 1245 & \textbf{ 15688.60} \\ 
		rl11849 & 11849 & 2080 & 2629 & 20.88 & 0 & TL & \textbf{2119} & \textbf{2273} & TL2 \\ 
		usa13509 & 13509 & 43592 & 56610 & 23.00 & 0 & TL & \textbf{44740} & \textbf{46719} & TL2 \\ 
		brd14051 & 14051 & 796 & 909 & 12.43 & 70 & TL & 803 & 803 & \textbf{ 16721.90} \\ 
		mo14185 & 14185 & 914 & 1022 & 10.57 & 23 & TL & 919 & 919 & \textbf{  3597.10} \\ 
		ho14473 & 14473 & 492 & 566 & 13.07 & 82 & TL & 497 & 497 & \textbf{  1470.20} \\ 
		d15112 & 15112 & 2539 & 3359 & 24.41 & 0 & TL & \textbf{2581} & \textbf{2717} & TL2 \\ 
		it16862 & 16862 & 1059 & 1059 & 0.00 & 0 & \textbf{   134.58} & 1059 & 1059 &    563.60 \\ 
		d18512 & 18512 & 902 & 1218 & 25.94 & 0 & TL & \textbf{912} & \textbf{969} & TL2 \\ 
		vm22775 & 22775 & 919 & 967 & 4.96 & 42 & TL & 932 & 932 & \textbf{  1491.90} \\ 
		sw24978 & 24978 & 1403 & 1662 & 15.58 & 0 & TL & 1421 & 1421 & \textbf{ 14835.00} \\ 
		fyg28534 & 28534 & 118 & 154 & 23.38 & 0 & TL & \textbf{128} & \textbf{138} & TL2 \\ 
		bm33708 & 33708 & 1245 & 1455 & 14.43 & 3 & TL & 1257 & 1257 & \textbf{ 17808.70} \\ 
		pla33810 & 33810 & 89123 & 118855 & 25.02 & 0 & TL & \textbf{91302} & \textbf{106076} & TL2 \\ 
		bby34656 & 34656 & \textbf{128} & 159 & 19.50 & 0 & TL & \textbf{128} & \textbf{138} & TL2 \\ 
		pba38478 & 38478 & \textbf{136} & 175 & 22.29 & 0 & TL & \textbf{136} & \textbf{148} & TL2 \\ 
		ch71009 & 71009 & 4637 & 5868 & 20.98 & 0 & TL & \textbf{4798} & \textbf{5250} & TL2 \\ 
		pla85900 & 85900 & 115361 & 163300 & 29.36 & 0 & TL & \textbf{119643} & \textbf{136984} & TL2 \\ 
		sra104815 & 104814 & 227 & 294 & 22.79 & 0 & TL & \textbf{232} & \textbf{236} & TL2 \\ 
		usa115475 & 115475 & 4198 & 5855 & 28.30 & 0 & TL & \textbf{4453} & \textbf{4747} & TL2 \\ 
		ara238025 & 238025 & 359 & 466 & 22.96 & 0 & TL & \textbf{372} & \textbf{407} & TL2 \\ 
		lra498378 & 498378 & 1342 & 1726 & 22.25 & 0 & TL & \textbf{1442} & \textbf{1573} & TL2 \\ 
		lrb744710 & 744710 & 474 & 697 & 31.99 & 0 & TL & \textbf{524} & \textbf{580} & TL2 \\ 
		\bottomrule
	\end{tabular}
	\endgroup
\end{table}
% latex table generated in R 3.4.4 by xtable 1.8-3 package
% Fri Jun 25 13:30:55 2021
\begin{table}[ht]
	\centering
	\caption{Detailed results for the \texttt{TSPlib} instances with 
	$p=25$\label{ta:tsplib25}} 
	\begingroup\footnotesize
	\begin{tabular}{lr|rrrrr|rrr}
		\toprule
		& & \multicolumn{5}{|c|}{$\fCLH$} & 
		\multicolumn{3}{|c}{\cite{contardo2019scalable}} \\ name & $|V|$  & LB 
		& UB & $g[\%]$  & \#BC & $t[s]$ & LB & UB & $t[s]$ \\ \midrule
		rw1621 & 1621 & 166 & 166 & 0.00 & 8 & \textbf{    12.83} & 166 & 166 &     38.70 \\ 
		u1817 & 1817 & 265 & 275 & 3.64 & 1766 & TL & 272 & 272 & \textbf{  1875.80} \\ 
		rl1889 & 1889 & 1833 & 2063 & 11.15 & 294 & TL & 1866 & 1866 & \textbf{  1154.80} \\ 
		mu1979 & 1979 & 639 & 639 & 0.00 & 0 & \textbf{     2.11} & 639 & 639 &     30.40 \\ 
		pr2392 & 2392 & 1492 & 1956 & 23.72 & 135 & TL & 1520 & 1520 & \textbf{ 23767.70} \\ 
		d15112-modif-2500 & 2500 & 2205 & 2581 & 14.57 & 225 & TL & 2243 & 2243 & \textbf{  5266.20} \\ 
		pcb3038 & 3038 & 425 & 545 & 22.02 & 104 & TL & \textbf{433} & \textbf{470} & TL2 \\ 
		nu3496 & 3496 & 441 & 441 & 0.00 & 11 & \textbf{    41.79} & 441 & 441 &    116.80 \\ 
		ca4663 & 4663 & 5966 & 5966 & 0.00 & 13 &    121.12 & 5966 & 5966 & \textbf{   118.20} \\ 
		rl5915 & 5915 & 1786 & 2201 & 18.86 & 20 & TL & \textbf{1823} & \textbf{1916} & TL2 \\ 
		rl5934 & 5934 & 1825 & 2113 & 13.63 & 19 & TL & 1850 & 1850 & \textbf{ 25300.40} \\ 
		tz6117 & 6117 & 1121 & 1426 & 21.39 & 237 & TL & \textbf{1152} & \textbf{1258} & TL2 \\ 
		eg7146 & 7146 & 855 & 855 & 0.00 & 0 & \textbf{     9.76} & 855 & 855 &     79.10 \\ 
		pla7397 & 7397 & 68964 & 78142 & 11.75 & 35 & TL & 69508 & 69508 & \textbf{   776.80} \\ 
		ym7663 & 7663 & 785 & 819 & 4.15 & 20 & TL & 785 & 785 & \textbf{   189.40} \\ 
		pm8079 & 8079 & 367 & 367 & 0.00 & 0 & \textbf{    50.52} & 367 & 367 &    184.80 \\ 
		ei8246 & 8246 & 421 & 532 & 20.86 & 71 & TL & \textbf{429} & \textbf{461} & TL2 \\ 
		ar9152 & 9152 & 2302 & 2717 & 15.27 & 50 & TL & 2355 & 2355 & \textbf{ 30531.70} \\ 
		ja9847 & 9847 & 1362 & 1362 & 0.00 & 0 & \textbf{    43.02} & 1362 & 1362 &    154.80 \\ 
		gr9882 & 9882 & 772 & 772 & 0.00 & 59 & \textbf{   236.54} & 772 & 772 &    532.30 \\ 
		kz9976 & 9976 & 2439 & 2800 & 12.89 & 72 & TL & 2479 & 2479 & \textbf{  4666.40} \\ 
		fi10639 & 10639 & 1075 & 1400 & 23.21 & 10 & TL & \textbf{1103} & \textbf{1173} & TL2 \\ 
		rl11849 & 11849 & 1823 & 2506 & 27.25 & 0 & TL & \textbf{1838} & \textbf{2099} & TL2 \\ 
		usa13509 & 13509 & 37471 & 46954 & 20.20 & 0 & TL & \textbf{38150} & \textbf{40578} & TL2 \\ 
		brd14051 & 14051 & 693 & 843 & 17.79 & 7 & TL & \textbf{703} & \textbf{737} & TL2 \\ 
		mo14185 & 14185 & 806 & 907 & 11.14 & 11 & TL & 818 & 818 & \textbf{ 42123.30} \\ 
		ho14473 & 14473 & 433 & 547 & 20.84 & 99 & TL & 436 & 436 & \textbf{ 11941.10} \\ 
		d15112 & 15112 & 2201 & 2877 & 23.50 & 0 & TL & \textbf{2233} & \textbf{2447} & TL2 \\ 
		it16862 & 16862 & 939 & 1140 & 17.63 & 6 & TL & 951 & 951 & \textbf{  5284.40} \\ 
		d18512 & 18512 & 792 & 1029 & 23.03 & 0 & TL & \textbf{795} & \textbf{881} & TL2 \\ 
		vm22775 & 22775 & 798 & 815 & 2.09 & 1 & TL & 799 & 799 & \textbf{  3102.50} \\ 
		sw24978 & 24978 & 1209 & 1567 & 22.85 & 0 & TL & \textbf{1233} & \textbf{1285} & TL2 \\ 
		fyg28534 & 28534 & \textbf{103} & 133 & 22.56 & 0 & TL & 102 & \textbf{112} & TL2 \\ 
		bm33708 & 33708 & 1069 & 1358 & 21.28 & 0 & TL & \textbf{1106} & \textbf{1140} & TL2 \\ 
		pla33810 & 33810 & 80661 & 108074 & 25.37 & 0 & TL & \textbf{81800} & \textbf{88861} & TL2 \\ 
		bby34656 & 34656 & \textbf{113} & 144 & 21.53 & 0 & TL & \textbf{113} & \textbf{128} & TL2 \\ 
		pba38478 & 38478 & \textbf{118} & 160 & 26.25 & 0 & TL & \textbf{118} & \textbf{134} & TL2 \\ 
		ch71009 & 71009 & 4013 & 5491 & 26.92 & 0 & TL & \textbf{4145} & \textbf{4321} & TL2 \\ 
		pla85900 & 85900 & 103546 & 145695 & 28.93 & 0 & TL & \textbf{107527} & \textbf{124693} & TL2 \\ 
		sra104815 & 104814 & 203 & 241 & 15.77 & 0 & TL & \textbf{206} & \textbf{216} & TL2 \\ 
		usa115475 & 115475 & 3649 & 5347 & 31.76 & 0 & TL & \textbf{3808} & \textbf{4453} & TL2 \\ 
		ara238025 & 238025 & 307 & 393 & 21.88 & 0 & TL & \textbf{319} & \textbf{357} & TL2 \\ 
		lra498378 & 498378 & 1109 & 1608 & 31.03 & 0 & TL & \textbf{1228} & \textbf{1339} & TL2 \\ 
		lrb744710 & 744710 & 415 & 636 & 34.75 & 0 & TL & \textbf{454} & \textbf{533} & TL2 \\ 
		\bottomrule
	\end{tabular}
	\endgroup
\end{table}
% latex table generated in R 3.4.4 by xtable 1.8-3 package
% Fri Jun 25 13:30:55 2021
\begin{table}[ht]
	\centering
	\caption{Detailed results for the \texttt{TSPlib} instances with 
	$p=30$\label{ta:tsplib30}} 
	\begingroup\footnotesize
	\begin{tabular}{lr|rrrrr|rrr}
		\toprule
		& & \multicolumn{5}{|c|}{$\fCLH$} & 
		\multicolumn{3}{|c}{\cite{contardo2019scalable}} \\ name & $|V|$  & LB 
		& UB & $g[\%]$  & \#BC & $t[s]$ & LB & UB & $t[s]$ \\ \midrule
		rw1621 & 1621 & 147 & 147 & 0.00 & 32 & \textbf{    14.84} & 147 & 147 &     64.80 \\ 
		u1817 & 1817 & 240 & 251 & 4.38 & 2438 & TL & 241 & 241 & \textbf{  4228.30} \\ 
		rl1889 & 1889 & 1638 & 1885 & 13.10 & 226 & TL & 1657 & 1657 & \textbf{  1603.40} \\ 
		mu1979 & 1979 & 552 & 552 & 0.00 & 0 & \textbf{     4.21} & 552 & 552 &     34.40 \\ 
		pr2392 & 2392 & 1351 & 1765 & 23.46 & 362 & TL & \textbf{1379} & \textbf{1471} & TL2 \\ 
		d15112-modif-2500 & 2500 & 1980 & 2604 & 23.96 & 414 & TL & 2029 & 2029 & \textbf{ 50627.40} \\ 
		pcb3038 & 3038 & 382 & 508 & 24.80 & 235 & TL & \textbf{386} & \textbf{412} & TL2 \\ 
		nu3496 & 3496 & 396 & 396 & 0.00 & 43 & \textbf{   135.53} & 396 & 396 &    205.70 \\ 
		ca4663 & 4663 & 5361 & 5361 & 0.00 & 0 & \textbf{    85.75} & 5361 & 5361 &    152.60 \\ 
		rl5915 & 5915 & 1618 & 2210 & 26.79 & 10 & TL & \textbf{1624} & \textbf{1853} & TL2 \\ 
		rl5934 & 5934 & 1631 & 2116 & 22.92 & 60 & TL & \textbf{1658} & \textbf{1812} & TL2 \\ 
		tz6117 & 6117 & 1001 & 1304 & 23.24 & 224 & TL & \textbf{1025} & \textbf{1142} & TL2 \\ 
		eg7146 & 7146 & 744 & 744 & 0.00 & 0 & \textbf{     8.53} & 744 & 744 &     76.80 \\ 
		pla7397 & 7397 & 62057 & 71915 & 13.71 & 0 & TL & 63770 & 63770 & \textbf{ 20045.10} \\ 
		ym7663 & 7663 & 703 & 748 & 6.02 & 0 & TL & 715 & 715 & \textbf{   536.60} \\ 
		pm8079 & 8079 & 330 & 330 & 0.00 & 3 & \textbf{   119.69} & 330 & 330 &    459.80 \\ 
		ei8246 & 8246 & 384 & 508 & 24.41 & 42 & TL & \textbf{386} & \textbf{412} & TL2 \\ 
		ar9152 & 9152 & 2050 & 2648 & 22.58 & 14 & TL & 2080 & 2080 & \textbf{ 10145.70} \\ 
		ja9847 & 9847 & 1220 & 1220 & 0.00 & 0 & \textbf{    16.65} & 1220 & 1220 &    268.30 \\ 
		gr9882 & 9882 & 677 & 677 & 0.00 & 9 & \textbf{   334.72} & 677 & 677 &   1024.60 \\ 
		kz9976 & 9976 & 2191 & 2820 & 22.30 & 44 & TL & 2230 & 2230 & \textbf{ 11820.20} \\ 
		fi10639 & 10639 & 967 & 1251 & 22.70 & 10 & TL & \textbf{974} & \textbf{1017} & TL2 \\ 
		rl11849 & 11849 & \textbf{1649} & 2255 & 26.87 & 0 & TL & 1641 & \textbf{1855} & TL2 \\ 
		usa13509 & 13509 & 34137 & 45591 & 25.12 & 0 & TL & \textbf{34771} & \textbf{37036} & TL2 \\ 
		brd14051 & 14051 & 618 & 812 & 23.89 & 4 & TL & \textbf{620} & \textbf{668} & TL2 \\ 
		mo14185 & 14185 & 719 & 936 & 23.18 & 0 & TL & \textbf{732} & \textbf{767} & TL2 \\ 
		ho14473 & 14473 & 389 & 511 & 23.87 & 30 & TL & 396 & 396 & \textbf{ 29545.80} \\ 
		d15112 & 15112 & 1994 & 2662 & 25.09 & 0 & TL & \textbf{2009} & \textbf{2254} & TL2 \\ 
		it16862 & 16862 & 849 & 947 & 10.35 & 57 & TL & 854 & 854 & \textbf{ 16851.80} \\ 
		d18512 & 18512 & \textbf{712} & 963 & 26.06 & 0 & TL & \textbf{712} & \textbf{786} & TL2 \\ 
		vm22775 & 22775 & 690 & 802 & 13.97 & 7 & TL & 696 & 696 & \textbf{ 13315.20} \\ 
		sw24978 & 24978 & 1083 & 1514 & 28.47 & 0 & TL & \textbf{1096} & \textbf{1215} & TL2 \\ 
		fyg28534 & 28534 & \textbf{93} & 122 & 23.77 & 0 & TL & 92 & \textbf{108} & TL2 \\ 
		bm33708 & 33708 & 930 & 1221 & 23.83 & 0 & TL & \textbf{968} & \textbf{1065} & TL2 \\ 
		pla33810 & 33810 & 70365 & 95203 & 26.09 & 0 & TL & \textbf{71480} & \textbf{83187} & TL2 \\ 
		bby34656 & 34656 & 101 & 137 & 26.28 & 0 & TL & \textbf{102} & \textbf{133} & TL2 \\ 
		pba38478 & 38478 & \textbf{107} & 149 & 28.19 & 0 & TL & \textbf{107} & \textbf{125} & TL2 \\ 
		ch71009 & 71009 & 3581 & 4947 & 27.61 & 0 & TL & \textbf{3724} & \textbf{4098} & TL2 \\ 
		pla85900 & 85900 & 90212 & 127681 & 29.35 & 0 & TL & \textbf{93871} & \textbf{114244} & TL2 \\ 
		sra104815 & 104814 & 182 & 216 & 15.74 & 0 & TL & \textbf{186} & \textbf{204} & TL2 \\ 
		usa115475 & 115475 & 3258 & 4770 & 31.70 & 0 & TL & \textbf{3391} & \textbf{3874} & TL2 \\ 
		ara238025 & 238025 & 277 & 386 & 28.24 & 0 & TL & \textbf{288} & \textbf{368} & TL2 \\ 
		lra498378 & 498378 & 1015 & 1486 & 31.70 & 0 & TL & \textbf{1073} & \textbf{1192} & TL2 \\ 
		lrb744710 & 744710 & 370 & 580 & 36.21 & 0 & TL & \textbf{403} & \textbf{475} & TL2 \\ 
		\bottomrule
	\end{tabular}
	\endgroup
\end{table}

\end{document}